\pdfoutput=1
\documentclass[a4paper,11pt]{article}	
\input{mystyle.sty}		

\title{Hitchin grafting representations II: Dynamics}
\date{\today}
\author{Pierre-Louis Blayac, Ursula Hamenstädt, Théo Marty, Andrea Egidio Monti}

\begin{document}

	\maketitle

\begin{abstract}
The Hitchin component of the character variety of representations of a surface group $\pi_1(S)$ into $\mathrm{PSL}_d(\mathbb R)$ for some $d\geq 3$ can be equipped with a pressure metric whose restriction to the Fuchsian locus equals the Weil-Petersson metric up to a constant factor. We show that if the genus of $S$ is at least $3$, then the Fuchsian locus contains quasi-convex subsets of infinite diameter for the Weil--Petersson metric whose diameter for the path metric of the pressure metric is finite. This is established through showing that biinfinite paths of bending deformations have controlled bounded length. 
     \end{abstract}

 { 
    \hypersetup{linkcolor=black}
    \renewcommand{\baselinestretch}{0.85}
    \normalsize
\setcounter{tocdepth}{2}
    \tableofcontents
    \renewcommand{\baselinestretch}{1.0}
    \normalsize
 }

\section*{Introduction}
\addcontentsline{toc}{section}{Introduction}

The \emph{Teichm\"uller space} ${\mathcal T}(S)$ of a closed oriented surface $S$ 
of genus $g\geq 2$ is the space of \emph{marked} hyperbolic structures on $S$. Equivalently, it can be described as a distinguished component of the space of conjugacy classes of homomorphisms $\pi_1(S)\to \PSL_2(\mathbb{R})$, with target the group $\PSL_2(\mathbb{R})$ of orientation preserving isometries of the hyperbolic plane. It was discovered by Hitchin that 
an analog of the Teichm\"uller space also exists for conjugacy classes of 
representations of $\pi_1(S)$ into simple split real Lie groups of higher rank.

The so-called \emph{Hitchin component} ${\rm Hit}(S)$ for the target group $\PSL_d(\mathbb{R})$ $(d\geq 3)$
is the component of the \emph{character variety} containing conjugacy classes of discrete representations 
which factor through an irreducible embedding $\PSL_2(\mathbb{R})\to \PSL_d(\mathbb{R})$. 
Hitchin~\cite{Hitchin} showed that the Hitchin component is homeomorphic to $\R^m$ for some explicit 
$m>0$, and later Labourie~\cite{Lab06} and Fock--Goncharov~\cite{FG06} independently proved that all representations in the Hitchin component are faithful with discrete image.

In~\cite{BCLS15}, a ${\rm Mod}(S)$-invariant metric on ${\rm Hit}(S)$ was introduced, the so-called \emph{pressure metric} (see also~\cite{BCLS18,BCS18}). 
Fix a positive linear functional $\alpha_0$ on 
the convex cone of vectors $x=(x_1,\dots,x_d)$ with $x_1\geq \dots\geq x_d$ and $x_1+\dots+x_d=0$, for instance
\begin{equation}\label{eq:ex of alpha0}\alpha_0(x)=(d-1)x_1+(d-3)x_2+\dots + (1-d)x_d.\end{equation}
The starting point is that $\alpha_0$ determines a conjugacy invariant translation length function on $\PSL_d(\R)$ by applying this functional to the logarithm of the absolute values of the eigenvalues.
Hence this yields a length function for the images of $\pi_1(S)$ under a representation in the Hitchin component.

Now let us assume that, after choosing a hyperbolic metric on $S$, this function can be represented by integration of a H\"older continuous positive function over periodic orbits for the \emph{geodesic flow} $\Phi^t$ on the unit tangent bundle $T^1S$ of $S$.
Then one can associate to such a representation the equilibrium state of a positive multiple of the function, chosen to be of vanishing pressure.
Using that the pressure is a strictly convex functional, this construction gives rise to a positive semi-definite 
symmetric bilinear form on the tangent space of the Hitchin component (see \eqref{eq:def pressure metric}).
Although this symmetric bilinear form may not be positive definite everywhere, 
it nevertheless defines a length metric on the component~\cite{Sam24}. 

Using results of~\cite{Sam14,BCLS15}, we verify in Section~\ref{sec:currents} that this assumption of 
representability by Hölder potentials always holds, so that we can talk about the induced pressure metric on ${\rm Hit}(S)$. 
Related statements are contained in~\cite{BPS19}. 
Note that while the pressure metric is defined for any Hitchin component, 
there are several other interesting  
constructions of Riemannian metrics on the Hitchin component for $\PSL_3(\R)$, 
see~\cite{DG96,Li2016,KZ17}. It seems unknown  
whether these metrics coincide with one of the pressure metrics.

As there are many natural choices for the positive linear functional $\alpha_0$ on the Weyl cone, 
there are many natural length functions on the Hitchin component, and hence many natural pressure metrics. 
The linear functional $\alpha_0$ we shall use for the pressure metric 
determines a $\PSL_d(\mathbb{R})$-invariant Finsler metric $\mf F$ on the symmetric space $\X=\PSL_d(\mathbb{R})/\PSO_d(\R)$, see e.g.\ Section 5.1 of~\cite{KL18}, called a \emph{nice} Finsler metric in the sequel.
We refer to Section \ref{sec:lietheoryI} for the precise constraints on the Finsler metrics we shall use,
the metric in Equation~\ref{eq:ex of alpha0} is an example.
The metric $\mf F$ depends on $\alpha_0$, but its geodesics  do not by Lemma 5.10 of~\cite{KL18}.

The  restriction to the \emph{Fuchsian locus} of each of these metrics is a multiple of the \emph{Weil--Petersson metric} on 
the Teichm\"uller space~\cite{Bo88} (relying in an essential way on the article~\cite{wolpert},  see also Corollary 1.6 of~\cite{BCLS15} and~\cite{McM08} for an explicit statement). 
Infinitesimal properties of the pressure metric at the Fuchsian locus were studied in~\cite{LW18}, and also in~\cite{Da19} in the case $d=3$. 
The large scale geometry 
of the Weil--Petersson metric on Teichm\"uller space is quite well understood.
In particular, this metric is incomplete~\cite{wolpert}, and 
its metric completion $\bar{\mc T}(S)$, sometimes called \emph{augmented Teichm\"uller space}, 
is a ${\rm CAT}(0)$ stratified space whose strata are Teichm\"uller spaces of marked  finite area hyperbolic
metrics on surfaces obtained from $S$ by replacing some essential simple closed curves by nodes. 

For $d=3$, Loftin~\cite{Loftin04} defined an augmented Hitchin space (see also~\cite{LZ21}). 
However, it turns out that unlike in the case of Teichm\"uller space, this construction is 
not well related to geometric properties of the pressure metric. 
A first instance of this possibility 
arose in the study 
of large scale properties of pressure metrics for other kinds of moduli spaces: for instance, for hyperbolic structures with boundary~\cite{Xu19}, for  marked metric graphs~\cite{ACR23}, and for quasi-Fuchsian representations~\cite{FHJZ24}.
The following result gives an illustration of this fact in the context of Hitchin representations. 

\begin{maintheorem}\label{loftin}
    Suppose $S$ has genus at least $3$ and consider a pressure metric coming from a nice Finsler 
    metric on the symmetric space $\PSL(3,\R)/\PSO(3)$.
    Let $\Hit_3^{\rm aug}(S)$ be Loftin's augmented Hitchin space and $\overline\Hit_3(S)$ the pressure metric completion.

    Then there exist paths $(x_t)_{t\geq0},(y_t)_{t\geq0}\subset \Hit_3(S)$ converging to two distinct points $x,y\in \Hit_3^{\rm aug}(S)$, but converging to the same point of $\overline\Hit_3(S)$.
\end{maintheorem}

Theorem~\ref{loftin} is obtained as an application of 
a large scale geometric study of the path metric for all Hitchin components which we discuss next. 
%

Consider an essential subsurface $S_0$ of $S$ 
whose connected boundary $\partial S_0$ is an essential simple closed curve.
Choose a hyperbolic metric $X$ on $S$ and a marked point on the geodesic representing $\partial S_0$. 
Let $\ell>0$ be the length of $\partial S_0$ for the metric $X$, and let 
${\mathcal T}(S_0,\ell)$ be the Teichm\"uller space  of marked hyperbolic metrics on $S_0$ with geodesic boundary of fixed length $\ell>0$ and one marked point on $\partial S_0$. 
The choice of $X$ 
determines an embedding ${\mathcal T}(S_0,\ell)\to {\mathcal T}(S)$ 
by associating to a point  $X_0\in {\mathcal T}(S_0,\ell)$ the metric on $S$ obtained by gluing $X_0$ to $X\vert_{S-S_0}$ identifying marked points. 

It is known that this embedding is quasi-isometric for the Weil--Petersson metric on 
${\mathcal T}(S_0,\ell)$ and ${\mathcal T}(S)$. 
Although there does not seem to be an explicit statement available in the literature, this 
is a fairly easy consequence of the fact that augmented Teichm\"uller space 
contains the product of
the Teichm\"uller spaces ${\mathcal T}(S_0)$, ${\mathcal T}(S-S_0)$
of the surfaces $S_0,S-S_0$, with the boundary replaced by cusps, as convex 
subspaces.
Each factor in this product is of infinite diameter, and  
the image of the embedding of ${\mathcal T}(S_0,\ell)$ contains ${\mathcal T}(S_0)\times \{pt\}$ 
in a uniformly bounded neighborhood
\cite{Ya03}. 
%
%
In particular, the image of the embedding ${\mathcal T}(S_0,\ell)\to {\mathcal T}(S)$ 
has infinite diameter for the 
Weil--Petersson metric on~${\mathcal T}(S)$.

The mapping class group ${\rm Mod}(S)$ acts on the Hitchin component by precomposition of marking. 
As this action is isometric for the pressure metric, it  extends to an action on the metric completion of ${\rm Hit}(S)$.
For any essential subsurface $S_0$ of $S$, the mapping class group ${\rm Mod}(S_0)$ of $S_0$ is a subgroup of ${\rm Mod}(S)$.
The following theorem is our first main result.

\begin{maintheorem}\label{thm:shortcut} Let $g\geq 3$ and let $S_0\subset S$ be any essential connected subsurface of genus $g_0\leq  g-2$ with connected boundary.
\begin{enumerate} 
\item  The image of an embedding ${\mathcal T}(S_0,\ell)\to {\mathcal T}(S)$ has finite diameter for the pressure metric. 
\item 
The action of the subgroup ${\rm Mod}(S_0)$ of ${\rm Mod}(S)$ on the metric completion of ${\rm Hit}(S)$ with respect to the pressure metric has a global fixed point.
\end{enumerate}
\end{maintheorem}

As in the case of Teichm\"uller space, metric completions and partial 
compactifications of ${\rm Hit}(S)$ can be studied through embeddings 
into the space $\mc{PC}(S)$ of \emph{projective 
geodesic currents} on $S$. 
Namely,  the length function $f_\rho$ on $T^1S$ defined by a 
representation $\rho\in {\rm Hit}(S)$ determines 
the projective geodesic current $\Theta(\rho)$
defined by the 
equilibrium state of a multiple of $f_\rho$.
In this way one obtains 
a continuous mapping $\Theta$ of ${\rm Hit}(S)$ 
into the space of projective geodesic currents on $S$. 
Its restriction to the Fuchsian locus coincides with the 
standard embedding of Teichm\"uller space~\cite{Bo88} which 
associates to a hyperbolic metric its 
projective Liouville current.

Since $S$ is compact, 
the space of projective geodesic currents on  $S$, equipped with
the weak$^*$-topology, is a compact space.
To establish Theorems~\ref{loftin} and~\ref{thm:shortcut}, we shall make use the following result, 
interesting in its own right, on the behavior of $\Theta$ along the degenerating sequences of Hitchin 
representations obtained by bending a fixed Fuchsian representation.

Note that   
by Corollary 1.4 of~\cite{PS17}, 
for each nice Finsler metric on $\X$
the entropy of a Hitchin representation is 
defined and is maximized only on the Fuchsian locus.

\begin{maintheorem}\label{main1}
Let $S_0\subset S$ be a proper connected essential subsurface such that no component of $S_1=S-S_0$ is a pair of pants, 
and let $h$ be any hyperbolic metric on $S_0$ so that the boundary
of $S_0$ is geodesic. 
Then there exists a sequence $\rho_i$ of Hitchin representations with the
following properties.
\begin{enumerate}
\item The projective currents $\Theta(\rho_i)$ converge weakly to 
the projective current of maximal entropy for the geodesic flow 
on $(S_0,h)$.
\item The entropies of the representations $\rho_i$ converge to the 
entropy of the geodesic flow on $(S_0,h)$.
\end{enumerate}
\end{maintheorem}


There exists a natural embedding of the Teichm\"uller space 
${\mathcal T}(S_0,\ell)$ into the space of geodesic currents for $S_0$, and 
the pressure metric on this space of currents is defined. Theorem~\ref{main1} 
implies that the metric 
completion of the Hitchin component contains a subspace which is naturally isometric to
${\mathcal T}(S_0,\ell)$ 
equipped with the pressure metric (which does not coincide with the Weil--Petersson metric, see~\cite{Xu19}).
It then also contains a subspace which is isometric to the space of marked metric 
graphs equipped with the Weil--Petersson metric~\cite{Xu19}, as this space is contained in the metric
completion of ${\mc T}(S_0,\ell)$ equipped with the pressure metric.

By work of Bonahon (Corollary 16 of~\cite{Bo88}), 
the restriction of the map $\Theta$ to ${\mathcal T}(S)$ is an embedding into the space of 
projective geodesic currents $\mc{PC}(S)$, and the boundary of the resulting compactification $\overline{\Theta(\mc{T}(S))}-\Theta(\mc T(S))$ is precisely the space $\mc{PML}(S)$ of projective measured geodesic laminations, that is, currents with 
vanishing self-intersection.
Theorem~\ref{main1} implies that $\overline{\Theta({\rm Hit}(S))}-\Theta({\rm Hit}(S))$ 
is bigger than $\mc{PML}(S)$, since the projective current of maximal entropy for $S_0$ is not a measured geodesic lamination.
Section 1.3 of~\cite{BIPP21} contains related results.
That the map $\Theta$ is an embedding for some choices of length functions, different from ours, 
is due to 
Bridgeman, Canary, Labourie and Sambarino (Theorem 1.2 of~\cite{BCLS18}).

\begin{question}
For $n\geq 3$, is $\Theta({\rm Hit}(S))$ dense in the space of 
projective geodesic currents?
\end{question}

As the map which associates to a 
H\"older continuous positive length function $f$ on $T^1S$ 
the entropy of the normalized Gibbs current of $f$ is continuous we obtain
the following 

\begin{maincorollary}\label{maincorollary}
For any number $a\in [0,1)$ there exists a sequence of degenerating 
Hitchin representations whose
entropy converges to $a$.
\end{maincorollary}

The case $a=0$ is due to Zhang~\cite{Zh15} and was reworked in~\cite{SWZ20},
using mainly algebraic methods. 
Our proof is entirely geometric. 
For $d=3$ and in the context of real projective structures on surfaces, 
Corollary~\ref{maincorollary} is independently due to Nie~\cite{Nie15}.
In this context, the article~\cite{FK16} also contains related results, embarking from 
the same deformations we use, but with a different geometric interpretation.

Theorems~\ref{loftin},~\ref{thm:shortcut} and~\ref{main1} rest on a geometric
understanding of specific paths in ${\rm Hit}(S)$  which 
was established in \cite{BHM25}.
These paths are so-called \emph{grafting deformations}, also called 
\emph{bending deformations} or \emph{bulging deformations}.
They are defined as follows.

For $d\geq 2$, the (unique up to conjugation) 
$d$-dimensional irreducible representation of $\PSL_2(\mathbb{R})$ defines an 
embedding $\PSL_2(\mathbb{R})\to \PSL_d(\mathbb{R})$ whose image stabilizes
a totally geodesic subspace $\bH^2\subset \X=\PSL_d(\mathbb{R})/\PSO(d)$
where $\X$ is equipped with the symmetric metric. 
Up to scaling, $\bH^2$ is isometric to the hyperbolic plane. Its tangent bundle 
$T\bH^2$ consists of regular tangent vectors in $T\X$.
In particular, for $d\geq 3$, every
geodesic line $\gamma\subset \bH^2$ is contained in a unique maximal flat $F$ of 
dimension $d-1$.
This flat intersects $\bH^2$ orthogonally along $\gamma$.

Let now $\Gamma$ be the fundamental group of a closed oriented surface $S$ of genus at least $2$.
Let $\gamma\in\Gamma$ be defined by a separating simple closed curve on $S$.
This curve defines a one edge graph of groups decomposition $\Gamma=\Gamma_1*_C\Gamma_2$ where $C$ is the infinite cyclic group generated by $\gamma$.
Let $\rho$ be a discrete and faithful representation of $\Gamma$ into $\PSL_2(\mathbb{R})\subset \PSL_d(\mathbb{R})$.
Let $\alpha \in \PSL_d(\mathbb{R})$ be an element in the centralizer of $\rho(C)$ but not contained in 
the one-parameter subgroup containing $\rho(C)$.
Partial conjugation of $\rho$ by $\alpha$ then defines a new representation, obtained from the Fuchsian representation $\rho$ by 
\emph{Hitchin grafting at $\gamma$ with $\alpha$}.
More precisely, this new representation coincides with $\rho$ on $\Gamma_1$, but maps any $\beta\in \Gamma_2$ to $\alpha \rho(\beta) \alpha^{-1}$.
More generally, if $t\to \alpha(t)$ is a one-parameter subgroup of the centralizer of $\rho(C)$ 
not containing $\rho(C)$ 
then we obtain in this fashion a path in ${\rm Hit}(S)$ which we call a \emph{Hitchin grafting path}. 
Such paths are also well defined if $\gamma$ is non-separating and determines a decomposition of 
$\pi_1(S)$ as an HNN-extension. 
We show

\begin{maintheorem}\label{main4} 
    Hitchin grafting paths have finite length for the pressure metric. 
\end{maintheorem}

Theorem~\ref{pressurelength} contains a more precise version of this result. Note that 
the grafting paths we consider correspond in Teichm\"uller space 
to shearing (or twisting) paths along a simple 
geodesic. These paths are contained in the thick part of Teichm\"uller space and have
infinite Weil-Petersson length, but any two points on the path can be connected by
a Weil-Petersson geodesic whose length is bounded from above by a constant only depending on an 
upper bound for the length of $\gamma$.

Hitchin grafting paths 
are also well defined if the grafting is performed at a simple geodesic multicurve with more 
than one component and if the starting representation is not contained in the Fuchsian locus. 
It seems likely that our argument can be extended to show finite pressure length for such paths
as well, however we do not carry out such an extension. In view of 
the work~\cite{BD17}, it may be possible to extend this analysis to an even larger class of 
naturally defined paths in ${\rm Hit}(S)$.
This raises the following 

\begin{question}
Is the diameter of ${\rm Hit}(S)$ with respect to the pressure metric finite?
\end{question}

The answer to this question is yes in a different context, namely 
for the pressure metric on quasi-Fuchsian space
\cite{FHJZ24}. 
Note that any two points in ${\rm Hit}(S)$ can be connected by finitely many grafting paths
\cite{AZ23}, however not starting from points in the Fuchsian locus, and  
it is unclear whether the number of such paths needed has a uniform upper bound.

The proof of Theorem~\ref{main4} rests on the main results of the companion article \cite{BHM25} which 
gives a geometric interpretation of the concept of positivity of Hitchin representations 
due to Fock and Goncharov \cite{FG06}. 

\paragraph{\bf Organization of the article and structure of the proof}

The first three sections are introductory and mainly used to collect results from the literature,
especially from the first part \cite{BHM25} of this work. 
The results in  Section~\ref{sec:currents} 
    can mostly be found in the literature, 
    although not always in the form we need.  
    We establish that the nice Finsler metrics defined in Section~\ref{sec:lietheoryI} indeed define a 
    pressure metric for the Hitchin component. 

Section \ref{sec:intersection} contains a first instance on the interplay between geometry and
dynamics of Hitchin representations. 
 We show that there are sequences of representations in the Hitchin component
    whose normalized intersection with any Fuchsian representation tend to infinity. Here the 
    normalized intersection number is the entropy normalized intersection number in the space of currents.

    The remaining part of the article is devoted to the study of the pressure metric. 
    In Section~\ref{sec:ehresmann} we use the geometric interpretation of positivity established in \cite{BHM25}
    to give precise norm bounds for first and second derivatives of the 
    Finsler length of a conjugacy class in $\pi_1(S)$ restricted to two specific classes of paths in ${\rm Hit}(S)$.

    Sections~\ref{sec:convergence-of-currents} and Sections~\ref{sec:pressurelength} contain the main dynamical
    results of this article. We use the geometric information on Hitchin grafting representations obtained in 
    \cite{BHM25} and the results of Section \ref{sec:ehresmann} to analyze the geodesic currents defined by such representations. This leads to 
    the proof of Theorem~\ref{main1} and Theorem~\ref{main4}. The proofs of Theorem \ref{loftin} and 
    Theorem~\ref{thm:shortcut} is contained in 
    Section~\ref{sec:distortion}.
    
The appendix collects information on the entropy of the geodesic flow on compact hyperbolic surfaces 
with boundary which we were unable to find in the literature in the form we need.

{\bf Acknowledgement:} This project started as a working seminar in fall 2021, during the pandemic,
held in person at the Max Planck Institute for Mathematics in Bonn.
We thank the MPI for the hospitality and financial support, and we thank
Gianluca Faraco, Elia Fioravanti, Frieder J\"ackel, Yannick Krifka, Laura Monk and Yongquan Zhang for 
many enjoyable discussions and good company.
P.-L.B. is grateful to Dick Canary, Fanny Kassel, Charles Reid, and Ralph Spatzier for helpful discussions, and 
U.H. and P.-L.B. thank Andr\'es Sambarino for 
helpful discussions.
All authors thank Beatrice Pozzetti for pointing out an error in an earlier version of this article.

\section{Lie groups and symmetric spaces}\label{sec:lietheoryI}

This section collects some basic facts on Lie groups and symmetric spaces and introduces conventions and notations used lated on.

Consider the unique (up to conjugacy) irreducible representation $\tau:\PSL_2(\R)\to G=\PSL_d(\R)$, which can be described as follows.
A matrix $M=\smallpmatrix{a}{b}{c}{e}\in\SL_2(\R)$ acts on the algebra $\R[X,Y]$ of polynomials in two variables by $M\cdot X=aX+cY$ and $M\cdot Y=bX+eY$. 
This action preserves the $d$-dimensional linear subspace $\R^h_{d-1}[X,Y]$ of degree $d-1$ homogeneous polynomials, which we identify with $\R^d$.

This representation is regular, in the sense that it maps diagonalisable  2-by-2 matrices with distinct real eigenvalues to diagonalisable $d$-by-$d$ matrices with distinct real eigenvalues.
In fact, using a suitable basis of $\R^h_{d-1}[X,Y]$, the representation $\tau$ maps
\begin{itemize}
    \item the group of diagonal 2-by-2 matrices with positive diagonal entries into the abelian subgroup $A\subset \PSL_d(\R)$ of diagonal $d$-by-$d$ matrices with positive diagonal entries;
    \item $\PSO(2)$ into $K=\PSO(d)\subset \PSL_d(\R)$;
    \item the subgroup $T$ of triangular 2-by-2 matrices with positive diagonal entries into the subgroup $P\subset\PSL_d(\R)$ of triangular $d$-by-$d$ matrices with positive diagonal entries;
    \item and 2-by-2 matrices with positive entries to totally positive $d$-by-$d$ matrices (namely whose minors are all positive).
\end{itemize}

As a consequence, $\tau$ induces
\begin{itemize}
    \item an isometric embedding the hyperbolic plane $\H^2=\PSL_2(\R)/\PSO(2)$ into the symmetric space $\X=G/K$, which is endowed with a nonpositively curved $G$-invariant Riemannian metric;
    \item an embedding of the boundary at infinity $\partial\H^2=\PSL_2(\R)/T$ into the flag variety $\mc F=G/P$, which can be seen as the space of full flags, i.e. sequences 
    \[\xi=(\xi_1\subset \xi_2\subset \cdots \subset \xi_{d}=\mathbb{R}^d)\]
where $\xi_i$ is a linear subspace of $\mathbb{R}^d$ of dimension $i$ for each $i\leq d$.
\end{itemize}

We also fix a basepoint $\basepoint=K\in \X=G/K$, whose stabiliser is $K$.
The subspace $A\cdot\basepoint$ is a totally geodesic embedded Euclidean subspace of $\X$ of maximal dimension. 
This flat identifies with the Cartan subspace $\mathfrak{a}$, which is the linear space of diagonal $(d,d)$-matrices with vanishing trace, through the map $v\in\mf a\mapsto \exp(v)\cdot \basepoint$.
The maximal Euclidean subspaces, called maximal flats, are the translates of $A\cdot \basepoint$ under some $g\in G$.

The stabiliser in $K$ of $\mf a$ is finite and acts by permuting the diagonal entries; the quotient by the subgroup acting trivially on $\mf a$ is the Weyl group, denoted by $\weyl$.
This action is generated by the swaps of two diagonal entries, which act on $\mf a$ by reflections along hypersurfaces called walls.
The open Weyl cone $\mathfrak{a}^+\subset\mf a$ is a natural fundamental domain for this action: it is the open cone of diagonal matrices whose entries $(\lambda_1,\dots,\lambda_d)$ fulfill $\lambda_1> \lambda_2> \cdots > \lambda_d$.

Putting $A^+=\exp(\mf a^+)$, the $K$-orbit of every point $y=g\basepoint \in\X$ intersects the closed Weyl cone $\overline{A^+}\basepoint$ at exactly one point $\exp(u)\basepoint$, and we write $u=\kappa(g)$ and call it the Cartan projection of $g$.
Similarly, the $G$-orbit of any vector $v\in TX$ intersects $\overline{\mf a^+}$ (seen as a subspace of $T_\basepoint\X$) in 
precisely one point $\kappa(v)$ called the \emph{Cartan projection} of $v$.

Being nonpositively curved, $\X$ has a visual boundary $\partial_\infty\X$ on which acts $G$.
The $G$-orbit of every point of $\partial_\infty\X$ intersects exactly once the visual boundary $\partial_\infty (\overline{A^+}\basepoint)$ of our preferred Weyl cone; in other words the $G$-translates of $\partial_\infty (\overline{A^+}\basepoint)$, called the Weyl Chambers (at infinity), cover $\partial_\infty\X$.
The stabiliser of $\partial_\infty (\overline{A^+}\basepoint)$ is $P$, so the space of Weyl Chambers identifies with the flag variety $\mc F=G/P$.

Two flags $\xi=(\xi_1,\dots,\xi_d)$ and $\eta=(\eta_1,\dots,\eta_d)$ are transverse if $\xi_i$ and $\eta_{d-i}$ are in direct sum for every $i$.
This is equivalent to the existence of a maximal flat $F(\xi,\eta)$ and two opposite Weyl Cones in it whose boundaries at infinity are $\xi$ and $\eta$.

The \emph{Jordan projection} $\lambda(g)\in \overline{\mf a^+}$ of $g\in G$ is the diagonal matrix whose diagonal entries are the moduli of the eigenvalues of $g$ in descending order.
The element $g\in G$ is called \emph{loxodromic} if $\lambda(g)$ is contained in the interior $\mf a^+$ of  $\overline{\mf a^+}$, which is equivalent to saying that $g$ has an attracting/repelling fixed pair of transverse flags $(g^-,g^+)$. 
Then $g$ acts as a translation on the flat $F(g^-,g^+)$ with direction prescribed by its Jordan projection.

\medskip\noindent 
  {\bf A Finsler metric coming from a linear functional on $\mathfrak a$}\label{finslergeo}

\begin{notation}\label{nota:alpha}
   We fix a linear functional $\alpha_0$ on $\mf a$ which is positive on 
   $\overline{\mf a^+}$ and such that $\alpha_0(gv)< \alpha_0(v)$ for all 
   $v\in {\mf a^+}$ and $g\in\weyl$.

   We assume that $\alpha_0$ is symmetric in the sense that if $g$ is the transformation in 
   the Weyl group that maps $\mf a^+$ to its opposite $-\mf a^+$ then $\alpha_0(gv)=-\alpha_0(v)$ for any $v\in \mf a$.
%
\end{notation}

An example of a linear functional satisfying the above conditions is given in Equation~\ref{eq:ex of alpha0}.

For any vector $v\in T\X$ we set
\begin{equation}\label{eq:mfF}\mf F(v)=\alpha_0(\kappa(v))\end{equation}
where as before, $\kappa(v)\in \overline{\mf a^+}$ is the Cartan projection of $v$.

\begin{proposition}
[{Lemmas 5.9-10 of~\cite{KL18}}]\thlabel{ex:admissible finsler}
The following hold.
\begin{enumerate}
\item 
$\mf F$ defines a $G$-invariant Finsler metric on $\X$.
\item The unparameterized Riemannian geodesics of $\X$ are also geodesics for $\mf F$.
\item 
The translation length for $\mf F$ of any element $g\in G$ acting on $\X$ is given by $\ell^{\mf F}(g):=\alpha_0(\lambda(g))$ 
where $\lambda(g)\in\overline{\mf a^+}$ is the Jordan projection.
\end{enumerate}
\end{proposition}

In the sequel we always normalize the functional $\alpha_0$ in such a way that 
the embedding $\H^2\to \X$ which is isometric for the symmetric metric also is
isometric for the Finsler metric $\mf F$.

Finsler geodesics  between two distinct points in $\X$ are in general not unique. Indeed, for 
$x,y\in \X$, the \emph{diamond} defined by 
\[D(x,y)=\{z\mid d^{\mf F}(x,z)+d^{\mf F}(z,y)=d^{\mf F}(x,y)\}\]
is the set of all points on a geodesic connecting $x$ to $y$ (see §5.1.3 of \cite{KL18} or §3 of \cite{BHM25}). 
The diamond is contained in any maximal flat $F$ containing $x,y$, where it is a compact convex polytope.
It is the intersection of a Weyl Cone centered at $x$ and a Weyl Cone centered at $y$, opposite to each other.

\medskip\noindent 
  {\bf Busemann functions and Gromov product}

The Busemann functions, or horofunctions, are generalizations of distance functions on $\X$: they record relative distances to a point at infinity.
The \emph{Busemann function} associated to our choice of Finsler metric is given by 
\begin{equation}\label{busemann}
b^{\mf F}_\xi(x,y)=\lim_{n\to\infty}d^{\mf F}(x,z_n)-d^{\mf F}(y,z_n)\in\R.
\end{equation}
where $(z_n)_n\subset\X$ converges to a point of the visual boundary in the interior of $\xi$.


The \emph{Gromov product} between two transverse flags $\xi,\eta\in\mc F$ computed at the basepoint $x\in \X$ is defined as
\begin{equation}\label{gromovproduct}
\langle\xi|\eta\rangle_x=\lim_{n\to\infty}
\left( d^{\mf F}(y_n,x)+d^{\mf F}(x,z_n)-d^{\mf F}(y_n,z_n) \right) \in\R_{\geq 0} 
\end{equation}
where $(y_n)_n,(z_n)_n\subset\X$ are sequences converging to  points of the visual boundary in the interior of $\xi$ and $\eta$ respectively.
Note that we used an unusual convention by not including the factor $\tfrac12$.
This will make the computations a bit easier to read. 

If $x$ is contained in the flat connecting $\eta$ to $\xi$ then $\langle\xi|\eta\rangle_p=0$, which leads to
\begin{equation*}
\langle\xi|\eta\rangle_x=b^{\mf F}_\xi(x,p)+b^{\mf F}_\eta(x,p)
\end{equation*}
We refer to~\cite{KLPMorse} for more information on this construction, in particular on the existence of the limit
in the formula (\ref{gromovproduct}).

\section{Equilibrium states, Hitchin representations and pressure metrics}\label{sec:currents}

In this section we introduce the main structures and tools for this article. 
It is subdivided into
three subsections. In the first subsection we introduce geodesic currents for closed surfaces and the 
intersection form. The second subsection contains 
an account of Hitchin representations and length functions defined by Finsler norms. 
We show, using~\cite{BCLS15}, that such length functions can be used to construct a pressure metric on the Hitchin component.
The third subsection contains a summary of 
the main properties of Patterson--Sullivan theory we shall use later on. 

Throughout, $S$ denotes a closed surface of genus $g\geq 2$, equipped with a fixed choice of a 
hyperbolic metric. Thus the universal covering $\tilde S$ of $S$ can naturally be identified with 
the hyperbolic plane $\H^2$.

\subsection{Geodesic currents, length and intersection}\label{sec:geodescicurrents}

A \emph{geodesic current} for $S$ is a non-trivial $\pi_1(S)$-invariant 
Radon measure on the space of oriented geodesics
$\deinf\H^2\times \deinf\H^2-\Delta$ of the hyperbolic plane
$\mathbb{H}^2$ (here $\Delta$ is the diagonal).
Two such currents are projectively equivalent if
they are constant multiples of each other.
An equivalence class for this equivalence relation is a
\emph{projective geodesic current}. The space
${\mathcal C}(S)$ of geodesic currents for $S$ is equipped with the
weak$^*$-topology which descends to a topology on the
space $\mathcal{P}\mathcal {C}(S)$ of projective geodesic currents.
A \emph{(projective) measured geodesic lamination}
is a (projective) geodesic current
whose support consists of pairwise disjoint simple geodesics.
The space $\mathcal{PML}$  
of projective measured geodesic laminations is a closed
subset of $\mathcal{PC}(S)$.

Each hyperbolic metric on $S$ determines a geodesic
current, the \emph{Liouville current} of the metric.
The following is due to Bonahon~\cite{Bo88}.

\begin{theorem}[Bonahon]
Associating to a hyperbolic metric on $S$ its projective Liouville current
defines an embedding of the Teichm\"uller space into
${\mathcal P\mathcal C}(S)$, and its complement in its closure is the space of projective measured geodesic laminations.
\end{theorem}

The idea behind this theorem rests on the existence of an 
\emph{intersection form}
\[\iota:{\mathcal C}(S)\times {\mathcal C}(S)\to [0,\infty)\]
which extends the geometric intersection number between two 
closed curves on $S$. The form $\iota$ has 
the following properties (see Chapter 8 of~\cite{Mar16}).
\begin{enumerate}
    \item $\iota$ is continuous for the weak$^*$-topology.
    \item If $\lambda$ is the  Liouville current of a hyperbolic metric $\rho$ on $S$, and 
    if $\alpha\subset S$ is any closed geodesic, then 
    \begin{equation}\label{iota}\iota(\lambda,\alpha)=\ell_\rho(\alpha)\end{equation} the $\rho$-length of $\alpha$.
\end{enumerate}

The intersection $\iota(\lambda_1,\lambda_2)$
of two Liouville currents $\lambda_1,\lambda_2$ 
of two hyperbolic metrics on~$S$  has another interpretation
which is important for us. 
Namely, the choice of a hyperbolic metric $h_1$ on $S$ determines the 
\emph{geodesic flow} $\Phi^t$ on the unit 
tangent bundle $T^1S$ of $S$ and a H\"older structure on $T^1S$. 
Given these data, any geodesic current $\mu$ for $S$ extends
to a $\Phi^t$-invariant finite Borel measure $\hat \mu$ on~$T^1S$. Thus given 
a H\"older continuous positive function $f:T^1S\to (0,\infty)$, the 
integral 
\begin{equation}\label{I}\int f d\hat \mu={\bf I}(\mu,f)\end{equation}
is defined. By invariance, this integral only depends on the \emph{cohomology class}
of $f$.
This means that if $f^\prime$ 
is another H\"older function such that
$\int_\gamma f^\prime =\int_\gamma f$ for every periodic orbit~$\gamma$ for 
$\Phi^t$ then $\int f^\prime d\hat \mu=\int f d\hat \mu$ and hence
${\bf I}(\mu,f^\prime)={\bf I}(\mu,f)$.

As a consequence, the pairing ${\bf I}(\mu,f)$ is a pairing between cohomology
classes of (positive) H\"older functions on $T^1S$ and geodesic currents without having 
to make reference to the geodesic flow $\Phi^t$ which depends on the background metric.
Furthermore, the pairing is continuous, where as before, ${\mathcal C}(S)$ is 
equipped with the weak$^*$-topology, and the space of H\"older cohomology classes
is equipped with the quotient topology obtained from the space of 
H\"older functions on $T^1S$ for a fixed reference metric.
We refer to~\cite{Ha99} for more details. 

Assume now that  $f_2$ is a H\"older function which integrates over each 
periodic orbit $\gamma$ for $\Phi^t$ to the length of the free homotopy class of 
$\gamma$ for another hyperbolic metric $h_2$. If $\lambda_1,\lambda_2$ are 
the Liouville currents of $h_1,h_2$, then we have
\[\iota(\lambda_1,\lambda_2)={\bf I}(\lambda_1,f_2).\]

A H\"older continuous positive function $f$ on $T^1S$ can be used to 
reparameterize the flow $\Phi^t$. This reparameterization is defined 
by 
\[\Phi_f^t(v)=\Phi^{\sigma(v,t)}(v)\]
where $\int_0^{\sigma(v,t)}f(\Phi^sv)ds=t.$
For the reparameterized flow, the function $f$ is cohomologous to the 
constant function $1$. 
This is equivalent to stating that the period
of a periodic orbit $\gamma$ for the flow $\Phi_f^t$ equals the integral of 
$f$ over the corresponding orbit for $\Phi^t$.
The identity $(T^1S,\Phi^t)\to (T^1S,\Phi_f^t)$ is an \emph{order preserving
orbit equivalence} between the flows~$\Phi^t,\Phi^t_f$.

Denote by $h_\mu$ the entropy of a $\Phi^t$-invariant Borel probability
measure $\mu$ on $T^1S$. For a positive H\"older function $f$ 
let $\delta(f)>0$ be such that ${\rm pr}(-\delta(f)f)=0$ where 
\[{\rm pr}(u)=\sup_\mu \left( h_\mu+\int ud\mu\right)\]
and $\mu$ runs through all $\Phi^t$-invariant Borel probability
measures on $T^1S$. Then 
\[ h_\mu-\delta(f) \int f d\mu\leq 0\] for all $\mu$. 
A measure $\mu$ is
called a \emph{Gibbs equilibrium state} for $f$ if the equality in this inequality holds.
Using the fact that an order preserving orbit equivalence
between two flows induces an isomorphism between the flow invariant 
probability measures and a formula relating entropies due to Abramov~\cite{abramov}, existence and uniqueness of an equilibrium state for 
the continuous function $\delta(f)f$ is equivalent to existence and uniqueness
of a measure of maximal entropy for the geodesic flow $\Phi_f^t$ on 
$T^1S$, which is well known for H\"older functions (see~\cite{HK95} for more details). 
The constant
$\delta(f)$ then equals the topological entropy of $\Phi_f^t$. 

Let $\mu_{f}$ be the scalar multiple of the unique Gibbs equilibrium state for $f$ such that $\int fd\mu_{f}=1$ (so it is not a necessarily a probability measure).
Then $\mu_{f}$ can be obtained as a limit
\begin{equation}\label{eq:gibbs as limit}\mu_f=\lim_{R\to \infty} \frac{1}{\myhash  N_f(R)}
\sum_{\ell_f(\gamma)\leq R} \frac{\mc D_{\gamma}}{\ell_f(\gamma)}\end{equation}
where $\ell_f(\gamma)=\int_\gamma f$ is the period of $\gamma$ for $\Phi^t_f$, where 
$\mc D_{\gamma}$ is the $\Phi^t$-invariant measure on the periodic orbit $\gamma$ whose total mass
is the $\Phi^{t}$-period of $\gamma$, 
and where $N_f(R)=\{\gamma\mid \ell_f(\gamma)\leq R\}$. 
Thus by continuity of the pairing ${\bf I}$, for any (positive) H\"older function $u$ we have
\begin{equation}\label{intersectionformula}
{\bf I}(\mu_f,u)=\int ud\mu_{f} = \lim_{R\to \infty}\frac{1}{\myhash  N_f(R)} \sum_{\ell_f(\gamma) \leq R}\frac{\ell_u(\gamma)}{\ell_f(\gamma)}.
\end{equation}

Following~\cite{BCLS15}, we also define the \emph{normalized intersection number} by 
\[{\bf J}(f,u)=\frac{h(u)}{h(f)}{\bf I}(\mu_f,u)\]
where $h(u)=\lim_{R\to \infty} \frac{1}{R} \log \myhash  N_{u}(R)$.

\subsection{Hitchin representations}\label{sec:hitchinrep}

In this section we introduce Hitchin representations and summarize those of their properties which 
are important later on. Our main goal is to show that the $G$-invariant Finsler metric
$\mf F$ defined in (\ref{eq:mfF}) induces a pressure metric on the Hitchin component. 

The \emph{Hitchin component} ${\rm Hit}(S)$ 
for conjugacy classes of representations 
$\pi_1(S)\to \PSL_d(\mathbb{R})$ is the connected component of the set of conjugacy classes of representations which 
factor through an irreducible representation $\PSL_2(\mathbb{R})\to \PSL_d(\mathbb{R})$.
In the sequel we always work with explicit representations rather than with conjugacy classes. 

An important property possessed by Hitchin representations is the Anosov property first introduced by Labourie~\cite{Lab06}, which plays a central role in~\cite{BCLS15} in the definition of the pressure metric.
There are many different versions of the Anosov property, and many 
equivalent characterisations of the Anosov property, see for example~\cite{Lab06,GW12,KLP17,GGKW,BPS19,KP22}, and Theorem~4.37 of~\cite{SurveyFanny} for more details and history.

\begin{definition}\label{projectiveanosov}
A representation
$\rho:\pi_1(S)\to \PSL_d(\mathbb{R})$ is \emph{projective Anosov}
if there exist $\rho$-equivariant H\"older continuous maps 
$\xi:\deinf\tilde S\to \mathbb{R}P^{d-1}$, $\theta:\deinf\tilde S\to (\mathbb{R}P^{d-1})^*$ (where $(\mathbb{R}P^{d-1})^*$ is the dual projective space) such that
\begin{enumerate}
    \item if $x,y$ are distinct points in $\deinf\tilde S$, then $\xi(x)+ \ker\theta(y)=\mathbb{R}^d$, and
    \item if $\gamma_n\in\pi_1(S)$ is a sequence so that for some basepoint $\basepoint\in \tilde S=\H^2$, the sequence 
    $\gamma_n\basepoint$ 
    converges to $x\in\partial_\infty\H^2$, and $\gamma_n^{-1}\basepoint\to y\in \partial_\infty \H^2$, then 
    we have $\rho(\gamma_n)p\to \xi(x)$ for any $p\in\R P^{d-1}-\ker\theta(y)$ and $\rho(\gamma_n^{-1})q\to \theta(y)$ for any $q\in (\R P^{d-1})^*$ such that $\xi(x)\not\in \ker q$. 
\end{enumerate}
\end{definition}

\begin{remark}\label{consolidatedef}
In the references given for the characterisations of the Anosov property, the limit map is only required to be continuous, and then the H\"older regularity is derived as a consequence of the other conditions, see for instance Theorem~6.58 of~\cite{KLP17}.
\end{remark}

The following is due to Labourie~\cite{Lab06} and Fock--Goncharov~\cite{FG06}.

\begin{theorem}[Labourie, Fock--Goncharov]
Every representation in the Hitchin component is projective Anosov.
\end{theorem}

As in~\cite{BCLS15}, let $F$ be the total space of the bundle over 
\[(\mathbb{R}P^{d-1})^{(2)}=\mathbb{R}P^{d-1}\times (\mathbb{R}P^{d-1})^*
- \{(U,V)\mid U\subset {\rm ker}(V)\}\]
whose fiber at a point $(U,V)$ is the space
\[M(U,V)=\{(u,v)\mid u\in U,v\in V,\langle v\mid u\rangle =1\}/\sim\]
where $\langle v \mid u\rangle$ is the natural pairing between a vector 
and a covector and 
$(u,v)\sim (-u,-v)$. Note that $u$ determines $v$ so that 
$F$ is an $\mathbb{R}$-bundle. 

The bundle $F$ is equipped with a natural $\mathbb{R}$-action, given by
\[\Phi_F^t(U,V,(u,v))=(U,V,(e^tu,e^{-t}v)).\]
Given a projective Anosov representation 
$\rho:\pi_1(S)\to\PSL_d(\mathbb{R})$ and $\xi,\theta$ the associated limit maps, 
we consider the pullback bundle 
\[F_\rho=(\xi,\theta)^*F\to\deinf\wt S\times\deinf\wt S-\Delta\]
by the map $\deinf\wt S\times\deinf\wt S-\Delta\xrightarrow[]{(\xi,\theta)}(\R P^{d-1})^{(2)}$, which inherits an $\mathbb{R}$-action 
from the action of $\Phi^t_F$. The actions $\pi_1(S)\acts_\rho\R^d$ and $\pi_1(S)\acts\deinf\wt S\times\deinf\wt S-\Delta$ 
extend to an action on~$F_\rho$. If we let 
\[U_\rho S=\pi_1(S)\backslash F_\rho \] 
then the $\mathbb{R}$-action on $F_\rho$ descends to a flow $\Phi_\rho^t$ on 
$U_\rho S$ which is called the \emph{spectral radius flow} of the representation
(see p.1118 of~\cite{BCLS15}).

The following statement combines Propositions 4.1, 4.2 and 6.2 of~\cite{BCLS15}.
It is valid for any analytic family of 
projective Anosov representations.

\begin{proposition}\label{orbitequ}
\begin{enumerate}
    \item 
For every representation $\rho$ in the Hitchin component 
there exists a H\"older continuous order preserving orbit equivalence
$\Psi_\rho:(T^1S,\Phi^t)\to (U_\rho S,\Phi^t_\rho)$. Any 
primitive element $\gamma\in \pi_1(S)$ has period $\log \Lambda(\rho)(\gamma)$
where $\Lambda(\rho)(\gamma)$ is the spectral radius of 
$\rho(\gamma)\in\PSL_d(\mathbb{R})$.
\item If $D$ is the unit disk and if 
${\rho_u}$ $(u\in D)$ is a real analytic family of Hitchin representations, 
then up to decreasing the size of $D$, there exists a real analytic
family $\{f_{\rho_u}:T^1S\to \mathbb{R}\}_{u\in D}$ of positive H\"older functions
such that the reparameterization of $T^1S$ by $f_{\rho_u}$ is H\"older conjugate to $U_{\rho_u}$
for all $u\in D$.
\end{enumerate}
\end{proposition}

As a consequence, the spectral radius length defines a \emph{pressure metric}
on ${\rm Hit}(S)$ as follows.
For any smooth deformation $\rho_t$ of a representation $\rho=\rho_0$,
put 
\begin{equation}\label{pressuremetric1}
 \Vert \rho^\prime(0)\Vert^2=\frac{d^2}{dt^2}\vert_{t=0}{\bf J}(f_{\rho(0)},f_{\rho(t)})    
\end{equation} 
(Theorem 1.3 of~\cite{BCLS15}) where 
$f_{\rho(s)}$ is the H\"older function constructed from $\rho(s)$ as in Proposition~\ref{orbitequ}.
That this construction defines indeed a (mildly degenerate) Riemanian metric on 
the Hitchin component which determines a distance function was established in~\cite{BCLS15}.
It is based on the fact that projective Anosov representations are dominated in the sense of 
\cite{BPS19}. We refer to~\cite{BCLS15},~\cite{BPS19} and~\cite{Sam24} for more precise information.

The pressure metric we are interested in is a more geometric version of the metric  \eqref{pressuremetric1}. 
To define this metric we need to review some additional properties of representations in ${\rm Hit}(S)$. 
Let as before $\mc F$ be the variety of full flags in $\mathbb{R}^d$.

\begin{definition}[\cite{GGKW,KLP17}]\label{def:anosov}
    A representation $\rho:\pi_1(S)\to G$ is \emph{Borel Anosov} if the following holds true.
   \begin{enumerate}
   \item    
   There exists 
    a (unique) equivariant H\"older embedding 
    $\partial_\infty \rho:\partial\H^2\to\mc F$ such that 
    $\partial_\infty \rho(\xi)\pitchfork \partial_\infty \rho(\eta)$ for all $\xi\neq\eta\in\partial\H^2$. 
   \item 
   For any diverging sequence $(\gamma_n)_n\subset\pi_1(S)$ such that $\gamma_n\to\xi\in\partial\H^2$ and $\gamma_n^{-1}\to\eta$, we have $\rho(\gamma_n)\zeta\to \partial_\infty \rho(\eta)$ for any $\zeta\in\mc F$ transverse to $\partial_\infty \rho(\xi)$.
\end{enumerate}
\end{definition}

By the groundbreaking work of Labourie and Fock--Goncharov, we have

\begin{theorem}[\cite{Lab06,FG06}]\label{hitanosov}
    All Hitchin representations $\pi_1(S)\to\PSL_d(\R)$ are Borel Anosov.
\end{theorem}

Our goal is to construct a pressure metric on ${\rm Hit}(S)$ as in
\eqref{pressuremetric1}, but using the fixed length function $\ell^{\mf F}(g)=\alpha_0(\lambda(g))$ of the Finsler norm 
and its associated renormalised intersection form 
$$
{\bf J}(\rho,\rho^\prime) = \frac{h(\rho')}{h(\rho)} \lim_{R\to \infty}\frac{1}{\myhash  N_\rho(R)} \sum_{\ell^{\mf F}(\rho(\gamma)) \leq R}\frac{\ell^{\mf F}(\rho'(\gamma))}{\ell^{\mf F}(\rho(\gamma))},
$$
where 
\[N_\rho(R)=\{[\gamma]\in[\pi_1(S)]:\ell^{\mf F}(\rho(\gamma))\leq R\}\text{ and }h(\rho)=\lim_{R\to\infty}\frac1R\log\myhash N_\rho(R).\]
To this end we have to establish an analog of 
Proposition~\ref{orbitequ} for this new length function. 
We shall reduce this statement to 
Proposition~\ref{orbitequ} using the following classical observation.

Let $\xi=(\xi_1\subset \cdots \subset \xi_{d})$ be a full flag in $\R^d$. 
Then for each $k\leq d-1$ the $k$-th exterior power $\bwedge^k(\xi_k)$ is one-dimensional.
A non-zero element $\omega$ of this vector space defines up to a non-zero multiple 
a non-zero linear functional
$\Psi(\omega):\bwedge^{d-k}(\mathbb{R}^d)\to \mathbb{R}$ as follows. Choose a non-zero element 
$\nu\in \bwedge^d(\mathbb{R}^d)$ and put $\Psi(\omega)(\alpha)= c$ if 
$\omega \wedge \alpha=c \nu$. Note that the kernel of $\Psi(\omega)$ is spanned by all 
decomposable elements of $\bwedge^{d-k}\mathbb{R}^d$ which are not transverse to 
$\xi_k$. 

If $\rho:\pi_1(S)\to G$ is Borel Anosov, then by the definition of 
the transversality relation $\pitchfork$, for any two distinct points 
$\xi\not=\eta\in \partial \H^2$, the $d-k$-th subspace 
$\partial_\infty \rho(\xi)_{d-k}$ of the flag 
$\partial_\infty \rho(\xi)$ defines a line of linear functionals on $\bwedge^{k}(\R^d)$ which do not evaluate
to zero on $\bwedge^k \partial_\infty \rho(\eta)_k$, where 
$\partial_\infty \rho(\eta)_k$ is the $k$-dimensional subspace 
of the flag $\partial_\infty \rho(\eta)$. Thus if $\bwedge^k\rho:\pi_1(S)\to \PSL_{d_k}(\R)$ denotes 
the representation induced by $\rho$ 
 into the full linear group of $\bwedge^k(\R^d)$ 
where $d_k$ denotes the dimension of $\bwedge^k(\R^d)$, then 
as the map $\partial_\infty \rho:\partial_\infty \H^2\to {\mc F}$ is H\"older continuous, the following well-known statement 
holds true.

\begin{lemma}\label{borelyieldsprojective} 
If $\rho:\pi_1(S)\to G$ is Borel Anosov, then for any $k<d$, the induced representation 
$\bwedge^k\rho$ is projective Anosov.
\end{lemma}

\begin{remark}
It follows from the above discussion that in fact, $\rho$ is Borel Anosov if and only 
if for each $k\leq d-1$ the induced representation on $\bwedge^k(\R^d)$ is projective Anosov.
We refer to Section 4 of~\cite{BPS19} for more details on this relation. 
\end{remark}

Thus we can apply Proposition~\ref{orbitequ} to each representation $\bwedge^k\rho$. 
Recall from Section~\ref{sec:lietheoryI} the definition of the Jordan projection $\lambda$. 
As implicitly stated in 
\cite{BPS19}, we obtain the regularity statement on Finsler length functions needed 
to define a pressure metric.

\begin{proposition}\label{hoelderfunction}
    For every Borel Anosov representation $\rho_0:\pi_1(S)\to \PSL_d(\R)$, 
    there exists an open neighborhood $U$ of $\rho_0$ made of Borel Anosov representations
    and a real analytic family $\{f_\rho:T^1S\to \mf a\}_{\rho\in U}$ of H\"older functions, valued in $\mf a^+$, 
    such that for any~$\gamma\in \pi_1(S)$, we have
     $$\lambda(\rho(\gamma))=\int f_\rho d\gamma.$$
\end{proposition}
\begin{proof}
    Proposition~\ref{orbitequ} implies that there exists an open neighborhood $U$ of $\rho_0$ and real analytic families $\{g^k_\rho:T^1S\to \R\}_{\rho\in U}$ of H\"older functions such that for any $\rho\in U$, each exterior product $\bwedge^k\rho$ is projective Anosov,
    and for any $\gamma\in\pi_1(S)$, the logarithm $\log \Lambda(\bwedge^k\rho(\gamma))$ of the spectral radius
    of $\bwedge^k(\rho(\gamma))$ equals 
    $$
    \log \Lambda(\bwedge^k\rho(\gamma))=\int g^k_\rho d\gamma.
    $$

    Then 
    we can consider the following H\"older function
    $$
    f_\rho = (g^1_\rho,\ g^2_\rho-g^1_\rho,\ g^3_\rho-g^2_\rho,\dots,\ 
    g^d_\rho-g^{d-1}_\rho)\in \mf a.
    $$
    By Proposition~\ref{orbitequ}, the function
    $f_\rho$ depends analytically on $\rho$. Moreover, for any $\gamma\in \pi_1(S)$, we have
     $$\lambda(\rho(\gamma))=\int f_\rho d\gamma.$$

     It is not clear, however, that $f_\rho$ is valued in the open Weyl chamber $\mf a^+$.
     Let us solve this issue by first replacing $f_{\rho_0}$ by an $f'_{\rho_0}$ valued in $\mf a^+$, using work of Sambarino, and then extend $f'_{\rho_0}$ to a small neighborhood of representations $\rho$, using a theorem of Liv\v{s}ic.
     


     We apply Sambarino's reparametrization result to the lengths functions $\alpha_k\circ\lambda\circ\rho_0(\gamma)$ where $\alpha_k(v_1,\dots,v_d)=v_k-v_{k+1}$, see Theorem~3.2 of~\cite{Sam14} (Sambarino proved in pages 481-483 that we can apply this theorem to our setting).
     This gives us positive H\"older functions $u^k:T^1S\to\R$ such that for any $\gamma\in \pi_1(S)$, we have
     $$\alpha_k\circ\lambda\circ\rho_0(\gamma)=\int u^k d\gamma.$$
     Let $f'_{\rho_0}:T^1S\to\mf a$ be  such that $\alpha_k\circ f'_{\rho_0}(v)=u^k(v)>0$ for all $v\in T^1S$ and $1\leq k\leq d-1$.
     Then $f'_{\rho_0}$ is valued in the interior of $\mf a^+$ by definition, it is H\"older, and $\lambda(\rho_0(\gamma))=\int f'_{\rho_0} d\gamma$ for any $\gamma$.

     For any periodic orbit $\gamma$ in $T^1S$ 
     we have $\int f_{\rho_0} d\gamma=\int f'_{\rho_0} d\gamma$, so by Theorem~1 of~\cite{Livsic} $f'_{\rho_0}$ and $f_{\rho_0}$ are cohomologous, in the sense that there exists $F:T^1S\to\mf a$ differentiable in the direction of the geodesic flow $\Phi^t$  
     such that $f'_{\rho_0}=f_{\rho_0}+\frac d{dt}_{|t=0}F\circ\Phi^t$. Put 
     \[f'_{\rho}
     =f_\rho+\frac d{dt}_{|t=0}F\circ\Phi^t\] for any $\rho\in U$, so that 
     $\lambda(\rho(\gamma))=\int f'_\rho d\gamma$ for any $\gamma$.
     This yields an analytic family of H\"older functions which take values 
     in $\mf a^+$ for all $\rho$ contained in a sufficiently small neighborhood $U'\subset U$ of $\rho_0$. 
     This is what we wanted to show.
\end{proof}

Consider now a $C^2$--path of representations  $(\rho_t)_t$ in 
the neighborhood $U$ constructed in Proposition~\ref{hoelderfunction} with initial value $\rho_0$.
Let $\mu_0$ be the equilibrium state on $T^1S$ associated to $f_{\rho_0}$ introduced 
in Subsection~\ref{sec:geodescicurrents}, 
normalized so that $\int f_{\rho_0}d\mu_0=1$.
Denote by $h(t)$ the entropy associated to $f_{\rho_t}$.
Following~\cite{BCLS15} we set
\begin{equation}\label{eq:def pressure metric}
    \left\Vert\frac {d}{dt}_{|t=0}\rho_t\right\Vert^2 = 
    \frac{1}{h(0)} \int \frac {d^2}{dt^2}_{|t=0}(h(t)\cdot \alpha_0\circ f_{\rho_t}) d\mu_0.
\end{equation}
It follows from Proposition~\ref{hoelderfunction} and~\cite{BCLS15} that this is well defined and is
indeed the square norm for a (perhaps degenerate) Riemannian metric on ${\rm Hit}(S)$ which is a variant
of the simple root length metric (\ref{pressuremetric1}) considered in~\cite{BCLS15}.
We call this metric the \emph{Finsler pressure metric} on ${\rm Hit}(S)$.

\subsection{Patterson--Sullivan theory}\label{sec:meas-conv}

\paragraph{Patterson--Sullivan theory for hyperbolic metrics.}

    Patterson~\cite{Patterson76} and Sullivan~\cite{sullivan79} introduced a construction of measures on $\deinf\H^2$ which allows 
    to obtain the 
    entropy maximizing invariant probability measure of the geodesic flow on a compact hyperbolic surface as a product measure. This construction has been generalised in various settings. We recall some important facts about their theory and the generalization to the case of interest for us. 

    Let as before $S$ be a closed surface of genus $g\geq 2$ and let 
    $\Gamma=\rho(\pi_1(S))\subset\PSL_2(\R)$ be a Fuchsian representation, determined by the choice of a hyperbolic metric on $S$.
    For $\xi\in\deinf\H^2$ and $x,y\in\H^2$, we denote by $b_\xi(x,y)$ the 
    Busemann function of $(x,y)$ based at~$\xi$, defined as in (\ref{busemann}).
    Up to multiplication by a constant, there exists a unique family of finite measures $(\nu^x)_{x\in\H^2}$ which all define the same measure class, 
    and which satisfy the following. For all $x,y\in\H^2$ and $\xi\in\deinf\H^2$,
    \begin{equation}\label{eq:busemann-relation}
        \frac{\de\nu^x}{\de\nu^y}(\xi)=e^{b_\xi(x,y)}.
    \end{equation}
    The measures $\nu^y$ can be obtained as a limit of measure of the form 
    \begin{equation}\label{patterson1}
    \frac{1}{c_s}\sum_{g\in\Gamma}e^{sd(y,g\cdot x)}\delta_{g\cdot x}\end{equation}
    with $s$ converging from above toward $1$ (which equals the \emph{critical exponent} of $\Gamma$), and the constant $c_s$ is chosen
    so that for $y=x$, the measures in (\ref{patterson1}) are probability measures. 

    From the measure class $\nu^x$ we define a measure on $\deinf\H^2\times\deinf\H^2\setminus\Delta$ invariant by the action of $\Gamma$. 
    Recall from (\ref{gromovproduct}) the definition of 
    the Gromov product $\langle\xi|\eta\rangle_x$ of $(\xi,\eta)$ based at $x$. It can be computed by 
    \begin{equation}\label{eq:Gromov-product1}
      \langle\xi|\eta\rangle_x=b_\xi(x,z)+b_\eta(x,z)
    \end{equation}
    for any $z$ on the geodesic with endpoints $\xi$ and $\eta$. 
    The value does not depend on the choice of $z$.  
    Then define the measure $\hat \nu$ on $\deinf\H^2\times\deinf\H^2\setminus\Delta$ by 
    \begin{equation}\label{eq:invariant-current}
        d\hat \nu(\xi,\eta)=e^{\langle\xi|\eta\rangle_x}\cdot d\nu^x(\xi)\times d\nu^x(\eta)
    \end{equation}
    The measure $\hat \nu$ is invariant under the action $\Gamma\acts\deinf\H^2\times\deinf\H^2\setminus\Delta$, is finite on compact
    sets and does not depend on $x$~\cite{sullivan79}. 

    The unit tangent bundle $T^1S$ of the surface $S=\Gamma \backslash \H^2$ is endowed with a geodesic flow $\Phi^t$.
    It is Anosov, so it admits a unique entropy maximizing invariant probability measure. This measure lifts to 
    a $\Gamma$-invariant $\Phi^t$-invariant Radon measure on $T^1\mathbb{H}^2$ which disintegrates to the measure $\hat\nu$.
    Namely, $\deinf\H^2\times\deinf\H^2\setminus\Delta$ is just the set of oriented geodesics in $\H^2$, 
    and $d\hat \nu\times dt$ defines a $\Phi^t$-invariant $\Gamma$-invariant Radon measure on $T^1\H^2$, where 
    $dt$ is the one-dimensional Lebesgue measure on flow lines. This measure projects 
    to a finite Borel measure on $T^1S$ in the Lebesgue measure class, which can be scaled to be a probability measure. 

    \paragraph{Patterson--Sullivan theory for Hitchin representations} 

    Patterson Sullivan theory was generalized to many different geometric settings. In the setting of 
    Finsler metrics on higher rank symmetric space and Hitchin representations, such a generalization is due to Kapovich and Dey 
   ~\cite{DK22} (the results are valid for all Anosov representations). 
    Namely, given a Hitchin representation $\rho:\pi_1(S)\to \PSL_d(\mathbb{R})$, 
    define a Poincar\'e series
\[P^{\mf F}(\rho,s)(x,y)=\sum_{\psi\in \rho(\pi_1(S))}e^{-s d^{\mf F}(y,\psi x)}\] 
where as before, $d^{\mf F}(y,z)$ is the distance between $x,y$ for the Finsler metric $\mf F$. 
Part (iv) of Theorem A of~\cite{DK22} shows that this series diverges at the 
\emph{critical exponent} $\delta_\rho$. Moreover, it defines 
a family $\mu^x$ 
of Borel measures on the \emph{limit set} $\Lambda\subset {\mathcal F}$ of 
$\rho(\Gamma)$ in the flag variety~${\mathcal F}$, that is, the image of $\partial_\infty \mathbb{H}^2$
under a $\rho$-equivariant H\"older continuous map, 
indexed by the points $x\in \X$. These measures are a \emph{conformal density}, that is,  
they are equivariant under the action of $\rho(\pi_1(S))$ and transform via
\begin{equation}\label{conformaldensity}
\frac{d\mu^y}{d\mu^x}(\xi)=e^{\delta_\rho b^{\mf F}_\xi(x,y)}
\end{equation}
where $b^{\mf F}_\xi$ is the Busemann function for the Finsler metric. 

Conformal densities had been constructed earlier by Sambarino in~\cite{Sam14}, using a different method and work of Ledrappier~\cite{Ledrappier95}.
Sambarino's construction is dynamical and does not use the Finsler metric $d^{\mf F}$.
Here we will need the geometric approach of Dey--Kapovich.

\begin{remark}\label{sysmmetricbad}
As the limit curve of a Hitchin representation is a curve in the flag variety 
rather than the limit set of the representation in the geometric boundary of $\X$,
the above construction can not be carried out for the symmetric metric. Namely, 
as the limit set of the representation in the geometric boundary $\partial_\infty \X$ of $\X$ 
may have points in asymptotic Weyl chambers 
which are not opposite in the Weyl chamber and hence can not be connected by a geodesic, it 
may not be possible to correctly encode translation lengths for the symmetric metric 
by a global H\"older continuous function on $T^1S$.
\end{remark}

\section{Hitchin grafting representations}\label{section-HG}

The Hitchin representations we are interested in are the familiar \emph{bending} or \emph{bulging} deformations
of \emph{Fuchsian} representations, that is, representations  
which factor through the embedding $\tau:\PSL_2(\mathbb{R})\to \PSL_d (\mathbb{R})$. We refer to~\cite{Go86,AZ23,BHM25} for an account
on the bending construction. 
In this section we introduce these representations and
summarize the geometric results from \cite{BHM25} we need.

 \subsection{Grafting}\label{sec:Abstract grafting}

    Consider a closed oriented surface $S$ of genus $g\geq 2$ endowed with a hyperbolic metric.
	A \emph{simple (geodesic) multi-curve} $\gamma^*$ is the union of pairwise disjoint essential 
 mutually not freely homotopic simple closed curves (geodesics) on $S$. 
 We fix moreover an orientation on each component of $\gamma^*$.

 Consider the special direction $u=d\tau\hsl\in \mf a$ given by $\tau$.
 For any $z\in \mf a$ and $\ell>0$, let $\Cyl(\ell,z)\subset \mf a/\ell u$ be the cylinder obtained by quotienting  the strip $\{tu+sz: t\in\R,\ s\in[0,1]\}\subset\mf a$ under the translation by $\ell u$.
 The (Finsler) \emph{height} of such cylinder is defined as 
 \begin{equation}\label{eq:cylinder height}
     \mathrm{height} = \min\{\mf F(tu+z):t\in\R\}.
 \end{equation}
 
 We fix for every $\gamma\in \gamma^*$ a vector $z_\gamma\in \mf a$; the collection $z=(z_\gamma)_{\gamma\in\gamma^*}$ is interpreted as a grafting parameter.
 
	\begin{definition}\label{def:abstractgraf}
         The \emph{abstract grafting} of $S$ along the geodesic multi-curve 
         $\gamma^*$ with grafting parameter $z$  is the surface $S_z$ obtained by cutting $S$ open along
         each of the components $\gamma$ of $\gamma^*$, inserting 
         flat cylinders $C_\gamma=\Cyl(\ell_S(\gamma),z_\gamma)$  and gluing the surface back with the translation by $z_\gamma$.

         If $z_\gamma$ is not parallel to $u$ for any $\gamma\in\gamma^*$, then this grafting comes with a natural homotopy equivalence $\pi_z:S_z\to S$ projecting the flat cylinders onto $\gamma^*$, which allow us to identify $\pi_1(S_z)$ and $\pi_1(S)$.
	\end{definition}

    The abstract grafted surface $S_z$ decomposes into subsurfaces with geodesic boundary which are equipped with 
    a metric of constant curvature. The \emph{hyperbolic part} $S^{\rm hyp}$ is the union of the subsurfaces with a 
    metric of constant curvature $-1$ and can be identified with the union of the components of $S\setminus \gamma^*$.    
    The component $S\setminus S^{\rm hyp}$ is the \emph{cylinder part} and consists of a union of flat cylinders whose
    core curves are freely homotopic to the components of $\gamma^*$.
    
    As the pressure metric for the Hitchin component we are interested in is defined by a Finsler metric on 
    $\X$ 
    using a linear functional 
    $\alpha_0$ (see \eqref{eq:mfF}) rather than the Riemannian one, we also endow $S_z$ with a Finsler metric by equipping each cylinder $C_\gamma$ with the quotient of the non-Euclidean norm $\mf F$ on $\mf a$.
    Observe that in general, for a given $C^1$-structure on $S_z$ as constructed above, 
    this metric is \emph{discontinuous} at the gluing locus between the flat cylinders and the hyperbolic part.
    Additionally the metric on the flat part is sensitive in the direction of $z$, and does not depend only on the height of the grafting (contrarily to the Riemannian metric).
    Nevertheless it induces a well defined path metric on $S_z$.

    Let $G_{\gamma^*}$ be the oriented graph such that each vertex $v\in V$ corresponds to a component $\Sigma_v$ of $\Sigma-\gamma^*$, and each edge $e\in E$ corresponds to an oriented component $\vec \gamma_e$ of $\gamma^*$. 
    Take a discrete and faithful representation $\rho\colon\pi_1(G_{\gamma^*},T)\to \PSL_2(\R)\xrightarrow[]{\tau} \PSL_d(\R)$ 
    which factors through the embedding  
    $\tau:\PSL_2(\mathbb{R})\to \PSL_d(\mathbb{R})$. 
    We use the graphs of groups decomposition of $\pi_1(\Sigma)$ determined by $\gamma^*$
    to perform a bending of the representation in $\PSL_d(\R)$ with parameter $z=(z_\gamma)_{\gamma\in\gamma^*}\in \mf a^{\gamma^*}$.
    This construction can be thought of as
    bending the surface $S$ along the geodesic multicurve $\gamma^*$ in the space of representations into $G$.

    \begin{definition}\label{def:hitchin grafting rep}
        We denote by $\mathrm{Gr}_z^{\gamma^*}\rho\colon\pi_1(G_{\gamma^*},T)\to \PSL_d(\R)$ the representation induced by $\wt\rho_z$, 
        and sometimes just $\rho_z$ if there is only one hyperbolic structure involved. 
        We call it the \emph{Hitchin grafting representation} with data $z$ along $\gamma^*$.
    \end{definition}
    
    Up to conjugation, the representation $\rho_z$ only depends on the grafting parameter $z$.
A \emph{Hichin grafting ray} is a one-parameter family of Hitchin grafting representations $t\to \rho_{tz}$
defined by a ray in $(\mf{a})^k$ where $k$ is the number of components of the multicurve $\gamma^*$ along which the grafting is
performed.

    \subsection{The characteristic surface for Hitchin grafting representations}\label{sec-HG}

    Consider a Fuchsian representation $\rho:\pi_1(S)\to \PSL_2(\R)\to \PSL_d(\R)$ and 
    denote by $S$ the hyperbolic surface defined by this representation.
    Choose some grafting datum $z$ and let 
    $\rho_z$ be the Hithin grafted representation defined by $\rho$ and $z$.
    As this representation is contained 
    in the Hitchin component, it follows from Labourie~\cite{Lab06} and Fock--Goncharov~\cite{FG06} that $\rho_z$ is 
    faithful, with discrete image. 
    In particular, the quotient manifold $\rho_z\backslash \X$ is homotopy equivalent to $S$; 
    in fact $\rho$ induces a natural homotopy class of homotopy equivalences between 
    $\rho_z\backslash \X$ and $S$.

    The following statement is Proposition 2.5 of \cite{BHM25}.

    \begin{proposition}\thlabel{piecewise-isometry}
        Consider a Hitchin grafting representation $\rho_z$ obtained from $\rho$ and with 
        grafting datum $z$.
        Let $S_{z}$ be the  abstract grafting of $S$ from Definition~\ref{def:abstractgraf}, with universal covering $\wt S_{z}$.
        Then there exists a piecewise totally geodesic immersed surface $\wt S^\iota_{z}\subset\X$ and a $\rho_z$-equivariant 
        immersion 
        $\wt Q_z\colon\wt S_{z}\to\wt S_{z}^\iota\subset\X$. 
        
        The map $\wt Q_z$ is a path isometry for the Riemannian (resp.\ Finsler) metric on $\wt S_z$ and the induced path metric on $\wt S^\iota_{z}$ from the Riemannian (resp.\ Finsler) metric on $\X$.
    \end{proposition}

 \subsection{Admissible paths}\label{sec:admissible}
    
 An important tool for the geometric investigation of Hitchin grafting representations are 
   \emph{admissible paths} which were introduced in \cite{BHM25}. They are defined as follows
   (compare Section 2.5 of \cite{BHM25}).

    \begin{definition}
     Consider a closed hyperbolic surface $S$, a  multicurve $\gamma^*\subset S$ and a grafting parameter $z$. 
     Then $S_z$ is the abstract grafted surface with hyperbolic part $S^\hyp\subset S_z$ and flat (cylindrical) part $\mc C\subset S_z$.
     An \emph{admissible path} in $S_z$ is a continuous path $c\subset S_z$ such that
     \begin{itemize}
         \item $c$ is geodesic outside $\mu=S^\hyp\cap\mc C$; 
         \item the hyperbolic part $c\cap S^\hyp$ intersects $\gamma^*$ orthogonally;
         \item a component of the flat part $c\cap \mc C$ connects the two distinct
         boundary components of the flat cylinder containing it.
     \end{itemize}
     Similarly one can define \emph{admissible loops}.
    \end{definition}

    Note that if $z$ is trivial then $S_z=S$ and the above definition still makes sense. 
    The flat part $\mc C$ is just $\gamma^*$, and the path is allowed to contain arcs in $\gamma^*$ 
    separating two geodesic arcs which emanate to the two distinct sides of $\gamma^*$ in a 
    tubular neighborhood of $\gamma^*$.

    An \emph{admissible path} in the universal cover $\tilde S_z$ is the lift of an admissible path in $S_z$.
    Any two points of $\tilde S_z$ are connected by a unique admissible path; in other words, any path of $S_z$ is homotopic (with fixed endpoints) to a unique admissible path.
    Similarly, any loop in $S_z$ not homotopic to a component of $\gamma^*$  
    is freely homotopic to a unique admissible loop.

    We define an
    \emph{admissible path} in a characteristic surface of a Hitchin grafting representation to be
    the image of an admissible path in the abstract grafted surface under the natural path isometry.
    A more conceptual notion of admissible paths which is not needed toward our goal can be found in \cite{BHM25}. 

    For a constant $C>0$, a path $\gamma:[0,T]\to X$ in a metric space $(X,d_X)$ is called \emph{$C$-quasi-ruled} if 
    for any $0\leq t\leq s \leq u\leq T$ it holds
    \[d_X(\gamma(t),\gamma(s))+d_X(\gamma(s),\gamma(u))\leq d_X(\gamma(t),\gamma(u))+C.\]
    The following is Proposition 4.10 of \cite{BHM25}.

    \begin{proposition}\label{prop:triangle equality for admissible paths}
    For any $\omega>0$ there exists $C_\omega$ such that  any $(\omega,0)$-admissible path $c$ in $\X$ 
    is Finsler $C_\omega$-quasi-ruled. Moreover, it is at Hausdorff distance at most $C'_\omega$ 
    from some Finsler geodesic, where $C'_\omega$ only depends on $C_\omega$.
\end{proposition}

While the above discussion gives a geometric account on admissible paths in the symmetric space $\X$, 
we shall also use admissible paths in the group $G$. 
the following algebraic definition of admissible paths in $\X$.
    The description of these paths uses a 
    basepoint for the action of $G$ which is determined by the Fuchsian representation $\tau$ as well as the following notation.

    \begin{notation}\label{nota:a't}
     We set
     \begin{itemize}
      \item $a_t:=\tau\begin{psmallmatrix}e^t&0\\0&e^{-t}\end{psmallmatrix}\in G$;
      \item $r_\theta:=\tau \begin{psmallmatrix}\cos(\theta/2) & \sin(\theta/2) \\ -\sin(\theta/2) & \cos(\theta/2)\end{psmallmatrix}\in G$;
      \item $a'_t:=r_{\pi/2}\cdot a_t\cdot r_{\pi/2}^{-1}\in G$;
      \item for every $t\in\R\cup\{\infty\}=\partial\H^2$ we write $\xi_t=\partial_\infty \tau(t)$.
     \end{itemize}
    \end{notation}

    The group $G=\PSL_d(\R)$ identifies with one component of the space of $\H^2$-frames $Y$ in $G$; $g\mapsto g\cdot F_o$, 
    where $F_o=(o,v_o,w_o)$ is a fixed $\H^2$-frame, so that $o$ is fixed by $r_\theta$,  and $v_o=\tfrac{d}{dt}_{|t=0}a'_t\cdot o$ and $w_o=\tfrac{d}{dt}_{|t=0}a_t\cdot o$ are tangent to the axes of $a'_t$ and $a_t$, respectively.

    Under this identification, the geodesic flow on the space 
    $Y$ of $\H^2$-frames corresponds to the multiplication on the right by $a'_t$: i.e.\ $\mr{geod}_t(gF_o)=(ga'_t)F_o$.
    On the other hand, the orthogonal sliding flow corresponds to the multiplication on the right by 
    $\exp(z)$: that is,\ $\mr{slide}_z(gF_o)=(g \cdot \exp(z))F_o$ for any $z\in \mf a$.
    This leads us to the following definition of admissible path.

    \begin{definition}\label{def:admissible in G}
        A path $c:[0,T]\to G$ or $c:[0,\infty)\to G$ is said to be of 
        \begin{itemize}
            \item \emph{flat type} if $c(t)=g \cdot \exp(tz)$ for some $g\in G$ and $z\in \mathfrak a$ of norm $1$ for the Finsler metric $\mf F$;
            \item \emph{hyperbolic type} if $c(t)=ga'_t$ for some $g\in G$.
        \end{itemize}

        An \emph{admissible path} of $G$ is a \emph{continuous} (possibly infinite) concatenation of paths of flat and hyperbolic type.
\end{definition}

\subsection{Geometric control: Uniform quasi-isometry}\label{QI-proof}

    Recall that $S$ is a hyperbolic closed surface, let
    $G=\PSL_d(\R)$
    and $\tau:\PSL_2(\R)\to G$ be the usual irreducible representation.
    The following is Theorem 5.1 of \cite{BHM25} which was obtained as a consequence of Fock Goncharov positivity. 

    \begin{theorem}\label{thm:quasiisom}
     For every $\sigma>0$, 
     there exists $C_{\sigma}>0$ such that
     the following holds.
    
     Consider a closed hyperbolic surface $S$, a multicurve $\gamma^*\subset S$ whose components have length at most $\sigma$, and a grafting parameter $z$ such that all cylinder heights of the abstract grafting $S_z$ are bounded 
     from below by 
     some number $L>0$.
     
     Let us endow $\X$ with the $G$-invariant 
      admissible Finsler metric $\mf F$ and $S_z$ with the pullback of this metric under $Q_z$, denoted by 
      $d_{\tilde S_z}^{\mf F}$.
     Then the grafting map $\tilde Q_z:\wt S_z\rightarrow\X$ is an injective quasi-isometric embedding with multiplicative constant $(1+C_\sigma/(L+1))$ and additive constant $C_{\sigma}$; more precisely, for all $x,y\in\tilde S_z$ we have
     $$\left(1+\frac{C_\sigma}{L+1}\right)^{-1}d^{\mf F}_{\tilde S_z}(x,y)-C_{\sigma}
     \leq d^{\mf F}(\tilde Q_z(x),\tilde Q_z(y)) \leq d^{\mf F}_{\tilde S_z}(x,y).$$
     Moreover, the image $\tilde S_z^\iota=\tilde Q_z(\tilde S_z)$ is $C_\sigma$-Finsler-quasiconvex in the sense that for all $x,y\in\tilde Q_z(\tilde S_z)$, there is a Finsler geodesic from $x$ to $y$ at distance at most $C_\sigma$ from $\tilde S_z^\iota$.
    \end{theorem}

    There also is the following coarse estimates on length (Theorem 5.2 of \cite{BHM25}).

    \begin{theorem}\label{thm:quasiisom lengths}
        In the setting of Theorem~\ref{thm:quasiisom}, let $(\rho_z)_z$ be the associated family grafted Hitchin representations.
        Then there is $C_\sigma'$ only depending on $\sigma$ such that for any $\gamma\in\pi_1(S)$,
        $$
        \ell^{\mf F}(\rho_z(\gamma)) \geq \frac{L+1}{C'_\sigma}\iota(\gamma,\gamma^*).
        $$
        Moreover, recalling that $z$ is the datum of a vector $z_e\in\mf a$ for each component $e\subset\gamma^*$, then~$C'_\sigma$ may be chosen so that if  $z_e\in \ker(\alpha_0)$ for any $e$ then
        $$
        \ell^{\mf F}(\rho_{z}(\gamma)) \geq  \left(1+\frac{C'_{\sigma}}{L+1}\right)^{-1}\ell_S(\gamma),
        $$
        where $\ell_S(\gamma)$ is the length of $\gamma $ in $S$.
    \end{theorem}

\section{Intersection in the Hitchin component}\label{sec:intersection}

This section contains an application of the main results of \cite{BHM25} to dynamical
properties of Hitchin grafting representations.
Recall from Section~\ref{sec:currents} 
the definition of the intersection number ${\bf I}(f_1,f_2)$ and the 
normalized intersection number ${\bf J}(f_1,f_2)$ for two H\"older continuous 
positive functions 
$f_1,f_2$ on $T^1S$. These numbers only depend on the cohomology classes of $f_1,f_2$. 
Thus by the results in Section~\ref{sec:hitchinrep}, for any 
two Hitchin representations $\rho_1,\rho_2:\pi_1(S)\to \PSL_d(\mathbb{R})$, we obtain intersection numbers
${\bf I}(\rho_1,\rho_2)$ and normalized intersection numbers ${\bf J}(\rho_1,\rho_2)$. 
We show

    \begin{theorem}\thlabel{intersectiontendstoinf}
    There exists a sequence $\rho_i$ of Hitchin representations such that
    ${\bf I}(\nu,\rho_i)\to \infty$ and ${\bf J}(\nu,\rho_i)\to \infty$
    for any Fuchsian representation $\nu$, and this divergence is uniform in $\nu$.
    \end{theorem}
 
The Hitchin representations which enter Theorem~\ref{intersectiontendstoinf} are Hitchin grafting 
representations. More precisely, let as before $\gamma$ be a simple closed geodesic on the hyperbolic surface~$S$.
This datum is used to construct for each 
$L>0$ a Hitchin representation $\rho_L$ obtained by Hitchin grafting along $\gamma$ of 
the Fuchsian representation defined by $S$, 
with cylinder height $L$. We do not specify the twisting number of the associated abstract
grafting datum as this does not play a role in our discussion, but we assume that 
$L\to \rho_L$ is a Hitchin grafting ray as introduced in Section~\ref{sec:Abstract grafting}. 

The proof of Theorem~\ref{intersectiontendstoinf} 
rests on statistical information on length averages, introduced in the next definition.
For its formulation, for a Hitchin representation $\rho$ put $R_\rho(T)=R_{\ell_\rho}(T)$ for all $T$,
where as before, $R_{\ell_\rho}(T)=\{\eta\in [\pi_1(S)]\mid \ell_\rho(\eta)\leq T\}$ and $\ell_\rho(\eta)$ is the 
Finsler translation length of $\rho(\eta)$. Moreover, $[\pi_1(S)]$ is the set of conjugacy classes of 
the fundamental group $\pi_1(S)$ of $S$. 
	
	\begin{definition}\thlabel{defFullDensity}
		Let $\rho$ be a Hitchin representation and $A$ a subset of $[\pi_1(S)]$. We say that $A$ is a 
  \emph{full density} set 
  for $\rho$ if 
  \[\liminf_{T\to+\infty}\frac{R_\rho(T)\cap A}{R_\rho(T)}=1.\] 
  If $\mathcal P$ is an assertion on $[\pi_1(S)]$, we say that a \emph{typical geodesic satisfies 
  $\mathcal P$} if the set $\{\gamma\in[\pi_1(S)] \mid  \gamma\text{ satisfies }\mathcal P\}$ is a full density set for $\rho$.
	\end{definition}

The following statement can be thought of as a statistical version of the duality between 
length and intersection for hyperbolic metrics on surfaces. 
Recall from Section~\ref{sec:geodescicurrents} the definition of the intersection form
$\iota:{\mc C}(S)\times {\mc C}(S)\to [0,\infty)$.

	\begin{proposition}\label{typicalIntersection}
 Let $\rho$ be a hyperbolic metric on $S$, and let 
 $\alpha\subset S$ be a closed geodesic. For any $\epsilon>0$, for a typical geodesic $\gamma$, we have 
 \[\left|\iota(\gamma,\alpha) - \frac{1}{-4\pi^2 \chi(S)}\ell_\rho(\gamma)\ell_\rho(\alpha)\right|<\epsilon\ell_\rho(\gamma).\]
	\end{proposition}
\begin{proof} 
The Borel measures 
\[\mu_T=\frac{1}{\myhash  R_\rho(T)} \sum_{\ell_\rho (\gamma)\leq T} {\rm Leb}_\gamma\]
converge weakly as $T\to \infty$ to the normalized Lebesgue Liouville measure 
$\lambda_0$ on $T^1S$ (see~\cite{Mar04}). 

Let $\lambda\in {\mc C}(S)$ be the (unnormalized) Liouville \emph{current} of $\rho$, the current defined
by the Lebesgue Liouville measure on $T^1S$, and for each $T$ let $\hat \mu_T$ be the current defined
by $\mu_T$. 
Let 
$\alpha$ be a closed geodesic on $S$. As $\iota(\alpha,\lambda)=\ell_\rho(\alpha)$
(see Section~\ref{sec:geodescicurrents}), 
by continuity of the intersection form 
$\iota$ for the weak topology on currents, 
we know that 
\[\iota(\hat \mu_T,\alpha)\underset{T\to\infty}{\longrightarrow} \frac{1}{-4\pi^2\chi(S)}\ell_\rho(\alpha).\] 
Note to this end that the total volume of 
$T^1S$ with respect to the Lebesgue Liouville current equals 
$-4\pi^2 \chi(S)$.

Put $\kappa=\frac{1}{-4\pi^2 \chi(S)}$ and let $\epsilon >0$.
To show that the geodesics $\gamma$ with 
\[\vert \iota(\gamma,\alpha)-\kappa \ell_\rho(\gamma)\ell_\rho(\alpha)\vert <\epsilon \kappa\ell_\rho(\gamma)\]
are typical we argue as follows. For $T>0$ let 
\[{\mathcal A}(T)=\{\gamma\mid \ell_\rho(\gamma)\leq T, \iota(\gamma,\alpha)\geq 
(1+\epsilon)\kappa \ell_\rho(\gamma)\ell_\rho(\alpha)\}.\]
We claim that $\frac{\myhash  {\mathcal A}(T)}{\myhash  R_\rho(T)}\to 0$ $(T\to \infty)$.

To see this assume otherwise. By passing to a subsequence, we may assume that the 
measures $\nu_T=\frac{1}{\myhash  R_\rho(T)}\sum_{\gamma\in {\mathcal A}(T)}{\rm Leb}_\gamma$
converge weakly to a nontrivial $\Phi^t$-invariant 
measure~$\nu$. By construction, the measure $\nu$ is absolutely 
continuous with respect to the Lebesgue Liouville measure $\lambda$. It defines a current 
$\hat \nu$ which satisfies 
\begin{equation}\label{intersectioncontrol}
\iota(\hat \nu , \alpha)/\nu(T^1S)\geq (1+\epsilon) \kappa \ell_\rho(\alpha).\end{equation}
But $\lambda$ is ergodic under the action of $\Phi^t$ and hence as $\nu$ is absolutely continous
with respect to $\lambda$, it is a positive constant multiple of $\lambda$. 
This contradicts the inequality (\ref{intersectioncontrol}) and equation (\ref{iota}). 

In the same way we conclude that 
$\frac{\myhash  {\mathcal B}(T)}{\myhash  R_\rho(T)}\to 0$  as $T\to \infty$ where 
\[{\mathcal B}(T)=\{\gamma\mid \ell_\rho(\gamma)\leq T, \iota(\gamma,\alpha)\leq 
(1-\epsilon)\kappa \ell_\rho(\gamma)\ell_\rho(\alpha)\}.\]
Since $\epsilon >0$ was arbitrary, this shows the proposition.
\end{proof}

	
	Let $X$ be a hyperbolic metric on $S$ and let $c$ be a non-separating simple closed geodesic on $X$ of length 
 $\ell >0$. 
For $L\geq 0$ denote by $\rho_L$ a representation obtained by Hitchin grafting of $X$ on $c$ of height $L$. 
 Our goal is to 
 estimate for a hyperbolic metric $Y$ on $S$ 
 the quantities ${\bf I}(Y,\rho_L)$ and ${\bf J}(Y,\rho_L)$ as $L\to \infty$. 

	
 
\begin{proof}[Proof of Theorem~\ref{intersectiontendstoinf}]
Let $X\in {\mathcal T}(S)$ be the marked hyperbolic 
metric which is the basepoint for the Hitchin grafting ray. 
According to the length control as formulated in Theorem~\ref{thm:quasiisom lengths}, 
for every $\epsilon >0$   there exist $C_\sigma >0$ depending on the hyperbolic length $\sigma$ 
of the simple closed curve $c$ such that we have 
  \begin{equation}\label{quasihitagain}
  \ell_{\rho_L}(\gamma)\geq \max\left\{C_\sigma L\iota(\gamma,c),\frac{L}{L+C_\sigma^{-1}} \ell_X(\gamma)\right\} \end{equation} 
  where we use the notations of Theorem~\ref{thm:quasiisom lengths}, lengths 
  in $\X$ are measured with respect to 
  an admissible Finsler metric, and $\ell_X$ denotes the length for the hyperbolic metric $X$.

Let $m>0$ be a fixed number. Our goal is to find a number $L>0$ so that
\[{\bf J}(Y,\rho_L)\geq m\]
for every $Y\in {\mathcal T}(S)$ where as before, 
${\mathcal T}(S)$ denotes the Teichm\"uller 
space of marked hyperbolic metric on $S$.

By Theorem 12 of~\cite{Bo88}, the map which associates to a marked hyperbolic 
metric on $S$ its Liouville current is a proper topological embedding.
More precisely, for the given number $m>0$, there exists a compact ball $B$ 
about $X$ in ${\mathcal T}(S)$ such that
$\iota(\lambda_X,\lambda_Y)\geq m$ for all marked hyperbolic 
metrics $Y\in {\mathcal T}(S)-B$, where 
$\lambda_X,\lambda_Y$ are the currents 
defined by the normalized Lebesgue Liouville measures. 
Note that this is symmetric in $X,Y$. 
Furthermore, we have $\iota(\lambda_Y,\lambda_X)=
{\bf J}(Y,X)$. We refer to p.152-153 in~\cite{Bo88} for details on these facts.

By the estimate (\ref{quasihitagain}), for any $\epsilon>0$ 
and all sufficiently large $L\geq 0$ depending on $\epsilon$, 
say for all $L\geq L(\epsilon)$, we have
\[\ell_{\rho_L}(\gamma)\geq (1-\epsilon)\ell_X(\gamma).\] 
Thus by possibly increasing 
the ball $B$ we may assume that 
${\bf J}(Y,\rho_L)\geq m$ for all $L\geq L_0$ and all 
$Y\not\in B$.

We are left with showing that by possibly increasing $L_0$, we also 
have ${\bf J}(Y,\rho_L)\geq m$ for all $Y\in B$. However, this 
follows once more from the estimate (\ref{quasihitagain}). Namely, 
let $Y\in B$. 
By Proposition~\ref{typicalIntersection}, we know that there exists a constant 
$\kappa >0$ such that 
\[\iota(\gamma,c)\geq \kappa (1-\epsilon) \ell_Y(\gamma) \ell_Y(c)\] 
for any geodesic $\gamma$ which is typical for $Y$.

On the other hand, by compactness of $B$, there exists a constant $\sigma >0$ such that 
$\ell_Y(c)\geq \sigma$ for every $Y\in B$. Then for 
a geodesic $\gamma$ which is typical 
for $Y$, we have $\ell_Y(\gamma)\leq \frac{1}{\kappa\sigma (1-\epsilon)}\iota(\gamma,c)$.
Thus for $L> m/\kappa\sigma (1-\epsilon) C_\sigma $ it holds 
\[\ell_{\rho_L}(\gamma)/\ell_Y(\gamma) \geq \kappa \sigma (1-\epsilon)C_\sigma L  \geq 
m \]
which is what we wanted to show.
Together with the definition, it shows that ${\bf I}(\nu,\rho_L)\to \infty$ for 
every Fuchsian representation $\nu$.

To show that we also have ${\bf J}(\nu,\rho_L)\to \infty$ for all Fuchsian representations
it suffices to observe that the entropy of $\rho_L$ is bounded from below by a universal positive
constant. To see that this is the case, recall that for each $L$, the restriction of the representation~$\rho_L$ to the free subgroup $\Lambda$ of $\pi_1(S)$ of all based loops which do not cross through 
$c$ does not depend on $L$. In particular, the image of $\Lambda$ under $\rho_L$ stabilizes
a totally geodesic hyperbolic plane in $\X$. As a consequence, for each $L$ the entropy of 
$\rho_L$ is not smaller than the entropy of the geodesic flow on the bordered surface $S-c$, which 
is positive as $S-c$ is a hyperbolic surface with geodesic boundary. Together with the control on 
${\bf I}(\nu,\rho_L)$ established in the beginning of this proof, this implies that
${\bf J}(\nu,\rho_L)\to \infty$ $(L\to \infty)$ for any Fuchsian representation $\nu$.
\end{proof}

\section{Upper bound on the derivatives of length functions}\label{sec:ehresmann}

In this section we show how to control the first and second derivatives
of the Finsler length for paths of
Hitchin grafting representations. The section is divided into four subsections.

\subsection{Derivative bounds for lengths of closed geodesics}\label{sec:derivativebound}

Recall that we have fixed a marked hyperbolic surface $(S,h)$,
that is, a point in the Teichm\"uller space ${\mathcal T}(S)$, 
 and a multicurve $\gamma^*=\gamma^*_{1}\cup\dots\cup\gamma^*_{N}\subset S$.
In this section we study grafted representations of $\pi_1(S)$ into $G=\SL_d(\R)$ via the irreducible representation 
$\tau:\SL_2(\R)\to \SL_d(\R)$.
Recall that a bending parameter $z\in \mf a^N$ above $\gamma^*$ is the datum of 
an element $z_i\in \mf a$ for each component $\gamma^*_{i}$ of $\gamma^*$.
Given such a bending parameter $z$, we can construct a conjugacy class of Hitchin representations 
$[\rho_{h,z}]=[\rho_{z}]:\pi_1(S)\to \SL_d(\R)$. Here
$\rho_{h,0}$ is just the image of $h\in {\mathcal T}(S)$ under the representation $\tau$.

The following proposition is the main technical ingredient 
toward a control of the pressure length of suitably chosen paths in ${\rm Hit}(S)$, namely
estimates on the derivatives of the map $z\mapsto \lambda\circ\rho_{h,z}(\gamma)$ where $\gamma\in\pi_1(S)$, and 
$\lambda:G\to \overline{\mf a^+}$ is the Jordan projection (see Section~\ref{sec:lietheoryI}). 
To this end recall from Section \ref{sec:currents} that for any $z_0\in {\mf a}^N$ and any $\gamma\in \pi_1(S)$, the differential 
$d(\lambda \circ \rho_{z}(\gamma))_{z=z_0}$ of the Jordan projections for the deformations of $\rho_{z_0}$ obtained by 
grafting along $\gamma^*$ is defined and can be thought of as a linear map $\mathfrak{a}^N\to \mathfrak{a}$. 
If we equip $\mf a$ with the norm $\Vert \,\Vert$ 
obtained from the Killing form on ${\mf sl}_d(\mathbb{R})$, then this linear map has an operator norm 
which we denote by $\Vert d(\lambda \circ \rho_{z}(\gamma))\vert_{z=z_0}\Vert$. Note that the hyperbolic 
metric $h$ is left fixed in this construction. 
Similarly, the Hessian $d^2(\lambda \circ \rho_z(\gamma))\vert_{z=z_0}$ 
evaluated on the subspace of ${\rm Hit}(S)$ defined by grafting along $\gamma^*$ 
can be thought of as an 
$\mathfrak{a}$-valued bilinear form on ${\mf a}^N$ which has an operator norm. 


\begin{proposition}\label{cor:firstandsecond}
    For any $\sigma>0$ there exists $C_\sigma>0$ such that for any $h\in\mc T(S)$ giving length $\leq \sigma$ to each component of $\gamma^*$, for any bending parameter $z_0\in\mf a^N$ along $\gamma^*$ and  for any $\gamma\in \pi_1(S)$, we have
    $$
    \left\Vert d(\lambda\circ\rho_{z}(\gamma))\vert_{z=z_0}\right\Vert,\ 
    \left\Vert d^2(\lambda\circ\rho_{z}(\gamma))\vert_{z=z_0}\right\Vert \leq C_\sigma \iota(\gamma,\gamma^*).
    $$
\end{proposition}


The idea for the proof of the proposition is the following.
Given a family of Hitchin grafting representations $(\rho_z)_{z\in\mf a}$ based at a Fuchsian representation $\rho_0$, and $\gamma\in\pi_1(S)$, we decompose $\gamma$ into an admissible path for the hyperbolic structure corresponding to $\rho_0$.
Say the admissible path travels in $S-\gamma^*$ for the time $t_1$, then meets orthogonally some component $\gamma^*_{i_1}$ of $\gamma^*$, then travels along it for some time $s_1$, then departs from it orthogonally and travels for some time $t_2$ in $S-\gamma^*$... etc.
This gives us the following formula for the conjugacy class of the holonomy $\rho_0(\gamma)$:
\begin{equation*}
    \rho_{0}(\gamma) \sim a'_{t_1}a_{s_1}a'_{t_2}\cdots a'_{t_k}a_{s_k},
\end{equation*}
where $k$ is the intersection number of $\gamma$ with $\gamma^*$ and where we use the notations interoduced in Definition \ref{def:admissible in G}.

The grafting operation deforms the above expression in a very explicit way, namely
\begin{equation}\label{producta}
    \rho_{z}(\gamma) \sim a'_{t_1}a_{s_1}  \exp({r_1(z)})a'_{t_2}\cdots a'_{t_k}a_{s_k} \exp({r_k(z)}),
\end{equation}
where $r_n(z)$ is either the $i$-th coordinate $z_{i_n}$ (where $\gamma^*_{i_n}$ is the $n$-th component of $\gamma^*$ crossed by $\gamma$), or $\iota(-z_{i_n})$ where $\iota:\mf a\to\mf a$ is the Cartan involution, depending on whether we cross $\gamma^*_{i_n}$ from left to right or from right to left.
Taking $\mf a$ as the vector space of trace free diagonal matrices, this involution 
permutes the diagonal entries as follows: $\mr{diag}(x_1,\dots,x_d)\mapsto \mr{diag}(x_d,\dots,x_1)$.

To estimate derivatives of the Jordan projection of the above product of $3k$ matrices we simply differentiate 
the expression in the equation (\ref{producta}), which is simply a product of matrices, using the product rule for differentiation. 
This yields the following formula for the first order derivative, where we put $X_n=dr_n(z)\vert_{z=z_0}$. 
\begin{align*}
   d (\lambda \circ \rho_z(\gamma))\vert_{z=z_0} \\
    =  \sum_{n=1}^k 
    d\lambda (a'_{t_1}a_{s_1}\exp( {r_1(z_0)})\cdots a'_{t_n}a_{s_n} 
    &\exp( {r_n(z_0)})X_n\cdots a'_{t_k}a_{s_k}\exp({r_k(z_0)}))\\
 = \sum_{n=1}^k d\lambda(a'_{t_{n+1}}a_{s_{n+1}}\exp({r_{n+1}(z_0)})\cdots  &a'_{t_1}a_{s_1}\exp({r_1(z_0)})\cdots 
 a'_{t_n}a_{s_n}\exp( {r_n(z_0)})X_n)
\end{align*}
where the second equation in this formula uses the fact that the Jordan projection is invariant under conjugation. 

The norm of this differential will be controlled using the fact that the matrix 
$$a'_{t_{n+1}}a_{s_{n+1}}\exp( {r_{n+1}(z_0)})\cdots a'_{t_k}a_{s_k}
\exp({r_k(z_0)})a'_{t_1}a_{s_1}\exp( {r_1(z_0)})\cdots a'_{t_n}a_{s_n}\exp({r_n(z_0)})$$
 is totally positive in a quantitative way. In Section~\ref{totpos implies loxo} we shall establish that it 
is loxodromic in a quantitative way, and in Section~\ref{sec:derivative loxodromic} 
we will prove general estimates for derivatives of the form $d\lambda(g \cdot X)$ when $g$ is loxodromic in a quantitative way and 
$X\in \mf a$. 
These results will be obtained by first estimating $d\lambda_1(g \cdot X )$ 
with $g$ quantitatively proximal (Section~\ref{sec:derivative proximal}) and then applying this to exterior powers of loxodromic elements.
The computations for the second order derivative are more involved but can also worked out using positivity.

Proposition~\ref{cor:firstandsecond} will be useful to bound the pressure length of Hitchin grafting paths in the Hitchin component, 
namely paths of the form $(\rho_{h,tz})=(\rho_{tz})_{t\geq 0}$, using the previous notations, that is, 
$h\in {\mc T}(S)$ is a fixed marked hyperbolic metric and 
$z\in \mf a^N$ is a fixed grafting parameter. 
It will also be used to control the pressure lengths of other kinds of paths, where we fix the grafting parameter $z$ 
and let the hyperbolic metric vary instead.
More precisely, we will deform the hyperbolic metric by \emph{shearing}, 
which is in fact a particular case of bending, and this allows to give a unified treatment for both types of deformations.

Proposition~\ref{cor:firstandsecond} has the following consequence. For its formulation, it is useful to keep in mind that grafting along 
a multicurve $\gamma^*$ commutes with modifying the hyperbolic metric in the complement of $\gamma^*$. 
More precisely, fix a basepoint $x_i$ on each component of $\gamma^*$. If $S_1\subset S$ is a component of the complement of 
$\gamma^*$ then using the basepoint as a marked point, the Teichm\"uller space ${\mc T}(S_1)$ 
of $S_1$ is the space of marked hyperbolic metrics on $S_1$ with geodesic boundary of fixed length and one marked point in 
each boundary component. Each choice of a hyperbolic metric $h$ on $S$ determines an embedding of ${\cal T}(S_1)$ into
${\cal T}(S)$ (where the boundary lengths depend on $h$) by first marking a point on each 
geodesic which defines a boundary component of $S_1$, 
cutting $S$ open along these $h$-geodesics and gluing a marked metric $h_1\in {\cal T}(S_1)$ to $S\setminus S_1$ matching marked
points. By the definition of grafting, this operation commutes with grafting along $\gamma^*$. 

Observe that one way to deform the marked hyperbolic metric on $S_1$ is to shear (or twist) along a closed geodesic 
entirely contained in $S_1$. As twisting is a grafting operation along a grafting parameter contained in the one-dimensional
subspace of $\mf a$ tangent to the image of the representation $\tau$, we obtain as a corollary of Proposition \ref{cor:firstandsecond}
the following result. In its formulation, a constant speed shearing path along a geodesic multicurve $\eta$
is a path of marked hyperbolic metrics which consists in
cutting $S$ open along $\eta$ and gluing back with a rotation whose rotation speed is constant one along the path, where
the speed is the absolute value of the 
derivative of the signed length of the shearing deformation where length is measured with respect to the 
length element of the geodesic multicurve. Note that this makes sense as the length element of a shearing 
multicurve is constant along a shearing path.

\begin{corollary}\label{cor:estimate derivative length on shearings}
Fix a connected component $S_1\subset S-\gamma^*$.
Let $(h_t)_{t\in[0,1]}$ be a smooth path of hyperbolic metrics 
obtained from $h_0$ by constant speed 
shearing along a multicurve $\eta\subset S_1$  
(we allow shearing along $\gamma^*$ and different speed of shearings along the components, including zero speed).

Then there is a constant $C>0$ only only depending on 
an upper bound on the lengths of the multicurve $\gamma^*\cup \eta$ and an upper bound on the shearing speeds along
the components of $\eta$ 
such that for all $z,t,\gamma$ we have
    $$
    \left\Vert\frac{d}{dt}\lambda\circ\rho_{h_t,z}(\gamma)\right\Vert,\ \left\Vert\frac{d^2}{dt^2}\lambda\circ\rho_{h_t,z}(\gamma)\right\Vert \leq C\ell_{h_t}(\gamma\cap S_1),   
    $$
    where $\ell_{h_t}(\gamma\cap S_1)$ is the $h_t$-length of the subarcs of $\gamma$ contained in $S_1$.
\end{corollary}
\begin{proof}
Suppose $(h_t)_{t\in[0,1]}$ is a path of hyperbolic metrics obtained by constant
speed shearing $h_0$ along a multicurve $\eta\subset S_1$.
Shearing is a special kind of grafting or bending, with grafting 
parameter collinear to $d\tau\!\left(\begin{smallmatrix}1&0\\0&-1\end{smallmatrix}\right)$. Note that 
$\beta=\gamma^*\cup \eta$ is a multicurve. 

For a fixed grafting parameter $z\in {\mf a}^N$ and $t\in [0,1]$ let $\rho_{z,t}$ be the representation obtained
by grafting $h_t$ along $\gamma^*$ with grafting parameter $z$. Then $\rho_{0,t}$ is just the 
Fuchsian representation defined by the marked hyperbolic metric $h_t$.
Since $S_1$ is a component of $S\setminus \gamma^*$, it follows from Proposition \ref{cor:firstandsecond} 
that there exists a number $C>0$ only 
depending on an upper bound for the $h_t$-lengths of the components of 
$\beta$, which is independent of $t$, and an upper bound on the shearing speed, 
such that for each grafting parameter $z$ along $\gamma^*$ and every 
$\gamma\in \pi_1(S)$, we have 
%
%
$$
    \left\Vert\frac{d}{dt}\lambda\circ\rho_{h_t,z}(\gamma)\right\Vert,\ \left\Vert\frac{d^2}{dt^2}\lambda\circ\rho_{h_t,z}(\gamma)\right\Vert \leq C\iota(\gamma,\beta).
    $$

It remains to check that $\iota(\gamma,\beta)\leq C' \ell_{h_t}(\gamma\cap S_1)$ for some constant $C'$ 
only depending on the upper bound of the lenghts of the components of $\beta$.
For this one just takes $C'$ equal to half the infimum among all $t$'s of the shortest $h_t$-distance from one component of $\beta$ to another,
which is uniformly bounded from below along the path by the collar lemma. 
\end{proof}

\subsection{Derivatives of lengths of proximal transformations}\label{sec:derivative proximal}

For a $(d,d)$-matrix $A$ and any $i\leq d$ we denote by 
$\lambda_i(A)\in [-\infty,+\infty)$ the logarithm of the absolute value of the $i$-th
eigenvalue where the eigenvalues are ordered in nonincreasing absolute values.
We use the convention $\log0=-\infty$.
Of course the derivatives of $\lambda_i$ at a point $A$ where $\lambda_i(A)=-\infty$ do not make sense.
In the following, every time we need to compute such derivatives, we will always make sure that $\lambda_i(A)>-\infty$.
We also write 
\[\lambda(A)=(\lambda_1(A),\lambda_2(A),\dots)\in [-\infty,+\infty)^d,\]
which generalizes the Jordan projection when $A\in\SL_d(\R)$.

Recall that  $A$ is \emph{proximal} if $\lambda_{1}(A)>\lambda_{2}(A)$, which means $A$ has a real eigenvalue of multiplicity one (called the dominant eigenvalue) whose absolute value is strictly greater than that of any other eigenvalue.
In particular, we have $\lambda_1(A)>-\infty$.
The eigenline associated to the dominant eigenvalue is called the \emph{attracting eigenline}.
If $A$ is of maximal rank, then $A$ acts on $\mathbb{R}P^{d-1}$, and the attracting eigenline 
is an attracting point for the action of $A$ on the projective space.
The complementary $A$-invariant hyperplane (the sum of the remaining generalized eigenspaces)
is called the \emph{repelling hyperplane}.


We will need a quantitative version of proximality.
For this we endow $\R^d$ with its usual inner product coming from its canonical basis.
In the following, all angles come from this fixed inner product.

For any $0<\theta<\pi/2$ and $\omega\in(0,\infty]$ we denote by $P_{\omega,\theta}$ the set of nonzero 
proximal (possibly noninvertible) matrices $A$ with the following two properties.
\begin{enumerate}
\item[(a)]$\lambda_1(A)- \lambda_2(A)\geq\omega$ (allowing $\lambda_2(A)=-\infty$),
\item[(b)]the attracting eigenline of $A$ makes an angle $\geq \theta$ with the repelling hyperplane.
\end{enumerate}
We denote by $DP_{\omega,\theta}\subset P_{\omega,\theta}\times P_{\omega,\theta}$ the set of pairs $(A_0,A_1)$ such that $A_{i}$'s attracting eigenline forms an angle $\geq \theta$ with  $A_{1-i}$'s repelling hyperplane, and such that the product $A_{i}A_{1-i}$ is in $P_{\omega,\theta}$.
Note that if $A\in P_{\omega,\theta}$ then $(A,A)\in DP_{\omega,\theta}$.

In the next lemma, the norm 
$\Vert X\Vert$ of an element in the Lie algebra ${\mf gl}_d(\R)={\mf sl}_d(\mathbb{R})\oplus \mathbb{R}$
is the norm
induced by the Killing form of $\mathfrak{sl}_d(\R)$ and the choice
of a Cartan involution. Furthermore as before, 
$\exp:\mf{gl}_d(\R)\to {\rm GL}_d(\R)$ denotes
the exponential map. By left translation, this exponential map defines 
for any $A\in {\rm GL}_d(\R)$ a map $A\cdot \exp:
X\in {\mf gl}_d(\R)\to A\cdot \exp(X)\in {\rm GL}_d(\R)$.
In the second statement of the lemma below, the Hessian is taken 
of a function defined on the direct sum ${\mf gl}_d(\R)\oplus {\mf gl}_d(\R)$, and 
the norm is the operator norm. 


\begin{lemma}\label{lem:proximal length estimate}
    For all $0<\theta<\pi/2$ and $\omega_0>0$, there exists $C=C_{\omega_0,\theta}>0$ such that for any $\omega>
    \omega_0$ and any $(A,B)\in DP_{\omega,\theta}$, the following is satisfied.
    \begin{enumerate}
    \item $\Vert d\lambda_1(A\cdot \exp)\vert_0\Vert \leq C_{\omega_0,\theta}$ and
      $\Vert d^2\lambda_1(A\cdot \exp)\vert_0\Vert \leq C_{\omega_0,\theta}$;
        \item For $(X,Y)\in {\mf gl}_d(\R)\oplus {\mf gl}_d(\R)$ it holds
        $\Vert d^2\lambda_1(A\cdot \exp \cdot B \cdot \exp)\vert_{(0,0)}(X,Y)\Vert \leq C_{\omega_0,\theta} e^{-\omega}$.
    \end{enumerate}
\end{lemma}
\begin{proof}
The idea of the proof is to use a compactness argument, by restricting without loss of generality to a compact subset of $P_{\omega,\theta}$. 
Namely;
let $P_{\omega,\theta}'\subset P_{\omega,\theta}$ be the set of proximal matrices $A\in P_{\omega,\theta}$ with spectral radius $e^{\lambda_1(A)}$ equal to $1$, and $DP'_{\omega,\theta}=DP_{\omega,\theta}\cap (P'_{\omega,\theta}\times P'_{\omega,\theta})$.

Let us check that restricting to this compact subset does no harm: For all $(A,B)\in DP_{\omega,\theta}$, if $A'=e^{-\lambda_1(A)}A$ and $B'=e^{-\lambda_1(B)}B$ then $ \lambda_1(A \cdot \exp(X)) = \lambda_1(A)+\lambda_1(A' \cdot \exp(X))$ and
$ \lambda_1(A\cdot \exp(X)\cdot B\cdot \exp(Y)) = 
\lambda_1(A)+\lambda_1(B)+\lambda_1(A' \cdot \exp(X) \cdot B' \cdot \exp(Y))$ for all $X,Y$.
Thus the derivatives we need to estimate are the same for $(A,B)$ and for $(A',B')$.

The set $P^\prime_{\omega,\theta}$ is a compact subset of the space 
$\mf{gl}_d(\R)=\mathbb{R}^{d^2}$ of $(d,d)$-matrices, and the 
restriction of the function $\lambda_1$ to an open neighborhood of the compact set $P_{\omega,\theta}^\prime$ 
in $\R^{d^2}$ is smooth. and 
Hence the differential $d\lambda_1$ is smooth section of the restriction of 
the cotangent bundle of $\R^{d^2}$ to 
$P_{\omega,\theta}^\prime$, 
and the Hessian is a smooth section of the bundle of symmetric bilinear forms on $\R^{d^2}$. 
Thus there exists a constant $K>0$ so that
for any $A\in P_{\omega,\theta}^\prime$ we have 
\begin{equation}\label{eq:estone}
       \norm{d\lambda_1(A\cdot \exp)\vert_0},\ \norm{d^2\lambda_1(A\cdot \exp)_0)}  \leq K.
\end{equation}
This shows the first part of the lemma.

To show the second part of the lemma note that 
 since $DP_{\omega,\theta}^\prime$ is compact, up to increasing $K$ we may assume that 
$\norm{d^2\lambda_1(A \cdot \exp \cdot B\cdot \exp )\vert_{(0,0)}}\leq K$. Furthermore, using smoothness of the 
map which associates to a pair of points $(A,B)$ in a neighborhood of the compact set
$DP_{\omega_0,\theta}$ the 
Hessian $d^2\lambda_1(A\cdot \exp \cdot B\cdot \exp )\vert_{(0,0)}$, viewed as a bilinear form on the direct sum
$\mf{gl}_d(\R)\oplus {\mf gl}_d(\R)$ depending on $(A,B)$, 
 up to increasing once more the control constant $K$ we obtain that 
    \begin{equation}\label{eq:est joint deriv}
        \norm{d^2 \lambda_1(A\cdot \exp \cdot B\cdot \exp)_{(0,0)}-
        d^2\lambda_1(C\cdot \exp \cdot D \cdot \exp)_{(0,0)}}  
         \leq K(||A-C||+||B-D||). 
    \end{equation}

This estimate allows to proceed by first showing the second part of the lemma 
    in the case $(A,B)\in DP'_{\infty,\theta}$,  that is, if all 
  nondominant eigenvalues are zero, equivalently if \ $A$ and $B$ are rank-one projectors.
    We will see that in this case 
    $d^2\lambda_1(A\cdot \exp  \cdot B \cdot \exp)\vert_{(0,0)}=0$, and we will be able to extend to the general case using \eqref{eq:est joint deriv}.

By the definition of $DP_{\omega,\theta}^\prime$, if $(A,B)\in DP_{\omega,\theta}^\prime$ and ${\mr rk}(A)={\rm rk}(B)=1$ 
then ${\rm ker}(AB)={\rm ker}(B)$, and 
for any $C$, we have $\ker(AB)=\ker(B)\subset\ker(ACB)$. Thus for any $C$ 
there is a number $\alpha_C\in \mathbb{R}$  
such that $ACB=\alpha_CAB$.
    Then $\lambda_1(ACB)=\log|\alpha_C|+\lambda_1(AB)$ and consequently  $\alpha_C=\pm e^{\lambda_1(ACB)-\lambda_1(AB)}$.
    Similarly, we have $BCA=\beta_{C}BA$ where $\beta_{C}=\pm e^{\lambda_1(BCA)-\lambda_1(BA)}$. 
    
    Recall that $\lambda_1(CD)=\lambda_1(DC)$ for all matrices.
    Using that $A^{2}=A$ and $B^{2}=B$ we note that for all $X,Y\in {\mf gl}_d(\R)$ we have
    \begin{align*}
    \lambda_1(A \cdot \exp(X) \cdot B\cdot \exp(Y)) &=
    \lambda_1(AA \cdot \exp(X) \cdot BB \cdot \exp(Y)) \\
   & =\lambda_1(A\cdot \exp(X) \cdot BB \cdot \exp(Y)\cdot A).
    \end{align*}
    Now $A\exp(X)B=\alpha_{\exp(X)}AB$ and $B\exp(Y)A=\beta_{\exp(Y)}BA$, so their product is 
$$(A\exp (X)B)(B\exp(Y)A)=\alpha_{\exp(X)}\beta_{\exp(Y)}ABBA=\alpha_{\exp(X)}\beta_{\exp(Y)}ABA$$
and hence 
\[\lambda_1(A\cdot \exp(X)\cdot B\cdot \exp(Y))=\log\vert \alpha_{\exp(X)}\vert +\log \vert \beta_{\exp(Y)}\vert +
\lambda_1(ABA).\] 
%
    In particular, for the function $\rho_{A,B}:{\mf gl}_d(\R)\oplus {\mf gl}_d(\R)\to \mathbb{R}$ defined by
    $\rho_{A,B}(X,Y)=
     \lambda_1(A \cdot \exp(X) \cdot B\cdot \exp(Y))$ we have $d^2\rho_{A,B}(X,Y)_{(0,0)}=0$. Equivalently, the 
     splitting ${\mf gl}_d(\R)\oplus {\mf gl}_d(\R)$ is orthogonal for the Hessian of $\rho_{A,B}$.

    Now take arbitrary $(A,B)\in DP'_{\omega,\theta}$.
    Let $C\in P'_{\omega,\theta}$ (resp.\ $D$) be the rank-one projector with kernel the repelling hyperplane of $A$ (resp.\ $B$) and with image the attracting line of $A$ (resp.\ $B$).
By \eqref{eq:est joint deriv}, combined with  $d^2\rho_{C,D} (X,Y)=0$ we have
    \begin{equation*}
    \norm{d^2\lambda_1(A \cdot \exp(X)\cdot B\cdot \exp(Y))} \leq K(||A-C||+||B-D||).
    \end{equation*}
    One can then check that there exists a constant $K'$ that depends on $\theta$ such that
    $$||A-C||\leq  K'e^{-\omega},$$
    and similarly for $B$ and $D$.
    This proves the second part of the lemma.
\end{proof}

\subsection{Derivatives of lengths of loxodromic transformations}\label{sec:derivative loxodromic}

We are going to deduce from the previous section an estimate for the derivatives of lengths of loxodromic transformations.
The argument is classical: it relies on the fact that a matrix $A\in \GL_d\R$ is loxodromic if and only if all its exterior products $\bwedge^kA\in\GL(\bwedge^k\R^d)$ are proximal for $1\leq k\leq d-1$.

Indeed, suppose $A\in\GL_d(\R)$ is loxodromic, i.e.\ $\lambda_{1}(A)>\lambda_{2}(A)>\dots>\lambda_{d}(A)$, which 
means $A$ is diagonalizable such that the eigenvalues have multiplicity $1$ and distinct absolute values.
Let $v_1,\dots,v_d$ be an eigenbasis for $A$ ordered by decreasing absolute values of eigenvalues.
Then for any $k$ the elements of the form $v_{i_1}\wedge\dots\wedge v_{i_k}$, where $1\leq i_1<\dots<i_k\leq d$, form an eigenbasis of $\bwedge^k\R^d$ for $\bwedge^kA$, such that the absolute value of the logarithm of the associated eigenvalue is $\lambda_{i_1}(A)+\dots+\lambda_{i_k}(A)$.

For any $0\leq k\leq d$ let $d_k$ denote the dimension of the exterior product $\bwedge^k\R^d$.
The canonical basis $e_1,\dots,e_d$ of $\R^d$ yields a natural basis of $\bwedge^k\R^d$ with elements $e_{i_1}\wedge\dots\wedge e_{i_k}$, where $1\leq i_1<\dots<i_k\leq d$.
This induces an identification of $\bwedge^k\R^d$ with $\R^{d_k}$ and an inner product on $\bwedge^k\R^d$.

For any transformation $X$ of $\bwedge^k\R^d$, seen as a matrix of size $d_k$, we add an upperscript $d_k$ to all quantities previously defined involving $X$ to specify the size of $X$.
For instance $\lambda^{d_k}_1(X),\dots,\lambda^{d_k}_{d_k}(X)$ are the logarithms of the absolute values of the eigenvalues of $X$.
We also denote by $P^{d_k}_{\omega,\theta}$ the set of proximal transformations of $\bwedge^k\R^d$ that satisfy the quantitative conditions of the previous section.

In particular, coming back to the computations on $A\in\GL_d\R$, we have
\begin{equation}\lambda_1^{d_k}(\bwedge^kA)=\lambda_{1}^d(A)+\dots+\lambda_{k}^d(A)\end{equation}
and
\begin{equation}\lambda_2^{d_k}(\bwedge^kA)=\lambda_{1}^d(A)+\dots+\lambda_{k-1}^d(A)+\lambda_{k+1}^d(A).\end{equation}
It is a well known fact that these formulas work for any matrix $A$ of size $d$, not necessarily invertible.
In particular, $\lambda_{k}^d(A)>-\infty$ if and only if $\lambda_1^{d_k}(\bwedge^kA)>-\infty$, and in this case
\begin{equation}\lambda^d_k(A)-\lambda_{k+1}^d(A)=\lambda_1^{d_k}(\bwedge^kA)-\lambda_2^{d_k}(\bwedge^kA).\end{equation}

For all $0<\theta<\pi$ and  $\omega>0$ we denote by $L^d_{\omega,\theta}$ the set of loxodromic invertible matrices $A$ of size $d$ such that $\bwedge^kA\in P^{d_k}_{\omega,\theta}$ for any $1\leq k\leq d-1$.
We also denote by $DL_{\omega,\theta}^d$ the set of pairs $(A,B)\in L_{\omega,\theta}^d\times L_{\omega,\theta}^d$ such that $(\bwedge^kA,\bwedge^kB)\in DP_{\omega,\theta}^{d_k}$ for any $1\leq k\leq d-1$. 

For a loxodromic matrix $A\in {\rm GL}_d(\R)$ put
\[\lambda(A)=(\lambda_1(A),\dots,\lambda_d(A))\in \R^{d}.\] 
The notations in the following lemma extend the notations in Lemma \ref{lem:proximal length estimate}.

\begin{lemma}\label{lem:lox length estimate}
    For all $0<\theta<\pi$ and $\omega_{0}>0$, there exists $C_{\omega_{0},\theta}>0$ such that for all $\omega>\omega_{0}$ and $(A,B)\in DL_{\omega,\theta}$ we have.
    \begin{enumerate}
        \item\label{item:lox length estimate} The differential and the Hessian at $X=0$ of the map 
        $X\mapsto \lambda(A\exp(X))$ are bounded above in norm by $C_{\omega_{0},\theta}$,
        \item\label{item:lox length estimate joint} 
        $
        \norm{d^2\lambda(A\cdot \exp(X) \cdot B \cdot \exp(Y))\vert_{(0,0)}} \leq C_{\omega_{0},\theta} e^{-\omega}.
        $
    \end{enumerate}
\end{lemma}
\begin{proof}
    By definition $\bwedge^kA\in P^{d_{k}}_{\omega,\theta}$ for any $k$.
    By Lemma~\ref{lem:proximal length estimate} we get a constant $C>0$, only depending on $\omega$, such that for any $k$ the first two derivatives at $X=0_{d_k}$ of $X\mapsto \lambda_1^{d_k}((\bwedge^kA)\cdot \exp(X))$ are bounded above by $C$.

    Since $\bwedge^k(A \cdot\exp(X))=(\bwedge^kA)\cdot \exp(\bwedge^kX)$ 
    and $X\mapsto \bwedge ^kX$ is linear, we deduce that  for any $k$ the first two derivatives at $X=0_{d}$ of 
    $$
    X\mapsto \lambda_1^{d_k}(\bwedge^k(A \cdot\exp(X)))=\lambda_{1}^d(A \cdot \exp(X))+\dots+\lambda_{k}^d(A\cdot \exp(X))
    $$
    are bounded above by some constant $C'$ only depending on $C$.

    This implies that  the first two derivatives at $X=0_{d}$ of 
    $$
    X\mapsto \lambda^d(A \cdot \exp(X))=(\lambda_1^d(A\cdot\exp(X)),\lambda_2^d(A\cdot \exp(X)),\dots,\lambda_d^d(A\cdot\exp(X)))
    $$
    are bounded above by some constant $C''$ only depending on $C'$.

    The second part \ref{item:lox length estimate joint} is obtained in exactly the same way, using the corresponding part of Lemma~\ref{lem:proximal length estimate}.
\end{proof}

\subsection{Totally positive matrices in $a'_\omega G_{\geq 0}$ are quantitatively loxodromic}\label{totpos implies loxo}

This section simply contains the following result that totally positive matrices in $a'_\omega G_{\geq 0}$ are quantitatively loxodromic.
It is probably well-known to experts, and was proved in the companion paper \cite{BHM25}, Proposition~4.12.

\begin{proposition}\label{fact:quantitative Pos are Lox}
    For any $g\in G_{>0}$ there exist $\omega>0$ and $0<\theta<\pi/2$ such that 
    \begin{enumerate}
        \item $g G_{\geq 0}\subset L_{\omega,\theta}$;
        \item for all $h_1,\dots,h_n\in gG_{\geq 0}$, denoting  $h=h_1\cdots h_n$ we have $\lambda_k(h)\geq \lambda_{k+1}(h) + n\omega$ for any $1\leq k \leq d-1$, namely $h\in L_{n\omega,\theta}$;
        \item for all $h,h'\in gG_{\geq 0}$,
        we have $(h,h')\in DL_{\omega,\theta}$.
    \end{enumerate}
\end{proposition}

\subsection{Proof of Propositions~\ref{cor:firstandsecond}}\label{sec:proof derivative estimates}

We fix $\gamma\in\pi_1(S)$ and decompose $\gamma$ into an admissible path for the hyperbolic structure corresponding to $\rho_0$.
As mentioned before, the admissible path travels in $S-\gamma^*$ for time $t_1>0$, then meet orthogonally some component $\gamma^*_{i_1}$ of $\gamma^*$, then travels along it for some time $s_1$ that can be positive or negative, then departs from it orthogonally and travels for some time $t_2>0$ in $S-\gamma^*$... etc.
This gives us the following formula for the conjugacy class of the holonomy $\rho_0(\gamma)$:
\begin{equation}\label{eq:conjclass before grafting}
    \rho_{0}(\gamma) \sim a'_{t_1}a_{s_1}a'_{t_2}\cdots a'_{t_k}a_{s_k},
\end{equation}
where $k$ is the intersection number of $\gamma$ with $\gamma^*$.

After grafting we obtain an admissible path in the characteristic surface, grafted with the parameter $z$, that first travels in a hyperbolic piece for time $t_1$ until it meets orthogonally the flat cylinder above $\gamma^*_{i_1}$, then it travels in the flat cylinder along a segment given by a vector of the form $s_1d\tau\hsl+r_1(z)\in\mf a$ where $r_1(z)$ is either the $i_1$-th coordinate $z_{i_1}$, or the image $\iota(-z_{i_1})$ under the Cartan involution $\iota$, depending whether we cross the component $\gamma_{i_1}^*\subset\gamma^*$ from left to right or from right to left.
Then we repeat these steps: we travel in a hyperbolic piece of the characteristic surface for time $t_2$, meet orthogonally a flat cylinder, travel along this cylinder...
In the end the grafting operation deforms the formula \eqref{eq:conjclass before grafting} in the following explicit way:
\begin{equation}\label{eq:conjclass}
    \rho_{z}(\gamma) \sim a'_{t_1}a_{s_1} \exp({r_1(z)})a'_{t_2}\cdots a'_{t_k}a_{s_k}\exp({r_k(z)}).
\end{equation}

We wish to estimate in norms the derivatives of its Jordan projection in direction of the grafting parameter $z$.
Let us first compute the general formula for the first and second derivatives of a map of the form
$$
f(v)=\lambda(A_1 \exp({f_1(v)})\cdots A_k \exp({f_k(v)})),
$$
where $v\to (f_1(v),\cdots f_k(v))\in {\mf  a}^k$ is a smooth path with 
$f_i(0)=0$ for any $i$.
To make the notations shorter, put  $A_i^j=A_i\cdots A_j$.
We have
\begin{equation}\label{eq:derive1}
    \frac{d}{dv}_{\vert v=0}f(v)=\sum_{i=1}^k d\lambda(A_1^i  \exp(\cdot )  A_{i+1}^k)_0 ( \frac{d}{dv}_{|v=0}f_i(v)),
\end{equation}
where the notation $d\lambda(A_1^i  \exp(\cdot)  A_{i+1}^k)\vert_{0}$ means that we take the differential of the map
$X\to \lambda(A_1^i  \exp(X)   A_{i+1}^k)$ at $X=0$. 

Similarly, 
\begin{equation}\label{eq:derive2}
    \begin{split}
        \frac{d^2}{dv^2}_{\vert v=0}f(v) 
        &= \sum_{i=1}^k d\lambda(A_1^i  \exp(\cdot) A_{i+1}^k)\vert_{0} (\frac{d^2}{dv^2}_{|v=0}f_i(v))\\
        &+ \sum_{i=1}^k d^2 \lambda(A_1^i \exp(\cdot) A_{i+1}^k)\vert_{0} \left(\frac{d}{dv}_{|v=0}f_i(v),\frac{d}{dv}_{|v=0}f_i(v)\right)\\
        + 2\sum_{1\leq i<j\leq k} d^2 \lambda &(A_1^i \exp(\cdot) A_{i+1}^j  
        \exp(\cdot)A_{j+1}^k)\vert_{0,0} \left(\frac{d}{dv}_{|v=0}f_i(v),\frac{d}{dv}_{|v=0}f_j(v)\right).
    \end{split}    
\end{equation}

Let us now start the computations, starting with the easiest one: According to \eqref{eq:conjclass} and \eqref{eq:derive1}, an upper bound for 
$\norm{d\lambda(\rho_{h_0,z}(\gamma))}$  is given by
\begin{align}\label{eq:I like to deriiiiive1}
 \sum_i & \norm{dr_i(z)\vert_{z=z_0}} \\
 \cdot
 &\norm{d\lambda\left( a'_{t_1(h_0)} \cdots a'_{t_i(h_0)} a_{s_i(h_0)} \exp({r_i(z_0)}) \exp(\cdot) a'_{t_{i+1}(h_0)}\cdots 
 \exp({r_k(z_0)}) \right)\vert_{0}}. \notag
\end{align}
The term $\norm{d r_i(z) \vert_{z=z_0} }$ is less than or equal to $1$ since $r_i:z\in \mf a \to r_i(z)$ 
can be thought of as either the identity map of $\mf a$ or the negative of the Cartan involution. 

Using that $\lambda$ is invariant under conjugacy, we can rewrite the term in the second line of equation (\ref{eq:I like to deriiiiive1}) as
\begin{equation}\label{eq:I like to deriiiiive2}
\norm{d\lambda\left( a'_\omega \underbrace{a'_{t_{i+1}(h_0)-\omega}\cdots \exp({r_k(z_0)})a'_{t_1(h_0)} \cdots a'_{t_i(h_0)} a_{s_i(h_0)} 
\exp({r_i(z_0)}}_{\in G_{\geq 0}} )\exp(\cdot)\vert_{0}  \right)}
\end{equation}
where $\omega>0$ is the collar size associated to $\sigma$.

By Proposition~\ref{fact:quantitative Pos are Lox}, there exists $\omega'>0$ such that $a'_\omega G_{\geq 0}\subset L_{\omega'}$.
By Lemma~\ref{lem:lox length estimate}, there exists $C>0$ only depending on $\omega'$ and hence on $\sigma$ such that for any $A\in L_{\omega'}$, the first two derivatives at $X=0$ of $X\mapsto \lambda(A\cdot \exp(X))$ are bounded above in norm by $C$.

Then the above quantity in \eqref{eq:I like to deriiiiive2} is bounded above by $C$, and the quantity in \eqref{eq:I like to deriiiiive1} is bounded above by $kC=C\iota(\gamma,\gamma^*)$.

Let us now estimate second derivatives.
Since the second derivatives of the $r_i$'s are zero, according to \eqref{eq:conjclass} and \eqref{eq:derive2}, an upper bound for this quantity is given by

\begin{equation}\label{eq:I like to deriiiiive3}
\begin{split}
  &\sum_{i}
 \norm{d^2\lambda\left(a'_{t_1(h_0)} \cdots \exp({r_i(z_0)}) \exp(\cdot) 
 a'_{t_{i+1}(h_0)}
 \cdots   \exp({r_k(z_0)}) \right)\vert_{0}}\\
 &+2\sum_{i< j}  \norm{d^2 \lambda\left( a'_{t_1(h_0)} \cdots \exp({r_i(z_0)})\exp(\cdot) 
 \cdots  
 \exp({r_j(z_0)})\exp(\cdot) 
 \cdots \exp({r_k(z_0)}) \right)\vert_{0,0}}. \notag
 \end{split}
\end{equation}
For all $1\leq i,j\leq k$, let $\mf d(i,j)=\min(\vert j-i\vert,k-\vert j-i\vert)$.
Up to taking $\omega'$ smaller we can assume $(a'_\omega G_{\geq 0})\times(a'_\omega G_{\geq 0})\subset DL_{\omega'}$, and 
$$a'_{t_{i+1}(h_0)}\cdots  \exp({r_j(z_0)})\in P_{(j-i)\omega,\theta}$$
and
$$
a'_{t_{j+1}(h_0)}\cdots \exp({r_k(z_0)})\cdot a'_{t_1(h_0)} \cdots \exp({r_i(z_0)})\in P_{(i-j+k)\omega,\theta}.
$$
Hence we can apply the second estimate of Lemma~\ref{lem:lox length estimate} on the derivatives of 
$(X,Y)\mapsto \lambda(A \cdot\exp(X) B \cdot\exp(Y))$ at $0$, which gives a bound on \eqref{eq:I like to deriiiiive3} of the form 
$$
\sum_i C + \sum_{i\neq j}Ce^{-\mf d(i,j)\omega'}= C\sum_{i=1}^k\sum_{j=1}^ke^{-\mf d(i,j)\omega'}\leq C\sum_{i=1}^k\sum_{m=-\infty}^{+\infty} e^{-\vert m\vert\omega'}\leq \frac{2Ck}{1-e^{-\omega'}}
$$
for some constant $C>0$. This completes the proof.

	\section{Quantitative convergence of currents} \label{sec:convergence-of-currents}

In Section~\ref{sec:hitchinrep} we introduced the measure of maximal entropy 
for Hitchin representations with respect to a Finsler metric. 
In this section we investigate the behavior of these
measures along grafting rays in the Hitchin component. 
Using the geometric control established in Section~\ref{QI-proof}, we 
compare length functions for representations obtained by Hitchin grafting
rays to length functions of the corresponding 
abstract grafted surfaces, viewed as functions on 
the unit tangent bundle of the hyperbolic surface $S$ which is the starting
point for the grafting, and estimate the entropy of the reparameterized
flow. 
This then leads to the proof of Theorem~\ref{main1} from the introduction.

The Finsler metric on $\X$ used for the pressure metric
is normalized in such a way that its restriction to 
a hyperbolic plane stabilized by an irreducible representation of $\PSL_2(\mathbb{R})$ 
coincides with the Riemannian metric of constant curvature $-1$.

We start with a hyperbolic metric on the closed surface $S$ of genus $g\geq 2$ and 
choose a simple geodesic multicurve $\gamma^*$ on $S$ (the grafting locus) with $k\geq 1$ components.
For each grafting parameter $z=(z_e)_{e\subset\gamma^*}\subset \mf a^k$, denote by
$\rho_z$ the Hitchin grafting representation with datum $z$ (see Definition~\ref{def:hitchin grafting rep}).

By Proposition~\ref{hoelderfunction}, for each $z$ there exists a positive H\"older continuous 
function $f_z$ on the unit tangent bundle $T^1S$ of $S$ with the property that 
for every periodic orbit $\gamma$ for the geodesic flow $\Phi^t$ on $T^1S$, we have that 
\[\ell_{f_z}(\gamma)=\int_\gamma f_z\] 
equals the translation length of the conjugacy class
determined by the element $\rho_z(\gamma)\in \PSL_d(\mathbb{R})$
with respect to the Finsler metric.

The H\"older continuous function $f_z$ on $T^1S$  determines a 
reparameterization $\Phi^t_{f_z}$ of the geodesic flow $\Phi^t$ on $T^1S$, whose measure of maximal entropy corresponds to a $\Phi^t$-invariant Gibbs equilibrium state $\nu(z)$ on $T^1S$.
There are several possible normalizations for this equilibrium state.
We assume $\nu(z)$ to be normalized in such a way that
\begin{equation}\label{normalization}
\int f_zd\nu(z)=1\text{ for all }z.\end{equation} 
Note that this normalization
only depends on the cohomology class of $f_z$ and hence it does not depend on 
choices. 
Our main goal is to determine the possible limits of $\nu(z)$ as 
the cylinder height of every component $z_e$ of $z$ (that is, at every component of the multi-curve $\gamma^*$)
tends to infinity, and to show that 
the intersection numbers with $\gamma^*$ of the geodesic currents $\hat \nu(z)$ determined 
by the measures $\nu(z)$ decay exponentially fast.




By Section~\ref{sec:meas-conv}, the equilibrium measure of the function $-f_z$ 
can be described in terms of Patterson--Sullivan measures. Denoting as  before by $\mc F$ the flag variety of 
$\PSL_d(\R)$,  
recall that for $\zeta,\eta\in\mc F$ and $x,y\in\X$, the function $b_\zeta^{\mf F}(x,y)$ denotes the Busemann cocycle  and $\langle \zeta|\eta\rangle_x$ denotes the Gromov product associated to the Finsler metric $\mf F$ (see Equations~\ref{busemann} and~\ref{gromovproduct}).

For any non-trivial grafting datum $z$ with nontrivial cylinder height, 
let $\Xi_z:\partial_\infty\H^2\to\mc F$ be the limit map associated to the Hitchin grafting representation  $\rho_z$. 
Then there exists a family of Patterson Sullivan measures $(\mu_z^x)_{x\in \X}$ on $\partial_\infty \H^2$ 
such that for all $x,y\in \X$ and $\gamma\in\pi_1(S)$ we have $\mu_z^{\rho_z(\gamma)x}=\gamma_*\mu_z^x$ and 
\begin{equation}\label{eq:conf measure}
\frac{d\mu_z^y}{d\mu_z^x}(\xi) = e^{\delta(z) b^{\mf F}_{\Xi_z(\xi)}(x,y)},
\end{equation}
where $\delta(z)$ is the \emph{critical exponent} of 
the group $\rho_z(\pi_1(S))$, or, equivalently, the topological 
entropy of the reparameterized flow $\Phi_{f_z}^t$ on $T^1S$. 
These measure are unique up to a global multiplicative positive constant. 
Note that the equality~\ref{eq:conf measure} is immediate from the fact that the topological entropy of 
the reparameterized flow equals the expansion rate of the conditional measures on strong unstable
manifolds for its unique measure of maximal entropy, which in turn equals the critical exponent by
construction.

Finally there is a choice of normalization for the measures $\mu_z^x$ such that 
$\nu(z)$ is the quotient under $\pi_1(S)$ of the measure 
\begin{equation}\label{eq:sullivan measure}
e^{\delta(z) \langle \Xi_z(\xi)|\Xi_z(\eta)\rangle_x}d\mu_z^x(\xi)d\mu_z^x(\eta)dt
\end{equation}
on $\partial_\infty\H^2\times\partial_\infty\H^2\times\R$.
Note that the measures
$\mu_z^x$ are finite but in general they are not probability measures, 
instead their normalization is determined by the normalization of $\nu(z)$.

Since $\nu(z)$ and hence the geodesic current $\hat \nu(z)$ defined by $\nu(z)$
depends continuously 
(in fact, analytically) on $z$ by Proposition~\ref{orbitequ}, we can 
estimate the intersection $\iota(\hat \nu(z),\gamma^*)$ 
(here $\gamma^*$ is viewed as a Dirac current) 
using continuity of the intersection form on the space of currents. However, 
although the space of projective currents, equipped with the 
weak$^*$-topology, is compact since this is the case for the 
space of $\Phi^t$-invariant Borel probability measures on $T^1S$ where $\Phi^t$ is the 
geodesic flow, the family $\hat \nu(z)$ may not be precompact as the corresponding 
$\Phi^t$-invariant measure $\nu(z)$ on $T^1S$ is determined by the normalization 
(\ref{normalization}) and 
in general is not a probability measure. 
We shall use the Patterson--Sullivan measures 
to control the total volume of $\nu(z)$ and overcome this difficulty.

\subsection{The entropy of the subsurfaces}

The geodesic multicurve $\gamma^*$ decomposes $S$ into (closed) complementary components $S_1,\dots,S_k$. 
For each $i\leq k$ we denote by $K_i\subset T^1S$ the set of all unit tangent
vectors $v\in T^1S_i$ with the property that $\Phi^tv\in T^1S_i$ for all~$t\in\R$.

\begin{lemma}\label{closed}
For each $i$ the set $K_i$ is compact and $\Phi^t$-invariant. 
\end{lemma}
\begin{proof} The set $K_i$ is clearly $\Phi^t$-invariant and closed by continuity
of $\Phi^t$, 
hence it is compact.
\end{proof}

Since $S$ is a closed hyperbolic surface, the geodesic flow $\Phi^t$ on $T^1S$ is an 
Anosov flow and hence for each $i$ its restriction to the compact invariant 
set $K_i$ is an Axiom A flow. 

The preimage of the geodesic multicurve $\gamma^*$ in the universal covering
$\mathbb{H}^2$ of $S$ consists of a countable union of pairwise
disjoint geodesic lines. These geodesic lines decompose~$\mathbb{H}^2$ into
countably many connected components which are permuted by
the action of the fundamental group $\pi_1(S)$ of $S$.
If we denote by $\Gamma\subset \pi_1(S)$ the stabilizer of
one of these components $\tilde \Sigma$, which is a
convex subsurface of $\mathbb{H}^2$ with geodesic boundary,
then $\Gamma$ acts properly and cocompactly on $\tilde \Sigma$, with quotient
one of the components $S_i$ of $S-\gamma^*$.
Thus~$\Gamma$ is a non-elementary convex cocompact Fuchsian group.

The \emph{limit set}, that is, the set of accumulation points of
a $\Gamma$-orbit $\Gamma x\subset \mathbb{H}^2$ $(x\in \tilde \Sigma)$ 
in $\mathbb{H}^2\cup \partial_\infty \mathbb{H}^2$, is a $\Gamma$-invariant
Cantor subset $\Lambda$ of $\partial_\infty\mathbb{H}^2$.
The quotient under the action of $\Gamma$ of 
the set of all unit tangent vectors of geodesics with both endpoints 
in $\Lambda$ has a natural identification with the invariant
set $K_i\subset T^1S$. In particular, the restriction of $\Phi^t$ to 
$K_i$ is topologically transitive. Its topological entropy equals
the Hausdorff dimension $\delta_i\in (0,1)$ of $\Lambda$
\cite{Su84}.

Write $K=\cup_i K_i$ and let $\delta >0$ be the topological entropy of $\Phi^t_{\vert K}$.
We have $\delta=\max\{\delta_i\mid i\leq k\}$.
Recall that $\delta(z)$ denotes the topological entropy of the reparameterized 
flow $\Phi^t_{f_z}$ on $T^1S$ and equals the critical exponent of the group
$\rho_z(\pi_1(S))\subset \PSL_d(\mathbb{R})$.

We have bounds on $\delta(z)$.
The upper bound is very general:

\begin{theorem}[{Corollary 1.4 of~\cite{PS17}}]\label{thm:upper bound entropy}
    There is a constant $m>0$ that bounds from above the entropy of any Hitchin representation.
\end{theorem}

The lower bound depends on the choice of the
grafting locus $\gamma^*$ and the hyperbolic metric on $S$, and its proof is classical.

\begin{lemma}[{e.g.\ Theorem 4.1 of~\cite{BZZ}}]\label{entropy}
$\delta(z)\in (\delta,m]$ for all $z$, where $m>\delta$ is the universal constant from the above Theorem~\ref{thm:upper bound entropy}.
\end{lemma}
\begin{proof}
By definition of a Hitchin grafting representation,
the image  $\rho_z(\Gamma)$
under $\rho_z$ of the fundamental group $\Gamma$ of any component of $S-\gamma^*$ is conjugate to its image under $\rho$, and hence has the same critical exponent.
Suppose we picked the component with largest critical exponent, namely $\delta$.

Then $\rho_z(\Gamma)$ is also Anosov ($\Gamma$ is quasi-convex in $\pi_1(S)$) and its limit set is a proper subset of that of $\rho_z(\pi_1(S))$ so by Theorem 4.1 of~\cite{BZZ} it has a strictly smaller critical exponent. Thus the critical exponent of 
$\rho_z(\pi_1(S))$ is bigger than $\delta$. 
%
\end{proof}

Let $h_{\rm top}(\Psi^t)$ be the topological entropy of a flow $\Psi^t$ on a compact space;
thus $\delta =h_{\rm top}(\Phi^t_{\vert K})$.
A \emph{measure of maximal entropy} for 
$\Phi^t_{\vert K}$ is an invariant probability measure $\mu$ with $h_\mu=\delta$.

Since  $\Phi^t_{\vert K_i}$ is a topologically transitive Axiom A flow and $K_i$ is compact,
it admits a unique measure $\nu_i$ of maximal entropy.  
The measure $\nu_i$ 
is a Gibbs equilibrium state for~$\Phi^t_{\vert K_i}$ with respect to 
the constant function $1$, and it can be 
obtained from a Patterson Sullivan construction~\cite{Su84}. 
The following well known fact will be useful later on.

\begin{lemma}\label{entropysplit}
A measure of maximal entropy for $\Phi^t_{\vert K}$ exists. 
It is unique if and only if
there exists a number $i\leq k$ such that $h_{\rm top}(\Phi^t_{\vert K_i})>
\max\{h_{\rm top}(\Phi^t_{\vert K_j})\mid j\not=i\}$.
In this case the measure of maximal entropy is supported in $K_i$.
\end{lemma}
\begin{proof}
Write again $K=\cup_iK_i$. 
The function which associates to a $\Phi^t_{\vert K}$-invariant probability measure $\mu$ its entropy $h_\mu$ is affine:
for $\mu,\eta$ and $s\in (0,1)$ we have
$h_{s\mu+ (1-s)\eta}=sh_\mu+(1-s)h_\eta$. 

The topologically transitive invariant subsets $K_i\subset K$ intersect at most along 
a finite number of periodic orbits. As a consequence, any $\Phi^t$-invariant 
probability measure $\mu$ on~$K$ can be decomposed as $\mu=\sum_i\mu_i$ where 
$\mu_i$ is supported in $K_i$. The decomposition is unique if the $\mu$-mass of any periodic 
orbit for $\Phi^t$ which projects to a component of $\gamma^*$ vanishes. 

Since $\Phi^t_{\vert K_i}$ is a topologically transitive axiom A flow, it admits a unique measure $\nu_i$ of maximal entropy.
Then we have $h_{\nu_i}=h_{\rm top}(\Phi^t_{\vert K_i})$. 
Let $\mu=\sum_i\mu_i$ be any $\Phi^t$-invariant Borel probability 
measure on $K$.
Let $s_i=\mu_i(K_i)$; then $\sum_is_i=1$ and 
\[h_\mu=\sum_is_ih_{\mu_i}\leq \sum_i s_ih_{\rm top}(\Phi^t_{\vert K_i})\leq \delta\]
with equality if and only if $s_j=0$ for all $j$ such that 
$h_{\rm top}(\Phi^t_{\vert K_j})<\delta$, and $\mu_j=\nu_j$ if $s_j>0$. 
In particular, a measure of maximal entropy exists, and 
if there exists a unique $i\leq k$ such that 
$h_{\rm top}(\Phi^t_{\vert K_i})=\delta$, then such a measure is unique and coincides with $\nu_i$. 
\end{proof}

\subsection{The total mass of the equilibrium state}

For the fixed hyperbolic metric on~$S$ with unit tangent bundle $T^1S$ and geodesic 
flow~$\Phi^t$ denote by $\nu^1(z)$ the $\Phi^t$-invariant probability measure on 
$T^1S$ which is a multiple of~$\nu(z)$. 
It turns out that the two normalisations $\nu(z)$ and $\nu^1(z)$ for the equilibrium states
are comparable independently of $z$, as soon as 
the grafting datum $z$ is taken in $\ker\alpha_0$ where $\alpha_0$ is the linear functional
which determines the Finsler norm of the tangent of a Riemannian geodesic in $X$ 
which is invariant under $\rho(\gamma^*)$ (or a component of $\rho(\gamma^*)$).

\begin{lemma}\label{lem:total mass of nu}
    For any $\sigma>0$ there exists a constant $C>0$ such that if the length of each component of $\gamma^*\subset S$ 
    is at most $\sigma$,
    then for any grafting parameter $z\subset\ker\alpha_0^\perp$,
    $$C^{-1}\leq \Vert\nu(z)\Vert = \nu(z)(T^1S) \leq C.$$
\end{lemma}
\begin{proof}
Put $\nu^1(z)=\frac{\nu(z)}{\Vert\nu(z)\Vert}$ so that $\nu^1(z)$ is a probability measure on $T^1S$.
    Then $\Vert \nu(z)\Vert=(\int f_z d\nu^1(z))^{-1}$ since by equation (\ref{normalization}), $\nu(z)$ was 
    normalized so that $\int f_z d\nu(z)=1$.

    By definition of the equilibrium state of $-f_z$ 
    and the fact that the entropy of the reparameterized flow $\Phi^t_{f_z}$ equals $\delta(z)$, we have
    \begin{equation}\label{entropyeq}
    \int f_z d\nu^1(z)=\frac{h_{\nu^1(z)}}{\delta(z)}.
    \end{equation}
    Since $h_{\nu^1(z)}\leq 1$ (the topological entropy of $\Phi^t$ is $1$, and is greater than or equal to the entropy of any invariant measure) and $\delta(z)> \delta$ by Lemma~\ref{entropy}, it holds $\int f_z d\nu^1(z)\leq \frac1\delta$.
    It remains to get a lower bound.

    By Theorem~\ref{thm:quasiisom lengths},
we have 
$$
\int f_{z}\frac{d\zeta}{\ell(\zeta)} \geq \left( 1+\frac{C}{L+1} \right)^{-1},
$$
for any $\zeta\in\pi_1(S)$, represented by a periodic orbit for $\Phi^t$ of length
$\ell(\gamma)$, and 
where $L\geq 0$ is any lower bound on the heights of the cylinders added 
along the components of $\gamma^*$ to construct $S_z$ (see Definition~\ref{def:abstractgraf}).

Then by density of the convex hull of currents supported 
on closed geodesics in the space of all currents, we get
\begin{equation*}
\int f_{z}d\nu^1(z) \geq \left( 1+{C} \right)^{-1}.\qedhere
\end{equation*}
\end{proof}

\subsection{The total mass of the Patterson--Sullivan measure}

In this section we establish estimates on the total mass of some of the Patterson--Sullivan measures (see Equations~\eqref{eq:conf measure} and \eqref{eq:sullivan measure}). To this end we use the equivariant path isometry
$\tilde Q_z:\tilde S_z\to \X$ to view the family $(\mu_z^x)$ of Patterson Sullivan measures on $\partial \H$
as a family of measures parameterized by points in the universal covering 
$\tilde S_z$ of the abstract grafted surface $S_z$. In the sequel we use this convention without further
mention. 


\begin{proposition}\label{prop:upper bound on PS mass}
    For any $\sigma>0$ there is a constant $C>0$ such that if the length of each component of $\gamma^*$ is at most $\sigma$, then for any grafting parameter $z\subset\ker\alpha_0^\perp$, in any hyperbolic piece of $\wt{S}_z$ there exists a point $x$ such that 
    $$\mu^x_z(\partial_\infty\H^2)\leq C.$$
\end{proposition}


    

    


The strategy of the proof is as follows (see Figure~\ref{fig-control-mass}). 
Assume that each component of~$S-\gamma^*$ is a pair of pants. We fix one of them, say the
pair of pants $\Sigma$, and its fundamental group $\Gamma$.
Let $\tilde \Sigma$ be the universal covering of $\Sigma$. Then $\tilde \Sigma\subset \H^2$ is a convex hyperbolic 
surface with geodesic boundary. 
We find two disjoint intervals $I,J\subset\partial_\infty\H^2$, 
numbers $C_1,C_2>0$ and a fundamental domain for the action $\Gamma\acts\tilde \Sigma$,
made of two right-angle hexagons $H\cup H'$ whose diameters are bounded from above by 
a constant only depending on $\sigma$ and 
which depend in a suitable sense continuously on the data, so that the following holds.
Let $x$ be the center of $H$.
First, the masses $\mu^x_z(I)$ and $\mu^x_z(J)$ are bounded from below by~$C_1 \mu^x_z(\partial_\infty\H^2)$. 
Second, each geodesic connecting a point in $I$ to a point in $J$ intersects~$H$ in an arc of length at least $C_2$
and hence passes uniformly near $x$. We then can estimate the Gromov product and bound 
the product measure $\mu^x_z(I)\times\mu^x_z(J)=\mu^x_z\times\mu^x_z(I\times J)$ 
from above by a constant multiple of $\nu(z)(T^1S)$, 
which is uniformly bounded from above by Lemma~\ref{lem:total mass of nu}. 

\begin{figure}[ht]
    \begin{center}
        \begin{picture}(80,80)(0,0)
        \put(0,0){\includegraphics[width=80mm]{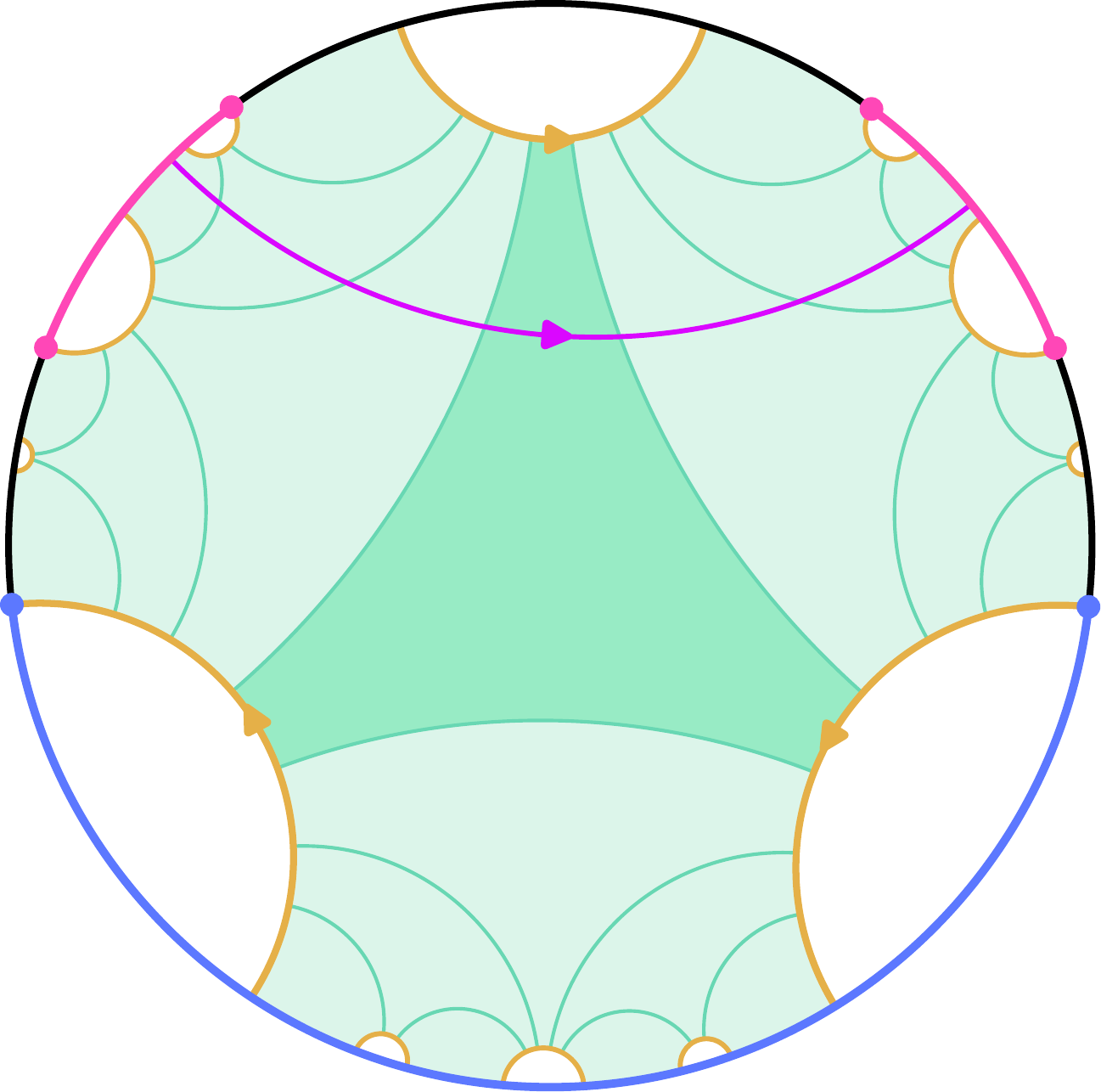}}
        \color[RGB]{228,176,72}
        \put(39,72){$\bar a_u$}
        \put(62.5,24){$\bar b_u$}
        \put(13.5,24){$\bar c_u$}
        \put(51,79){$\bar a_u^+$}
        \put(25,79){$\bar a_u^-$}
        \put(61,3){$\bar b_u^+$}
        \put(80,34){$\bar b_u^-$}
        \put(-5,34){$\bar c_u^+$}
        \put(15,3){$\bar c_u^-$}
        \color[RGB]{92,120,255}
        \put(5,12){$A_u$}
        \color[RGB]{255, 71, 183}
        \put(69,68){$\bar a_u\cdot A_u$}
        \put(-3,68){$\bar a_u^{-1}\cdot A_u$}
        \color[RGB]{29, 168, 105}
        \put(37,35){$H_u$}
        \end{picture}
    \end{center}
    \caption{Control of $\nu(T^1S)$ using the measure for $\mu$ of two intervals $I=\bar a_u\cdot A_u$ and $J=\bar a_u^{-1}\cdot A_u$ in $\de\H^2$. The green hexagon $H_u$ is half of a fundamental domain of a pair of pant $P_u$. The purple geodesic goes from $I$ to $J$ and intersects $H_u$ in an arc whose length is bounded from below.}
    \label{fig-control-mass}
\end{figure}

We begin with establishing a few estimates in a more general setting involving representations 
of the fundamental group of a pair of pants (the free group $F_2$ with two generators)
into $\PSL_2(\R)$.
Let us introduce some notations.
Let $P$ be a topological pair of pants, equipped with a fixed orientation.
We fix a basepoint $p_0$ in $P$ and three generators $a,b,c$ of the fundamental group 
$\pi_1(P,p_0)=F_2$ such that $c\cdot b\cdot a=1$ and each generator corresponds to one of the boundary components of $P$.

For a set of lengths $u=(u_a,u_b,u_c)\in [0,\infty)^3$ there is a unique hyperbolic structure on~$P$ 
whose boundary components have these lengths on $a,b,c$, and up to conjugation, there is a unique representation $j_u:\pi_1(P)\to\PSL_2(\R)$ associated to this hyperbolic structure which is normalized so that the following ordering assumption holds. 

Put $a_u,b_u,c_u$ instead of $j_u(a),j_u(b),j_u(c)$. Then $u_i$ are loxodromic elements of 
$\PSL_2(\mathbb{R})$ with axes
$\bar a_u,\bar b_u,\bar c_u\subset \H^2$, oriented to define the boundary orientation for the oriented
pair of pants $P$,
with endpoints  $\bar a_u^\pm,\bar b_u^\pm,\bar c_u^\pm\subset \partial_\infty\H^2=S^1$. 
We use the abuse of notation that if say $u_a=0$, then $j_u(a)$ has only one fixed point on $\partial\H^2$,
and $\bar a_u=\bar a_u^+=\bar a_u^-$. We require that 
the cycle 
$(\bar{a}_u^-,\bar{a}_u^+,\bar{b}_u^-,\bar{b}_u^+,\bar{c}_u^-,\bar{c}_u^+)$ is oriented clockwise for the circular 
order on $\de_\infty\H^2$. We may also assume that the center $0$ of the unit disk $D=\H^2$ is contained in the 
convex hull of the limit set of $j_u$ and that $j_u$ varies continuously in $u$.

Consider the three intervals of $\de_\infty\H^2$ (that is, the segment in $\partial_\infty \H^2$ determined
by the clockwise orientation of $S^1$ and its endpoints)

\begin{itemize}
    \item $A_u=[\bar{b}_u^-,\bar{c}_u^+]$,
    \item $B_u=[\bar{c}_u^-,\bar{a}_u^+]$,
    \item $C_u=[\bar{a}_u^-,\bar{b}_u^+]$.
\end{itemize}

By construction, we have 
$A_u\cup B_u\cup C_u=\de_\infty\H^2$, so for any finite measure $\mu$ on $\de_\infty\H^2$, one of the intervals has mass at least $\frac{1}{3}\mu(\de_\infty\H^2)$. Put $A_u^+=a_u\cdot A_u$, $A_u^-=a_u^{-1}\cdot A_u$, and similarly for $B_u^+,B_u^-,C_u^+,C_u^-$. 

\begin{lemma}
    The intervals $A_u^+$ and $A_u^-$ are disjoint. Similarly $B_u^+\cap B_u^-=\emptyset$ and $C_u^+\cap C_u^-=\emptyset$.
\end{lemma}

\begin{proof}
    First notice that $a_u\cdot \wb{c}_u^+$ belongs to $[\wb{a}_u^+,\wb{c}_u^+]$ since $a_u$ is an hyperbolic element with attractive fixed point $\wb{a}_u^+$. By construction, we have $c_u\cdot b_u\cdot a_u=1$, so that 
    $$a_u\cdot\bar{c}_u^+=({b}_u^{-1}\cdot{c}_u^{-1})\cdot\bar{c}_u^+={b}_u^{-1}\cdot\bar{c}_u^+$$ 

    So $a_u\cdot\bar c^+_u$ is included in both $[\bar a^+_u,\bar c^+_u]$ and $[\bar c^+_u, \bar b^-_u]$, whose intersection is equal to the interval $[\bar a^+_u,\bar b^-_u]$. Similarly, $a_u\cdot\bar b^-_u$ lies inside 
    $[\bar a^+_u,a_u\cdot\bar c^+_u]\subset[\bar a^+_u,\bar b^-_u)$. And since $A_u=[\bar b^-_u,\bar c^+_u]$, $a_u\cdot A_u\subset [\bar a^+_u,\bar b^-_u)$, and similarly $a^{-1}_u\cdot A_u\subset (\bar c^+_u,\bar a^-_u]$, it follows that they are disjoint.    
\end{proof}

Write $\Gamma_0=\{a,a^{-1},b,b^{-1},c,c^{-1}\}\subset\pi_1(P)$

\begin{corollary} \label{cor:mass-interval-bounded-below}
    If $\mu$ is a $j_u(\pi_1(P))$-quasi-invariant finite measure on $\de_\infty\H^2$, then the measure for $\mu\times\mu$ of one of the three products $A_u^-\times A_u^+$, $B_u^-\times B_u^+$, $C_u^-\times C_u^+$  has mass at least $\frac{C^2}{9}\mu(\de_\infty\H^2)^2$, where
    $$C=C_{\mu,u}=\inf\left\{\frac{dj_u(\gamma)_*\mu}{d\mu}(\xi):\ \xi\in\partial_\infty\H^2,\ \gamma\in \Gamma_0\right\}. $$
\end{corollary}

We will also need an estimate on the lengths of the intersection of geodesics from $A_u^-$ to $A_u^+$, with $H_u\subset\H^2$ the (possibly degenerate) right-angled hexagon adjacent to the axes of $j_u(a),j_u(b),j_u(c)$.

\begin{lemma}\label{fact:intersection with hexagon}
    For any $\sigma>0$ there exists $L_\sigma$ such that for any $u\in[0,\sigma]^3$, for all $(x,y)$ in $A_u^-\times A_u^+$, $B_u^-\times B_u^+$ or $C_u^-\times C_u^+$, the length of the intersection of the hexagon $H_u$ with the geodesic from $x$ to $y$ is at least $L_\sigma$.
\end{lemma}

\begin{proof}
    This is a direct consequence of the following three facts. $A_u^\pm$ varies continuously with $u$. The length $length(\gamma\cap H_u)$ for a geodesic $\gamma$ with ends in $A_u^-\times A_u^+$ is positive and continuous in the pair $(u,\gamma)$. And $[0,\sigma]^3$ is compact.
\end{proof}

\begin{proof}[Proof of Proposition~\ref{prop:upper bound on PS mass}]
    According to \thref{thm:short-pant-decomposition}, 
    we can choose a pair of pants in $S\setminus\gamma^*$, whose boundary curves 
    have length bounded from above by a constant only depending on the genus of $S$ and  $\sigma$.
    Let us identify it topologically with $P$, and identify a convex subsurface $\wt P$ of a hyperbolic piece 
    inside the universal covering $\wt{S}_z$ of the abstract grafted surface with the universal cover of $P$.
    We denote by $u=(u_a,u_b,u_c)$ the lengths of the boundary components of this pair of pants. 
    The surface $\tilde P$ contains a right angled hexagon $H_u$ whose double is a fundamental domain for 
    the deck group $\pi_1(P)$. 

    Identify $\PSL_2(\R)$ with a subgroup of $\PSL_d(\R)$ and $\H^2$ with a totally geodesic subspace of $\X$.
    Up to conjugation of the Hitchin grafting representation $\rho_z$, 
    we may assume that its restriction to $\pi_1(P)$ coincides with $j_u:\pi_1(P)\to \PSL_2(\R)$, so that the fixed hyperbolic piece $\wt P$ of the characteristic surface is contained in $\H^2\subset\X$. 
    More precisely, it is the convex hull of the limit set of $j_u(\pi_1(P))$. 
    Note that this makes sense since the boundary of $\H^2$ embeds naturally into 
    the flag variety $\mc F$ as well as the visual boundary $\partial_\infty \X$ of $\X$.
    We choose a point $x$ in the interior of the hexagon $H_u\subset  \wt P$ as a basepoint in $\H^2$. 

    By  Lemma~\ref{lem:total mass of nu}, Corollary~\ref{cor:mass-interval-bounded-below} and Lemma~\ref{fact:intersection with hexagon}, and since the Gromov product is nonnegative, there is a constant $C$ such that
    $$
    C\geq \left(e^{\delta \langle\cdot|\cdot\rangle_x}\mu^x_z\times\mu^x_z\times\Leb\right)(T^1 H_u) \geq \frac{L_\sigma C_\delta^2}{9} \mu^x_z(\partial_\infty\H^2)^2
    $$
    where $T^1H_u$ denotes the set of unit tangent vectors in $T^1\H^2$ with footpoint in the hexagon $H_u$, and
    $C_\delta$ is the infimum of the constants $C_{\mu,u}$ appearing in the corollary, that is
    $$C_\delta=\inf \left\{e^{\delta b^{\mf F}_\xi(x,j_u(\gamma)x)}:\ \xi\in\mc F,\ u\in[0,\sigma]^3,\ \gamma\in\Gamma_0\right\}.$$

    To conclude the proof one can use Theorem~\ref{thm:upper bound entropy}, which implies $C_\delta\geq C_{m}$ where $C_m>0$ is a constant that only depends on the genus of $S$ and the choice of length function on $\PSL_d(\R)$ 
    (i.e.\ the choice of a linear functional $\alpha_0$ on $\mf a$).
\end{proof}

\begin{numremark}\label{remark:upper bound on PS mass}
    In the proposition~\ref{prop:upper bound on PS mass}, we may actually take $x$ to be any point in $\epsilon$-thick part of $S$, for $\epsilon>0$ first fixed. To see that, modify the proof as follow. Fix $\epsilon>0$, and instead of taking a unique representation $j_u$ for a fixed data of $u\in[0,\sigma]^3$, take a larger compact set of representations $J$ so that each representation in $J$ is conjugated to one $j_u$ as above. Furthermore, for each $j_u$ and each point $p$ in the $\epsilon$-thick part of the hyperbolic pair of pants obtained as the quotient by $j_u$ of the convex core of $j_u$, there exists a representation~$j_{u,p}$ in $J$ and an isometry $f$ of $\H^2$ which conjugates $j_{u,p}$ to $j_u$ and which sends the origin of $\H^2$ to a preimage of $p$ in the convex core of $j_{u,p}$. The same compactness argument holds except that the constant $C$ my be larger. Additionally, for $\epsilon>0$ small enough, 
    the $\epsilon$-thick part of a surface $S$ is equal to the union of the $\epsilon$-thick parts of the pairs of pants 
    which are contained in some pants decomposition of $S$.
\end{numremark}

\subsection{Estimating the intersection number of the equilibrium state with the grafting locus}

Recall that $\gamma^*$ is the grafting locus and $\hat \nu(z)$ is the geodesic current defined by the equilibrium measure $\nu(z)$ for a grafting parameter $z$. 
The goal of this section is to show that if the heights of the flat cylinders in the grafted surface go to infinity, 
then $\nu(z)$ concentrates more and more on the components of $S-\gamma^{*}$ and avoids crossing the grafting locus, and this with exponential speed.
A more precise consequence will be that the intersection number $\iota(\hat \nu(z),\gamma^*)$ goes to zero with speed $Ce^{-\delta L/2}$ where $L$ is the minimal height of the flat cylinders, $\delta$ is the entropy of $S-\gamma^{*}$, and $C$ is a constant that depends on the hyperbolic length of $\gamma^{*}$.




In practice, we will prove a stronger result, which in vague terms states that if we see the equilibrium state $\nu(z)$ as a measure on the unit tangent bundle of the \emph{grafted} surface (instead of the hyperbolic surface), then as $L\to\infty$ the mass given by $\nu(z)$ to the flat cylinders (which have length at least $L$) goes to zero exponentially fast.
Let us now explain more rigorously what this means.

Let $\gamma_0^*$ be a component of $\gamma^*$.
Let $\wt \gamma^*\subset \wt S\simeq\H^2$ be the preimage of~$\gamma^*$  and choose a component $\hat \gamma^*_0\subset \wt \gamma^* $ of the preimage of $\gamma_0^*$. 
Denote by $I^-$ and $I^+$ the two connected components of $\de_\infty \H^2-\hat\gamma^*_0(\pm \infty)$. 
Recall that an oriented geodesic in $\H^2$ can be thought of as an ordered pair of distinct points $(\xi^-,\xi^+)\in \partial_\infty \H^2\times\partial_\infty\H^2-\Delta$.
Then a geodesic $(\xi^-,\xi^+)$ intersects $\hat \gamma^*_0$ transversely if and only if $(\xi^-,\xi^+)\in I^-\times I^+ \cup I^+\times I^-$.
Let $g_u\in\pi_1(S)$ be a generator of the infinite cyclic subgroup $\langle g_u \rangle$ of $\pi_1(S)$ which preserves $\hat \gamma_0^*$ and acts on it as a group of translations.
Choose a fundamental domain $\Omega^\pm$ for the action of $\langle g_u\rangle$ on $I^\pm $ of the form $\Omega^\pm=[\xi_0^\pm,g_u\xi_0^\pm)\subset I^\pm$, where $\xi_0^\pm\in I^\pm$ are taken so that the geodesic $(\xi_0^-,\xi_0^+)$ crosses $\hat\gamma_0^*$ orthogonally at a point $x$.

To every pair $(\xi^{-},\xi^{+})\in I^{-}\times I^{+}$ is associated an infinite admissible path $a$ in the grafted surface $\wt S_{z}$ that projects onto the admissible path in $\wt S\simeq \H^{2}$ from $\xi^{-}$ to $\xi^{+}$, and hence that crosses the flat band $B\subset \wt S_{z}$ sitting above $\gamma_{0}^{*}\subset\wt S$.
We define $l^{\flat}_{\gamma_{0}^{*},z}(\xi^{-},\xi^{+})$ to be the Finsler length of $a\cap B$, and use it to define the length in the flat part of any current $\hat\lambda$:
\begin{equation}
\ell^{\flat}_{\gamma_{0}^{*},z}(\hat\lambda) :=  \int_{\Omega_{-}\times I_{+}\cup\Omega_{+}\times I_{-}}l^{\flat}_{\gamma_{0}^{*},z}d\hat\lambda \quad \text{and} \quad \ell^{\flat}_{z}(\hat\lambda)=\sum_{\gamma_{0}^{*}\subset\gamma^{*}}\ell^{\flat}_{\gamma_{0}^{*},z}(\hat\lambda) .
\end{equation}

\begin{remark}
\begin{enumerate}
\item If $\hat\lambda$ is the current associated to a closed curve $\eta\subset S$ then we retrieve $\ell^{\flat}_{z}(\hat\lambda)=\ell^{\flat}_{z}(\eta)$ from Theorem~\ref{cor:firstandsecond}.
\item One has $\iota(\hat\lambda,\gamma^{*})=\hat\lambda(\Omega_{-}\times I_{+}\cup\Omega_{+}\times I_{-})$, whence
$\iota(\hat\lambda,\gamma^{*})\leq L\ell^{\flat}_{z}(\hat\lambda)$ (with $L$ the minimal height of the flat cylinders).
    For a reminder on currents that implies this, see the appendix of~\cite{Bo88} and  Chapter 8.2.11 of~\cite{Mar16}).
\item Let $\hat\gamma_{1},\hat\gamma_{2}\subset\H^{2}$ be two other components of $\tilde\gamma^{*}$ such that $\hat\gamma_{1},\hat\gamma_{0}^{*},\hat\gamma_{2}$ lie in this order in $\H^{2}$ with no other components of $\tilde\gamma^{*}$ in between, and $\hat\gamma_{1},\hat\gamma_{2}$ bound intervals $J_{1},J_{2}\subset\partial\H^{2}$ such that $J_{1}\subset I_{-}$ and $J_{2}\subset I_{+}$.
Then $l^{\flat}_{\gamma_{0}^{*},z}$ is constant on $J_{1}\times J_{2}$.
The reason is that the admissibles paths associated to two different geodesics from $J_{1}$ to $J_{2}$ coincide on their three pieces from $\hat\gamma_{1}$ to $\hat\gamma_{2}$: the first piece follow the orthogeodesic from $\hat\gamma_{1}$ to $\hat\gamma_{0}^{*}$, the last piece is the orthogeodesic from $\hat\gamma_{0}^{*}$ to $\hat\gamma_{2}$, and the middle piece is the geodesic in the flat band $B$ that connects the first and third pieces, and whose length is computed by $l^{\flat}_{\gamma_{0}^{*},z}$.
\item As a function on the space of currents, $\ell^{\flat}_{z}$ is linear but not continuous. 
However, it is continuous at currents that give zero measure to the set of geodesics that are asymptotic to a components of $\tilde\gamma^{*}$ in $\wt S$,
because of the previous point of this remark.
Since $\hat \nu(z)$ satisfies this condition (it is of the form $f\mu\otimes\mu$ where $\mu$ is a \emph{nonatomic} measure on $\partial\H^{2}$, see \eqref{eq:sullivan measure}), by \eqref{eq:gibbs as limit} we have
\begin{equation}
\ell^{\flat}_{z}(\hat\nu(z)) = \lim_{R\to\infty}\frac{1}{\myhash  N_{\rho_{z}}(R)} \sum_{\ell^{\mf F}(\rho_{z}(\alpha)) \leq R}\frac{\ell^{\flat}_{z}(\alpha)}{\ell^{\mf F}(\rho_{z}(\alpha))}
\end{equation}
\end{enumerate}
\end{remark}

The main result of this section is the following.

\begin{proposition}\label{boundaryestimate2}
  For any $\sigma>0$ there are $C,C',\delta_\sigma>0$ such that
  if every component of~$\gamma^*$ has hyperbolic length at most $\sigma$, then for any grafting parameter $z$ we have
  $$\ell^{\flat}_{z}(\hat \nu(z)) \leq  C(L+1)^{2}e^{-\delta(z) L}\leq C'e^{-\delta_\sigma L},$$
  where $L$ is the minimum of the heights of the flat cylinders in the abstract grafting (see Equation~\eqref{eq:cylinder height}).
\end{proposition}

\begin{remark}\label{rk:boundaryestimate}
\begin{enumerate}
\item Recall that $\delta(z)$ is bounded from below by the topological entropy $\delta$ of the geodesic flow on $S-\gamma^*$, which,  by Proposition~\ref{prop:entropy of hypsurf},  is itself  bounded from below by a positive constant that only depends on $\sigma$.
So in the above proposition, one can take $\delta_\sigma$ to be \emph{half} of the smallest possible entropy of hyperbolic surfaces homeomorphic to $S-\gamma^*$, with boundary lengths at most $\sigma$.
Then $C'$ could be $C$ times the maximum of the function $x\in[0,\infty)\mapsto (x+1)^2e^{-\delta_\sigma x}$.
\item One can check that the proof of Proposition~\ref{boundaryestimate2} gives the following estimate on the intersection number of $\hat\nu(z)$ with $\gamma^{*}$:
$$\iota(\hat \nu(z),\gamma^{*}) \leq  C(L+1)e^{-\delta(z) L},$$
where $C$ is a constant that depends on $\sigma$.
\end{enumerate}
\end{remark}

We will need the following result about Hitchin representations. 
In its formulation, $\partial \pi_1(S)$ denotes the Gromov boundary of the surface group $\pi_1(S)$.

\begin{lemma}\label{fact:hitchin cv to int of weyl}
    Let $\rho':\pi_1(S)\to G$ be a Hitchin representation with limit map $\Xi':\partial\pi_1(S)\to\mc F$, and let $(\gamma_n)_n\subset \pi_1(S)$ be a sequence converging to $\xi\in\partial\pi_1(S)$.
    Then for any compact set $K\subset\X$, the accumulation points of $\rho'(\gamma_n)K$ 
    in the visual boundary $\partial_\infty \X$ of $\X$ 
    are contained in the interior of the Weyl Chamber $\Xi'(\xi)$.
\end{lemma}
\begin{proof}
    This is a consequence of the Anosov property discussed in Section~\ref{sec:hitchinrep}, 
    which is satisfied by Hitchin representations, and a characterisation of this property in terms of Cartan decompositions of the images $\rho'(\gamma)$ with $\gamma\in \pi_1(S)$.

    Let $\rho'(\gamma_n)=k_n\exp({a_n})\ell_n$ be a Cartan decomposition, so that $k_n,\ell_n\in K$ (the maximal compact subgroup) and $a_n\in\mf a^+$.
    By a characterisation of the Anosov property (see Theorem 4.37 of~\cite{SurveyFanny} for more details and a history of this result), the angle formed by $a_n$ with each wall of the Weyl Cone $\mf a^+$ is bounded 
    from below independently of $n$. In other words, denoting by $\Vert \cdot \Vert$ the Euclidean norm 
    on $\mf a$, we have 
    $\alpha(a_n)\geq \mathrm{Cst}||a_n||$ for any positive root $\alpha$, which means precisely that 
    $(\exp({a_n}))_n$ accumulates in the interior of the Weyl Chamber 
    $\partial_\infty \exp({\mf a^+})\subset \partial_\infty \X$ in the ideal 
    boundary of the flat cone $\exp({\mf a^+})\subset  \X$.

    Up to passing to a subsequence 
    we may assume that $k_n\to k$ and $\ell_n\to \ell$. Let $\mf a^-$ be the Weyl chamber opposite to 
    $\mf a^+$, with boundary $\partial_\infty \exp({{\mf a}^-})$,
    viewed as a point in $\mc F$. Then 
    for any $\eta\in\mc F$ transverse 
    to $\ell^{-1}\partial_\infty \exp({\mf a^-})$ we have $\rho'(\gamma_n)\eta\to k\partial_\infty \exp({\mf a^+})$.
    By the definition of the limit map $\Xi^\prime$, 
    this implies that $\Xi'(\xi)=k\partial_\infty \exp({\mf a^+})$ (see Definition~\ref{def:anosov}).

    Then $\rho'(\gamma_n)\basepoint=k_ne^{a_n}\basepoint$ only accumulates in the interior of the Weyl Chamber $\Xi'(\xi)$, and the same holds for $\rho'(\gamma_n)K$ which lies at bounded distance from $\rho'(\gamma_n)\basepoint$.
\end{proof}

\begin{proof}[Proof of Proposition~\ref{boundaryestimate2}]
    Let $\gamma_0^*$ be a component of $\gamma^*$.
    Let $\wt \gamma^*\subset \H^2$ be the preimage of~$\gamma^*$  and choose a component
    $\hat \gamma^*_0\subset \wt \gamma^* $ of the preimage of $\gamma_0^*$. 
    Denote by $I^-$ and $I^+$ the two connected components of 
    $\de_\infty \H^2-\hat\gamma^*_0(\pm \infty)$. Recall that an oriented geodesic in 
    $\H^2$ can be thought of as an ordered sets of distinct points $(\xi^-,\xi^+)\in \partial_\infty \H^2\times\partial_\infty
    \H^2-\Delta$. Then  
    a geodesic $(\xi^-,\xi^+)$ intersects $\hat \gamma^*_0$ transversely if and only if 
    $(\xi^-,\xi^+)\in I^-\times I^+ \cup I^+\times I^-$.
    
    Let $g_u\in\pi_1(S)$ be a generator of the infinite cyclic subgroup $\langle g_u \rangle$ of $\pi_1(S)$ 
    which preserves $\hat \gamma_0^*$ and acts on it as a group of translations.
    Choose a fundamental domain $\Omega^\pm$ for the action of 
    $\langle g_u\rangle$ on $I^\pm $ of the form $\Omega^\pm=[\xi_0^\pm,g_u\xi_0^\pm)\subset I^\pm$, where $\xi_0^\pm\in I^\pm$ are taken so that 
    the geodesic $(\xi_0^-,\xi_0^+)$ crosses $\hat\gamma_0^*$ orthogonally at a point $x$.
     
    Using these notations, it follows from the definitions of the intersection number between
    two geodesic currents of $S$ (see the appendix of~\cite{Bo88} and
    Chapter 8.2.11 of~\cite{Mar16})  
    that
    \begin{equation}\label{intersect}\iota(\hat\nu(z),\gamma^*_0)=\hat\nu(z)(\Omega^-\times I^+)+\hat\nu(z)(I^+\times \Omega^-).
    \end{equation}
    Namely, the intersection number of $\hat \nu(z)$ with $\gamma_0$ is just the total 
    $\hat \eta(z)$-mass of all geodesics crossing transversely through a fundamental domain for the 
    action of $g_u$ on $\hat \gamma_0^*$. This set in turn is a fundamental domain for the action of 
    $\langle g_u\rangle$ on the space of all geodesics crossing through $\hat \gamma_0^*$.
    As $(\Omega^-\cup I^+)\cup (I^+\cup \Omega^-)$ is another such fundamental domain, and~
    $\hat \eta(z)$ is $\langle g_u\rangle$- invariant, this yields the formula (\ref{intersect}).

    By symmetry, it suffices to bound $\hat\nu(z)(\Omega^-\times I^+)$ from above.
    Recall that the map $\Xi_z:\partial_\infty\H^2\to\mc F$ is the limit map
    induced by the Hitchin grafting representation $\rho_z$. 
    Our computations rely on the characterisation of $\hat\nu(z)$ as the product 
    $$\hat\nu(z)=e^{\delta(z) \langle\Xi_z(\xi)|\Xi_z(\eta)\rangle_p}d\mu^p_z(\xi)d\mu^p_z(\eta),$$
    where $\langle\cdot|\cdot\rangle_p$ is the Gromov product based at $p$ and $p$ is any point in $\X$. 
    The measure $\hat\nu(z)(\Omega^-\times I^+)$ can then be bounded as follows, using that $I^+=\bigcup_ng_u^n\Omega^+$:
 \begin{align*} 
\hat\nu(z)(\Omega^-\times I^+) 
&= \sum_n \int_{\Omega^-\times g_{u}^{n}\Omega^+}e^{\delta(z) \langle\Xi_z(\xi)|\Xi_z(\eta)\rangle_p}d\mu^p_z(\xi)d\mu^p_z(\eta)\\
&\leq \mu_z^p(\Omega^-)\cdot 
\max_{(\xi,\eta)\in\Omega^-\times I^+}
(e^{\delta(z) \langle\Xi_z(\xi)|\Xi_z(\eta)\rangle_p})\cdot
 \sum_n\mu^p_z(g_u^n\Omega^+).
\end{align*}

 The strategy for     estimating these quantities and completing the proof  is the following.
    \begin{enumerate}[(i)]
        \item\label{itembdryesti} Make a suitable choice of basepoint $p$.
        \item \label{itembdryestii} Use Proposition~\ref{prop:upper bound on PS mass} to find a constant $C_1$ 
        only depending on $\sigma$ such that
        $\mu_z^p(\Omega^-)\leq C_1$.
        \item \label{itembdryestiii} Use admissible paths and Proposition~\ref{prop:triangle equality for admissible paths} to find a constant $C_2 $ only depending on~$\sigma$ such that $\langle \Xi_z(\xi),\Xi_z(\eta)\rangle_p\leq C_2$ for all $\xi\in\Omega^-$ and $\eta\in I^+$.
        \item \label{itembdryestiv} Use admissible paths and Propositions~\ref{prop:upper bound on PS mass} and~\ref{prop:triangle equality for admissible paths} to bound $\mu^p_z(g_u^n\Omega^+)$ and conclude the proof.
    \end{enumerate}
    The most involved part will be the last step \eqref{itembdryestiv} of the above list.

    {\bf First step  \eqref{itembdryesti}.}
    Put 
    $\ell := \ell_S(\gamma^*_0) = \ell^{\mf F}(\rho_z(g_u))$, which is bounded from above by $\sigma$ by assumption,
    and let $\omega=\sinh^{-1}\left(\tfrac1{\sinh(\ell/2)}\right)$ be the size of the collar in $S$ around $\gamma^*_0$, so that the two boundaries of the collar are in the $\epsilon_0$-thick part of $S$ for some universal constant $\epsilon_0$.

    The geodesic line $\hat \gamma^*_0$ is adjacent to two connected components $\wt S^-,\wt S^+$ of $\H^2-\wt \gamma^*$.
    Denote by $H^+,H^-$ the two closed half-planes of $\H^2$ with boundary $\hat \gamma_0^*$ and assume that
    $\tilde S^{\pm}\subset H^{\pm}$ and that $I^{\pm}\subset \partial_\infty H^{\pm}$.
    Let $x^-,x^+$ be the points lying in this order on the geodesic $(\xi_0^-,\xi_0^+)$, both at distance exactly $\omega$ from the intersection point $x$ of $(\xi_0^-,\xi_0^+)$ with $\hat\gamma^*_0$.
    In particular, $x^\pm$ projects into the $\epsilon_0$-thick part of $S$.

    Recall that for the abstract grafting surface $S_z$ there exists a natural projection map 
    $\pi_z:S_z\to S$ which is injective outside of the flat cylinders (see Definition~\ref{def:abstractgraf}).
    Lift $\pi_z$ to a $\pi_1(S)$-equivariant map $\wt\pi_z:\wt S_z\to\wt S=\H^2$, which is injective on the preimages $\wt S_z^\pm\subset \wt S_z$ of the components $\wt S^\pm$ of $\H^2-\wt\gamma^*$.
    Then $x^\pm\in \H^2$ (but not $x$) have unique preimages $\wt x^\pm\in \wt S_z^\pm$.
    The basepoint $p$ we were looking for is $p=\wt Q_z(\tilde x^-)$.

    {\bf Second step  \eqref{itembdryestii}.}
    Pulling the Patterson Sullivan measure $\mu_p$ based at $p$ for the action of $\rho_z(\pi_1(S))$
    back to a measure $\mu_z^{\tilde x^-}$ on $\partial_\infty \H^2$, 
    this is an immediate application of Proposition~\ref{prop:upper bound on PS mass} (and~\ref{remark:upper bound on PS mass}), which says that $\mu^{\wt x^-}_z(\de_\infty\H^2)$ is bounded from above by a constant depending only on $\sigma$.

    {\bf Third step  \eqref{itembdryestiii}.}
    Let $\xi\in\Omega^-$ and $\eta\in I^+$.
    There is a unique bi-infinite admissible path $a\colon\R\to\H^2$ from $\xi$ to $\eta$, which 
    is a lift of an admissible path in the hyperbolic surface $S$, defined with respect to the multicurve $\gamma^*$.
    Recall that $a$ is a concatenation of geodesic pieces, alternating between arcs contained in $\wt\gamma^*$, called flat-type, and geodesic arcs with endpoints on $\wt \gamma^*$ and orthogonal to $\wt\gamma^*$, called hyperbolic-type.

    Up to parameterization of the flat pieces,
    $a$ is the image under $\pi_z$ of a unique admissible path $\wt a\colon\R\to \wt S_z\subset \X$ (see \cite{BHM25}) for details). 
    By Lemma~\ref{fact:hitchin cv to int of weyl}, $\tilde Q_z(\tilde a(t))$ accumulates as $t\to\infty$ (\resp $t\to-\infty$) in the interior of the Weyl Chamber $\Xi_z(\eta)$ (\resp $\Xi_z(\xi)$).
    
    Using the map $\tilde Q_z$, which is a $\pi_1(S)$-equivariant embedding of 
    $\tilde S_z$ onto the universal covering of the characteristic surface of $\rho_z(\pi_1(S))$,
    pull the Finsler distance $d^{\mf F}$ back to $\tilde S_z$ and denote this distance by the same 
    symbol. With this notation and 
    by definition of the Gromov product (see \eqref{gromovproduct}), we have 
    $$\langle \Xi_z(\eta),\Xi_z(\xi)\rangle_{x^-} = 
    \lim_{T\to\infty}\frac{1}{2}\left(d^{\mf F}(\wt a(-T),\tilde x^-)+
    d^{\mf F}(\tilde x^-,\wt a(T))-d^{\mf F}(\wt a(-T),\wt a(T))\right).$$
    
    By Proposition~\ref{prop:triangle equality for admissible paths}, the path 
    $\tilde Q_z(\wt a)$ is $C_2$-quasi-ruled for some constant $C_2$ depending only on $\sigma$, so $\df(a(-T), \wt a(t))+\df(\wt a(t),\wt a(T))\leq \df(\wt a(-T),\wt a(T))+C_2$ for all $-T\leq t\leq T$.
    Thus, to find an upper bound on $\langle \Xi_z(\eta),\Xi_z(\xi)\rangle_{\wt x^-}$,
    it suffices to prove that there exists a number $R>0$ only depending on $\sigma$ such that 
    $\tilde a$ intersects the ball of radius $R$ about $\tilde x^-$. 
    
    As $a$ intersects $\hat \gamma_0^*$, it contains a (possibly degenerate) piece of flat type which 
    is a subarc of $\hat \gamma_0*$. Choose a parameterization of $a$ so that 
    $a(0)\in \hat \gamma_0^*$ is the starting point of this segment. Then the piecewise geodesic ray 
    $a\vert_{(-\infty,0]}:(-\infty,0]\to H^-$ ends on $\hat\gamma^*_0$ 
    with a hyperbolic-type geodesic piece of length at least $\omega$ (the collar size).
    By the definition of $\Omega^-$, 
    the shortest distance projection of $\xi$ to $\hat\gamma^*_0$ is at distance at most $\ell$ from 
    the shortest distance projection $x$ of $x^-$. 
    The constant $R$ we are looking for is provided by the following lemma, whose proof (which relies on hyperbolic trigonometry) is postponed until after this proof.

\begin{lemma}\label{obs:bdryesttech}
    For any $\sigma>0$ there is $R>0$ such that the following holds.
    Let $0<\ell\leq \sigma$ and let $\omega=\sinh^{-1}\left(\tfrac1{\sinh(\ell/2)}\right)>0$ be the collar size associated to $\ell$ by the hyperbolic collar lemma.

    Let $\mc L\subset\H^2$ be a line and $a:[0,\infty)\to\H^2$ an admissible path starting on $\mc L$ orthogonally to it, and with a hyperbolic-type piece of length at least $\omega$.
    Suppose $a(t)$ tends as $t\to\infty$ to  $\xi\in\partial\H^2$ whose orthogonal projection is $\ell$-close to the starting point of a ray $r:[0,\infty)\to\H^2$ orthogonal to $\mc L$ and in the same half-plane as $a$.
    Then
    $d_{\H^2}(a(\omega),r(\omega))\leq R$.
\end{lemma}

    Lemma~\ref{obs:bdryesttech} exactly tells us that the point $a(-\omega)$ is contained in the 
    ball of radius $R$ about $x^-$.
    It remains to check that the distance between $\wt a(-\omega)$ and $\tilde x^-$ is at most $R$ as well.
    This holds true because $a(-\omega)$ and $x^-$ are contained in 
    $\wt S^-$, so their preimages $\wt a(-\omega)$ and $\wt x^-$ are 
    contained in the same hyperbolic piece $\wt S^-_z\subset \wt S_z$. As this piece 
    is isometrically embedded in $\tilde S_z$ and the Finsler distance 
    $d^{\mf F}$ is not larger than the path distance on the grafted surface, this
    completes the distance estimate.

    {\bf Fourth step  \eqref{itembdryestiv}.}
    This part of the proof is the longest and most involved.
    By equivariance, we have 
    $$\mu^{\wt x^-}_z(\xi)=e^{\delta(z) b^{\mf F}_{\Xi_z(\xi)}(\wt x^+,\wt x^-)}d\mu^{\wt x^+}_z(\xi)$$
    where $\delta(z)$ is the critical exponent of $\rho_z$ and $b^{\mf F}_\eta(q,q')$ is the Busemann function of $(q,q')$ based at $\eta\in\mc F$ (for the Finsler metric), see Section~\ref{sec:lietheoryI}.
    
       Since $\mu_z^{\wt x^+}(\partial_\infty \H^2)\leq C_1$ for a constant $C_1>0$ only depending on $\sigma$ by Proposition~\ref{prop:upper bound on PS mass}, to get the desired upper bound on $\mu_z^{\wt x^-}(g_{u}^{n}\Omega^+)$ it suffices for $\xi^{+}\in g_{u}^{n}\Omega^+$  to bound from  above the Busemann function $b_{\Xi_z}^{\mf F}(g_{u}^{n}\wt x^+,\wt x^-)$.
    For this we use the admissible path from $\wt x^-$ to $\Xi_z(\xi^{+})$, which is quasi-ruled and passes near $g_{u}^{n}\wt x^+$, and we use our knowledge of the lengths of the pieces of admissible paths.
    We will see that the Busemann function is roughly $-\max(L,\vert n-n_{0}\vert\ell)-2\omega$ for some $n_{0}$ independent of $n$.
    We will then be able to conclude our estimate of $\mu_z^{\wt x^-}(I^+)$ by computing
     \begin{align*} 
\hat\nu(z)(\Omega^-\times I^+) 
&\leq \mu_z^p(\Omega^-)\cdot 
\max_{(\xi,\eta)\in\Omega^-\times I^+}
(e^{\delta(z) \langle\Xi_z(\xi)|\Xi_z(\eta)\rangle_p})\cdot
 \sum_n\mu^p_z(g_u^n\Omega^+)\\
&\leq \mr{Cst}e^{-\delta(z)\max(L,\vert n-n_{0}\vert\ell)-2\delta(z)\omega}
\end{align*}
    
    Fix $\xi\in g_u^n\Omega^+\subset I^+$.
    There exists a unique admissible path 
    $a_\xi:[0,+\infty)\to\wt S=\H^2$ from $x^-$ to $\xi$ (lifting an admissible path of $S$), and it is the image under $\wt\pi_z$ of a unique admissible path $\wt a_\xi:[0,\infty)\to \wt S_z\subset\X$ that starts at $\wt x^-$ and accumulates in the interior of the simplex $\Xi_z(\xi)$ by Lemma~\ref{fact:hitchin cv to int of weyl}.
    By the definition of Finsler Busemann cocycles (see Section~\ref{sec:lietheoryI}), this means that  we have
    \begin{equation}\label{eq:bdryest bus1}b^{\mf F}_{\Xi_z(\xi)}(\rho_z(g_u^n)\wt x^+,\wt x^-)=
    \lim_{T\to\infty}d^{\mf F}(\rho_z(g_u^n)\wt x^+,\wt a_\xi(T))-d^{\mf F}(\wt x^-,\wt a_\xi(T)).\end{equation}
    
    Notice that the third geodesic piece of $\wt a_\xi$ (the one that leaves the flat strip $\hat \gamma^*_0$) is the isometric image by $g_u^n$ of an admissible path going from $\hat\gamma^*_0$ to $g_u^{-n}\xi\in\Omega^+$. And therefore $\wt a_\xi$ passes within distance $R$ of $\rho_z(g_u^n)\wt x^+$ at some time $t$. 

    By Proposition~\ref{prop:triangle equality for admissible paths}, $\wt a_\xi$ is $C_2$-quasi-ruled 
    (and starts at $\wt x^-$) so \[d^{\mf F}(\wt a_\xi(t),\wt a_\xi(T))-d^{\mf F}(\wt x^-,\wt a_\xi(T))\leq - 
    d^{\mf F}(\wt x^-,\wt a_\xi(t))+C_2\text{ for any }T\geq t.\]
    This, combined with $d^{\mf F}(\rho_z(g_u^n)\wt x^+,\wt a_\xi(t))\leq R$ and \eqref{eq:bdryest bus1} yields:
    \begin{equation}\label{eq:bdryest bus2}b^{\mf F}_{\Xi_z(\xi)}(\rho_z(g_u^n)\wt x^+,\wt x^-)\leq - d^{\mf F}(\rho_z(g_u^n)\wt x^+,\wt x^-)+C_2+2R.\end{equation}

    We now need to estimate $d^{\mf F}(\rho_z(g_u^n)\wt x^+,\wt x^-)$, and we also do this using that the admissible path from $\wt x^-$ to $\rho_z(g_u^n)\wt x^+$ is quasi-ruled, except that this time this path is completely explicit.
    The unique admissible path $c$ in $\H^2$ from $x^-$ to $g_u^nx^+$ has three geodesic pieces: first the geodesic from $x^-$ to $x$, of length $\omega$, then the geodesic from $x$ to $g_u^nx$, of length $|n|\ell$, and finally the geodesic from $g_u^nx$ to  $g_u^nx^+$, of length $\omega$.
    It's image under $\wt\pi_z$ is the unique admissible path $\wt c$ from $\wt x^-$ to $\rho_z(g_u^n)\wt x^+$, which is also made of three explicit geodesic pieces.
    The first and last pieces are just translates of the corresponding pieces of~$c$, and hence have length $\omega$.
    
    The middle piece, however, is more complicated because instead of sliding along $\hat\gamma_0^*$ like~$c$, we are navigating in a flat strip that lifts the flat cylinder above $\gamma_0^*\subset\gamma^*$, and we must move diagonally in this flat strip to realise at the same time the horizontal translation prescribed by the middle piece of $c$ and the vertical translation prescribed by the grafting parameter $z$.
    Let $z_0\in\mf a$ be the coordinate of $z$ associated to the component $\gamma_0^*\subset\gamma^*$.
    Then the above mentioned flat strip is conjugate to the strip $\{tv_0\pm sz_0:t\in \R,\ s\in[0,1]\}$ where $v_0=d\tau\hsl$ is the special direction of $\mf a^+$ and the sign $\pm$ depends on the choice of orientation on $\gamma_0^*$ (see Section~\ref{sec:Abstract grafting}).
    Since we are moving horizontally a distance $|n|\ell$, the middle piece of $\wt c$ is conjugate in this strip to the geodesic segment from $0$ to $n\ell v_0\pm z_0$.
    As a consequence, the length of this middle piece is exactly $\mf F(n\ell v_0\pm z_0)$ where $\mf F$ is the norm on $\mf a$ defined in Equation~\eqref{eq:mfF}.
    Finally, using again that $\wt c$ is $C_2$-quasi-ruled (Proposition~\ref{prop:triangle equality for admissible paths}), we get
    \begin{equation}\label{eq:bdryest bus3} d^{\mf F}(\wt x^-,\rho_z(g_u^n)\wt x^+) \geq  2\omega + \mf F(z_0 \pm n\ell v_0) -2C_2.\end{equation}

    Now we must estimate $\mf F(z_0 \pm n\ell v_0)$.
    By the assumption,  the height of the cylinder at $\wt\gamma_0^*$ is $\min_{t\in\R}\mf F(z_0+tv_0)\geq L$ (see \eqref{eq:cylinder height}).
    Let $t_0$ be the unique point of $\R$ such that $z_0+t_0v_0\in \ker\alpha_0$ ($\alpha_0$ is the linear form which is equal to the Finsler norm in the Weyl cone that contains $w_0$).
    Then 
    $$\mf F(z_0+tv_0) \geq|\alpha_0(z_0+tv_0)| = |\alpha_0((t-t_0)v_0)|=|t-t_0|$$
    for any $t\in\R$.
    Let $n_0$ be the integer closest to $t_0/\ell$, so that
    $\mf F(z_0+n\ell v_0)\geq |n-n_0|\ell - \ell$, which is bounded below by $|n-n_0|\ell - \sigma$.
    Combining this with \eqref{eq:bdryest bus2} and \eqref{eq:bdryest bus3} we get
    $$b^{\mf F}_{\Xi_z(\xi)}(\rho_z(g_u^n)\wt x^+,\wt x^-)\leq -  2\omega-\max(|n_0\pm n|\ell, L) +C_2+2R+\sigma +2C_2.$$
    (Recall that $\pm$ is just some fixed sign depending on the choice of orientation of $\gamma_0^*$.)
    
    Recall the quasi-invariance of Patterson--Sullivan measures: 
    $$\mu_z^{\wt x}(g_u^n\Omega^+)=\int_{\xi\in g_u^n\Omega^+}e^{\delta(z)b^{\mf F}_{\Xi_z(\xi)}(\rho_z(g_u^n)\wt y,\wt x)}d\mu_z^{\rho_z(g_u^n)\wt y}(\xi).$$
    Since $\mu_z^{\rho_z(g_u^n)\wt y}(\partial\H^2)=\mu_z^{\wt y}(\partial\H^2)$ is bounded from
    above by some constant $C_1$ that only depends on $\sigma$ by Proposition~\ref{prop:upper bound on PS mass}, and 
    $\delta(z)\leq m$ for some constant $m$ depending only on $\alpha_0$  by Lemma~\ref{entropy},
    we get 
    \begin{align*}\mu^{\wt x}_z(\rho_z(g_u^n)\Omega^+)&\leq e^{m(2R+3C_2+\sigma)}C_1e^{-\delta(z)\max(|n_0\pm n|\ell,L)}e^{-2\delta(z)\omega}\\&=C_3e^{-\delta(z)\max(|n_0\pm n|\ell,L)}e^{-2\delta(z)\omega},\end{align*}
    where $C_3$ only depends on $\sigma$.
    After some computations, and using that (for $\alpha=\delta(z)\ell$ and $\beta=L/\ell$)
\[e^{-\omega}\leq \ell\cosh(\sigma)\quad\text{and}\quad\sum_{n\geq \beta}e^{-\alpha n}\leq\frac{e^{-\alpha\beta}}{\alpha},\]
we get
    \begin{align*}
\mu_z^{\wt x}(I^+)
&\leq C_3e^{-2\delta(z)\omega} \sum_ne^{-\delta(z)\max(|n_0\pm n|\ell,L)} \\
&\leq 2C_3e^{-2\delta(z)\omega}\left( \sum_{0\leq n< L/\ell}e^{-\delta(z)L}e^{-2\delta(z)\omega}+\sum_{n\geq L/\ell}e^{-\delta(z)n\ell}\right)\\
& \leq 2C_3\left(\frac{L}{\ell}+\frac{1}{\delta(z) \ell}\right)e^{-\delta(z)L}(e^{-\omega})^{2\delta(z)}\\
    &\leq 2C_3\left(\frac1\ell+\frac{1}{\delta(z)\ell}\right)(L+1)e^{-{\delta(z)}L}\ell^{2\delta(z)}\cosh(\sigma)^{2m}\\
    &\leq C_4\max(\ell^{2\delta(z)-1},1)(L+1)e^{-{\delta(z)}L}.
    \end{align*}
    To obtain $C_4$ only depending on $\sigma$, we use that $\delta(z)\geq\delta$ and that $\delta$ is bounded 
    from below by a constant that only depends on $\sigma$, by Proposition~\ref{prop:entropy of hypsurf}.

    By Proposition~\ref{prop:entropy of hypsurf}, there exists $\epsilon_\sigma\leq 1$ such that if $\ell\leq \epsilon_\sigma$ then $\delta>\tfrac12$.
    Thus, if on one hand $2\delta(z)-1\geq0$ then $\ell^{2\delta(z)-1}\leq \sigma^{2\delta(z)-1}\leq \sigma^{2m-1}$.
    On the other hand, if $2\delta(z)-1<0$ then we must have $\ell>\epsilon_\sigma$ so $\ell^{2\delta(z)-1}\leq \epsilon_\sigma^{2\delta(z)-1}\leq \epsilon_\sigma^{2m-1}$.
    In any case $\ell^{2\delta(z)-1}$ is bounded above by a constant that only depends on $\sigma$, which concludes the proof.
\end{proof}

We now prove the technical estimate we used in the proof.
\begin{proof}[Proof of Lemma~\ref{obs:bdryesttech}]
Parallel transport $\mc L$ along the first geodesic piece of $a$ until time~$\omega$, to obtain $\mc L'$ at distance $\omega$ from $\mc L$.
Let $H$ be the half-plane delimited by $\mc L'$ that does not contain $\mc L$.
Then by definition of admissible path one can check that $a(t)\in H$ for any~$t\geq \omega$.

By a classical formula of hyperbolic trigonometry, see Theorem 7.17.1 of~\cite{Bea83}, the orthogonal projection of any $x\in H$ is at distance at most $\sinh^{-1}(\tfrac1{\sinh(\omega)})$ from $a(0)$.

In particular the orthogonal projection of $\xi$ is at distance at most $\ell':=\sinh^{-1}\left(\tfrac1{\sinh(\omega)}\right)$ from $a(0)$, and by triangle inequality $a$ and $r$ start at distance at most $\ell+\ell'$.
Using for instance again Theorem 7.17.1 of~\cite{Bea83}, one can check that $a(\omega)$ and $r(\omega)$ are at distance at most twice the following:
$$\sinh^{-1}(\sinh(\ell/2)\cosh(\omega))+\sinh^{-1}(\sinh(\ell'/2)\cosh(\omega))\leq 2\sinh^{-1}\left(\frac{\cosh\omega}{\sinh\omega}\right),$$
which can be bounded above in terms of $\sigma$ because $\omega$ can be bounded below in terms of~$\sigma$ (since $\ell\leq \sigma$).
\end{proof}

\subsection{Convergence of currents}

Recall $\delta >0$ is the topological entropy of $\Phi^t_{\vert K}$.
The following is the main result of this section.

\begin{proposition}\label{escapeofmass2}
Let $L_i\to \infty$ and let $\rho_i=\rho_{z_i}$ be a sequence of Hitchin  representations 
obtained by Hitchin grafting of a Fuchsian representation 
at the simple geodesic multicurve $\gamma^*$ with cylinder heights bounded from below by $L_i$. 
Then 
$\delta(z_i)\to \delta$, and up to passing to a subsequence, 
the equilibrium measures $\nu_i=\nu(z_i)$ converge weakly to a measure of maximal entropy for 
$\Phi^t\vert K$.
\end{proposition}
\begin{proof}
Recall that $f_i=f_{z_i}$ denotes a positive 
H\"older continuous potential on $T^1S$ whose periods are the Finsler translation lengths of 
the elements of $\rho_i(\pi_1(S))$.

Up to passing to a subsequence, we may assume
that the $\Phi^t$-invariant probability measures 
  $\nu_i^1=\nu_i/||\nu_i||$ converges weakly to a $\Phi^t$-invariant probability measure $\nu$
  on $T^1S$. By Lemma~\ref{lem:total mass of nu}, we may also assume that the 
  geodesic currents $\hat \nu(z)$ converge weakly to a current $\hat \nu$ which is a positive
  multiple of the current defined by $\nu$. 
  
  By Proposition~\ref{boundaryestimate2}, we have $\iota(\hat \nu, \gamma^*)=0$ and hence the 
  limit measure $\nu$ must be supported on $K$.
  By Lemma~\ref{entropysplit}, we are thus left with showing that $h_\nu\geq \delta$.

From Lemma~\ref{lem:total mass of nu} we have $\delta(z_i)\in (\delta,m]$ for any $i$.
Recall from \eqref{entropyeq}  that 
\begin{equation}\label{integralentropy}
h_{\nu_i^1}=\delta(z_i)\int f_id\nu_i^1.\end{equation}

By Theorem~\ref{thm:quasiisom lengths},
it holds 
$$
\int f_i\frac{d\eta}{\ell_S(\eta)} \geq \left( 1+\frac{C}{L_i+1} \right)^{-1}
$$
for any $\eta\in\pi_1(S)$, and hence since the $\Phi^t$-invariant Borel probability measures
supported on closed geodesics are weak$^*$-dense in the space of all 
$\Phi^t$-invariant Borel probability measures, we get
$$
\int f_id\nu_i^1 \geq \left( 1+\frac{C}{L_i+1} \right)^{-1},
$$
and hence 
$$
\lim\inf_{i\to \infty} h_{\nu_i}\geq \lim\inf_{i\to \infty}\delta(z_i)\geq \delta.
$$

Since the entropy function is lower semi-continuous, 
we conclude that $h_\nu\geq \delta$.
As $\nu$ is supported in $K$, this 
implies that indeed, $\nu$ is a measure of maximal entropy for the
restriction of $\Phi^t$ to $K$ by Lemma~\ref{entropysplit}.
\end{proof}

Using the above results we are now ready to complete the proof of Theorem~\ref{main1}
from the introduction.

\begin{proof}[Proof of Theorem~\ref{main1}]
Part (3) of Theorem~\ref{main1} was shown in Section~\ref{QI-proof}, so we are left with showing part (1) and (2). 
Let $\gamma^*\subset S$ be a pair of pants decomposition of $S_1=S-S_0$ that contains $\partial S_0=\partial S_1$.
The metric $h$ on $S_0$ prescribes lengths for the components of~$\gamma^*$ in $\partial S_0$.

Since no component of $S_1$ is a pair of pants, every pair of pants in $S_1-\gamma^*$ has a boundary component in $\gamma^*-\partial S_1$.
By Proposition~\ref{lem:entropy-pants-convergence}, one can choose lengths large enough 
for  each component of $\gamma^*-\partial S_1$ such that each pair 
of pants of $S_1-\gamma^*$ has entropy very close to zero, and in particular strictly smaller than the entropy of $S_0$.

Then by Hitchin grafting along $\gamma^*$ flat cylinders with bigger and bigger heights,
we get a sequence $\rho_i=\rho_{z_i}$ of Hitchin representations satisfying the first two statement of  Theorem~\ref{main1}, according to Proposition~\ref{escapeofmass2}.
\end{proof}

\section{Pressure length control}\label{sec:pressurelength}

Define the \emph{entropy gap} of the pair consisting of a hyperbolic surface
and a separating 
simple closed geodesic to be the absolute value of the difference 
between the entropies of the two components of $S-\gamma^*$.
If $\gamma^*$ is non-separating
then we define the entropy gap to be one.

Consider a path $t\to \rho_{tz}$ of Hitchin representations obtained by 
Hitchin grafting along a single geodesic $\gamma$ and grafting parameters
a ray in $\mathfrak{a}$ with direction in the kernel of the linear functional
defining the Finsler length of $\gamma^*$.
The first goal of this section is to show

\begin{theorem}\label{pressurelength}
The pressure metric length of the path $t\to \rho_{tz}$ is finite and bounded from 
above by a constant 
only depending on $z$,
and an upper bound for the length of the grafting geodesic $\gamma^*$. 
\end{theorem}

We also show

\begin{theorem}\label{pressurelength2}
Consider a subsurface $S_1\subset S$ with $\partial S_1\subset\gamma^*$.
    Let $(\rho^{t})_{t\in[0,T]}$ be a path of hyperbolic structures on $S$ obtained by concatenating shearing paths along multicurves contained in $S_1$, such that for any $t\in[0,T]$ the entropy of 
    the geodesic flow on $T^1S_1$  and the restriction of the metric 
    $\rho_t$ is strictly smaller than the entropy of the geodesic flow on $T^1S_0$ for $S_0=S-S_1$ 
    (which does not depend on $t$).
    Denote by $\rho_{z}^{t}$ the Hitchin grafting of $\rho^{t}$ alogn $\gamma^*$ with parameter $z$.

    Then the pressure length of the smooth path $(\rho_{z}^{t})_{t\in[0,T]}$ tends to zero as the cylinder height associated to $z$ tends to infinity.
\end{theorem}

This section is subdivided into three subsections. 
In the first subsection we recall the definition of the pressure length of a path in the Hitchin component, and we give an upper bound purely in terms of the nonnormalized intersection form ${\bf I}(\rho_{1},\rho_{2})$ (see Section~\ref{sec:geodescicurrents}).
In the second subsection, we give upper bounds for the derivatives of this intersection form along the two paths in the above theorems, and use this to conclude.

\subsection{A general upper bound for pressure lengths}

Let $[\rho_t]_{a\leq t\leq b}$ be a smooth path in the Hitchin component.
By definition, its length for the pressure metric is
\[\int_a^b \left(\frac{d^2}{ds^2}{\bf J}(\rho_t,\rho_{t+s})\vert_{s=0}\right)^{\tfrac12}dt,\]
where, denoting by $\delta(t)$ the entropy of $\rho_t$, we have ${\bf J}(\rho_t,\rho_{t+s})=\tfrac{\delta(t+s)}{\delta(t)}{\bf I}(\rho_t,\rho_{t+s})$ and ${\bf I}(\rho_t,\rho_{t+s})$ is the nonnormalized intersection form.
We want an upper bound for this length in terms of ${\bf I}(\rho_t,\rho_{t+s})$ and its derivatives.
This is possible thanks to the following classical lemma.

\begin{lemma}\label{lem:entropy derivative}
$\delta'(t)=-\delta(t) \frac{d}{ds}{\bf I}(\rho_t,\rho_{t+s})\vert_{s=0}$.
\end{lemma}
\begin{proof}
    We fix a hyperbolic structure on $S$ and use its unit tangent bundle $T^1S$, equipped with the geodesic flow, as the underlying
    phase space for all computations.
    Let $f_t$ be a reparametrisation function associated with $\rho_t$.
    Recall that $P(-\delta(t)f_t) = 0$, where $P$ is the pressure function (see Section~\ref{sec:geodescicurrents}).
    We are going to differentiate this equality, using the fact due to Parry--Pollicott, see Propositions 4.10-11 of~\cite{ParryPollicott}, and Ruelle~\cite{Ruelle}, that for any $\mc C^1$ one-parameter family of H\"older functions $(g_t)_t$ we have $\frac d{dt}P(g_t)=\int (\frac d{dt}g_t)d\mu_t$ where $\mu_t$ is the equilibrium state associated to $g_t$.

    Let $\nu_t$ be the equilibrium state associated with $-\delta(t)f_t$.
Then
    $$
    0=\int \left(\delta'(t)f_t+\delta(t)\left( \frac d{dt}f_t \right)\right)d\nu_t
    $$

Now recall that
    \[{\bf I}(\rho_t,\rho_{t+s})=\frac{\int f_{t+s}d\nu_t}{\int f_td\nu_t} \quad \text{and} \quad \frac{d}{ds}{\bf I}(\rho_t,\rho_{t+s})\vert_{s=0} = \frac{ \int \left(\frac d{dt}f_t\right)d\nu_t}{\int f_td\nu_t},\]
which concludes the proof.
\end{proof}

We can now prove the following upper bound for the pressure length of $[\rho_t]_{a\leq t\leq b}$.

\begin{proposition}\label{prop:general bound on pressure lengths}
An upper bound for $\int_a^b \left(\frac{d^2}{ds^2}{\bf J}(\rho_t,\rho_{t+s})\vert_{s=0}\right)^{\tfrac12}dt$ is
\[ \sqrt {b-a} \left(  - \frac{d}{ds}{\bf I}(\rho_b,\rho_{b+s})\vert_{s=0} +\frac{d}{ds}{\bf I}(\rho_a,\rho_{a+s})\vert_{s=0} + \int_a^b \frac{d^2}{ds^2}{\bf I}(\rho_t,\rho_{t+s})\vert_{s=0}dt\right)^{\frac12}.\]
\end{proposition}
\begin{proof}
Let us differentiate twice the formula ${\bf J}(\rho_t,\rho_{t+s})=\tfrac{\delta(t+s)}{\delta(t)}{\bf I}(\rho_t,\rho_{t+s})$.
We get
\[\frac{d^2}{ds^2} {\bf J}(\rho_t,\rho_{t+s})\vert_{s=0} =\frac{\delta''(t)}{\delta(t)} 
    + 2\frac{\delta'(t)}{\delta(t)}\frac{d}{ds}{\bf I}(\rho_t,\rho_{t+s})\vert_{s=0} + 
    \frac{d^2}{ds^2}{\bf I}(\rho_t,\rho_{t+s})\vert_{s=0}\]
Using
$\tfrac{\delta^{\prime\prime}}{\delta}=(\log \delta)'' +(\tfrac{\delta^\prime}{\delta})^2$ and $\tfrac{\delta^\prime(t)}{\delta(t)}=-\frac{d}{ds}{\bf I}(\rho_t,\rho_{t+s})\vert_{s=0}$ from Lemma~\ref{lem:entropy derivative}, 
we get 
\begin{align*}
\frac{d^2}{ds^2} {\bf J}(\rho_t,\rho_{t+s})\vert_{s=0} 
&=(\log \delta)''(t) - \left(\tfrac{\delta^\prime(t)}{\delta(t)}\right)^2 + \frac{d^2}{ds^2}{\bf I}(\rho_t,\rho_{t+s})\vert_{s=0}\\
&\leq (\log \delta)''(t) + \frac{d^2}{ds^2}{\bf I}(\rho_t,\rho_{t+s})\vert_{s=0}
\end{align*}
We now conclude:
\begin{align*}
\int_a^b &\left(\frac{d^2}{ds^2}{\bf J}(\rho_t,\rho_{t+s})\vert_{s=0}\right)^{\tfrac12}dt
 \leq \sqrt{b-a} \left(\int_a^b\frac{d^2}{ds^2}{\bf J}(\rho_t,\rho_{t+s})\vert_{s=0}dt\right)^{\tfrac12}\\
&\leq \sqrt{b-a} \left(\int_a^b (\log \delta)''(t) dt + \int_a^b\frac{d^2}{ds^2}{\bf I}(\rho_t,\rho_{t+s})\vert_{s=0}dt\right)^{\tfrac12}\\
& \leq \sqrt{b-a} \left((\log \delta)'(b)-(\log \delta)'(a) + \int_a^b\frac{d^2}{ds^2}{\bf I}(\rho_t,\rho_{t+s})\vert_{s=0} dt\right)^{\tfrac12}\\
&\leq  \sqrt {b-a} \left(  - \frac{d}{ds}{\bf I}(\rho_b,\rho_{b+s})\vert_{s=0} +\frac{d}{ds}{\bf I}(\rho_a,\rho_{a+s})\vert_{s=0} + \int_a^b \frac{d^2}{ds^2}{\bf I}(\rho_t,\rho_{t+s})\vert_{s=0}dt\right)^{\frac12}.
\end{align*}
\end{proof}

\subsection{Proofs of Theorems~\ref{pressurelength} and \ref{pressurelength2}}\label{sec:secondderivativelength}

First we give an upper bound of the derivatives of the intersection form in the case of the grafting path from Theorem~\ref{pressurelength}.

\begin{proposition}\label{firstderivative}
Let $(\rho_{t}=\rho_{tz})_{t\geq 0}$ be a grafting path as in Theorem~\ref{pressurelength}.
Then there exist numbers $\kappa >0$ and $C>0$ only depending on $z$ and an upper bound $\sigma$ for the length of $\gamma$ such that
\[ \left\vert \frac{d}{ds}{\bf I}(\rho_{t},\rho_{t+s})\vert_{s=0}\right \vert \leq Ce^{-\kappa t} \, \text{ and } \,
\frac{d^{2}}{ds^{2}}{\bf I}(\rho_{t},\rho_{t+s})\vert_{s=0} \leq Ce^{-\kappa t}.\]
\end{proposition}
\begin{proof}
Resuming the notations from Section~\ref{sec:hitchinrep}, 
Proposition~\ref{hoelderfunction} shows that 
a Hitchin grafting path $t\to \rho_{tz}$ in the Hitchin component gives
rise to a real analytic family $f_t:T^1S\to (0,\infty)$ of H\"older functions
defining a reparameterization of the geodesic flow on $T^1S$ corresponding to 
the Finsler length of $\rho_{tz}(\pi_1(S))$. 

For each $t$ let  $\nu(t)$ be the Gibbs equilibrium state of $-\delta(t)f_t$
where $\delta(t)>0$ is such that the pressure of $-\delta(t)f_t$ vanishes. By our convention,
$\nu(t)$ is a probability measure on $T^1S$ which is invariant under the geodesic flow $\Phi^t$ on $S$.
For $\eta\in \pi_1(S)$ we also put $f_t(\eta)=\int_\eta f_t$, the Finsler translation length of the 
element $\rho_{tz}(\eta)$.

To bound the derivatives of the intersection form ${\bf I}$ we will use the following formula (see Section~\ref{sec:geodescicurrents})
\[{\bf I}(\rho_t,\rho_{t+s})=\frac{\int f_{t+s}d\nu_t}{\int f_td\nu_t}=\int f_{t+s}d\nu_t,\]
which is easier to differentiate with respect to $s$.

We know 
that
\[\nu(t) = \lim_{T\to \infty}\frac{1}{\myhash  N_{f_{t}}(T)} \sum_{\eta\in N_{f_{t}}(T)} \frac{{\mc D}_{\eta}}{f_t(\eta)},\]
where ${\mc D}_{\eta}$ is the flow-invariant measure supported on the periodic orbit $\eta$ (in our fixed hyperbolic structure).
Exchanging derivatives and integrals, we get
\begin{align*}
\frac{d}{ds}\int f_{t+s} d\nu(t)\vert_{s=0} 
&=\int \frac{d}{ds}f_{t+s} \vert_{s=0} d\nu(t) \\
&=\lim_{T\to \infty} \frac{1}{\myhash  N_{f_{t}}(T)} \sum_{\eta\in N_{f_{t}}(T)} \frac{\int \frac{d}{ds} f_{t+s}\vert_{s=0}  d\mc D_{\eta}}{f_{t}(\eta)}\\
&=\lim_{T\to \infty} \frac{1}{\myhash  N_{f_{t}}(T)} \sum_{\eta\in N_{f_{t}}(T)} \frac{\frac{d}{ds}f_{t+s}(\eta)\vert_{s=0} }{f_{t}(\eta)}
\end{align*}

By Theorem~\ref{cor:firstandsecond} we know 
\[\left\vert \frac{d}{ds}f_{t+s}(\eta)\vert_{s=0}\right\vert\leq A_{\epsilon,\sigma} \iota(\eta,\gamma^*)\]
and hence by the continuity of $\iota$ and 
H\"older continuity of the function $\frac{d}{ds}f_{t+s}$,
we conclude that 
\[\left\vert \frac{d}{ds}\int f_{t+s}d\nu(t)\vert_{s=0}\right\vert \leq A_{\epsilon,\sigma} \iota(\nu(t),\gamma^*).\]
By Proposition~\ref{boundaryestimate2} we have
\[ \left\vert \frac{d}{ds}\int f_{t+s} d\nu(t)\vert_{s=0}\right \vert \leq Ce^{-\kappa t}.\]
With a similar argument, which is omitted here, one proves
\[ \frac{d^2}{ds^2}\int f_{t+s} d\nu(t)\vert_{s=0}  \leq Ce^{-\kappa t}.\qedhere\]
\end{proof}

We can now conclude the proof of Theorem~\ref{pressurelength}.
\begin{proof}[Proof of Theorem~\ref{pressurelength}]
Using Propositions~\ref{prop:general bound on pressure lengths} and \ref{firstderivative} we have the following estimates on the pressure length of $(\rho_{t}=\rho_{tz})_{t\geq 0}$.
\begin{align*}
&\int_0^{\infty} \left(\frac{d^2}{ds^2}{\bf J}(\rho_t,\rho_{t+s})\vert_{s=0}\right)^{\tfrac12}dt
=\sum_{m\geq0}\int_m^{m+1} \left(\frac{d^2}{ds^2}{\bf J}(\rho_t,\rho_{t+s})\vert_{s=0}\right)^{\tfrac12}dt\\
&\leq\sum_{m\geq0} \left(  \frac{d}{ds}{\bf I}(\rho_{m+1},\rho_{{m+1}+s})\vert_{s=0} - \frac{d}{ds}{\bf I}(\rho_m,\rho_{m+s})\vert_{s=0} + \int_m^{m+1} \frac{d^2}{ds^2}{\bf I}(\rho_t,\rho_{t+s})\vert_{s=0}dt\right)^{\frac12}\\
&\leq\sum_{m\geq0} \left(   3Ce^{-\kappa m}\right)^{\frac12}\\
&\leq \frac{\sqrt {3C}}{1-e^{-\kappa/2}}. \qedhere
\end{align*}
\end{proof}

We now turn to the proof of Theorem~\ref{pressurelength2}, which is very similar, except that it uses Corollary~\ref{cor:estimate derivative length on shearings} instead of Theorem~\ref{cor:firstandsecond}. 
\begin{proof}[Proof of Theorem~\ref{pressurelength2}]
Recall that in the present setting, the geodesic $\gamma$ divides $S$ into  subsurfaces $S_0,S_1$,
the smooth path $t\to \rho^t$ $(t\in [0,T])$ in the Teichm\"uller space of marked Riemann surfaces is such that:
\begin{enumerate}
\item The restriction of the marked hyperbolic metric $\rho^t$ to the subsurface $S_0$ does not depend on $t$.
\item The entropy of the geodesic flow of $\rho^t$ restricted to the subspace of all geodesics entirely contained in $S_0$ is strictly larger than the entropy of the restriction of the flow to the subspace of all geodesics entirely contained in $S_1$.
\end{enumerate}
We denote by $\rho_{z}^{t}$ the representation obtained by grafting $\rho^{t}$ at $\gamma$ with parameter $z$.
Let $f_z^t$ be a corresponding  positive H\"older function and let $\nu_z^t$ be the corresponding Gibbs state, such that $\int f_{z}^{t}d\nu_{z}^{t}=1$ (on the hyperbolic surface associated to $\rho^{t}$).

As in the proof of Proposition~\ref{firstderivative}, denoting $f_{z}^{t}(\eta)=\int_{\eta} f_{z}^{t}$ the Finsler length of $\rho_{z}^{t}(\eta)$,  by Corollary~\ref{cor:estimate derivative length on shearings} we have
\begin{align*}
\left\vert\frac{d}{ds}{\bf I}(\rho_{z}^t,\rho_{z}^{t+s})\vert_{s=0}\right\vert
&=\left\vert\frac{d}{ds}\int f_{z}^{t+s} d\nu_{z}^t\vert_{s=0}\right\vert\\ 
&\leq\lim_{T\to \infty} \frac{1}{\myhash  N_{f_{z}^{t}}(T)} \sum_{\eta\in N_{f_{z}^{t}}(T)} \frac{\left\vert\frac{d}{ds}f_{z}^{t+s}(\eta)\vert_{s=0}\right\vert }{f_{z}^{t}(\eta)}\\
&\leq\lim_{T\to \infty} \frac{1}{\myhash  N_{f_{z}^{t}}(T)} \sum_{\eta\in N_{f_{z}^{t}}(T)} C\frac{\ell^{S_{0}}_{\rho^{t}}(\eta) }{f_{z}^{t}(\eta)}\\
&=C\lim_{T\to \infty} \ell^{S_{0}}_{\rho^{t}} \left(\frac{1}{\myhash  N_{f_{z}^{t}}(T)} \sum_{\eta\in N_{f_{z}^{t}}(T)} \frac{\mc D_{\eta}}{f_{z}^{t}(\eta)}\right)\\
&=C\ell^{S_{0}}_{\rho^{t}}(\nu_z^t),
\end{align*}
where $C$ is a constant that depends on $(\rho^{t})_{0\leq t\leq T}$,  and $\ell^{S_{0}}_{\rho^{t}}(\nu)$ is simply the mass given by $\nu$ to the set of unit tangent vectors of $S$ that are footed on $S_{0}$.
This is a linear function of $\nu$ that is continuous at $\nu$ when it gives zero measure to the set of geodesics asymptotic to $\gamma$; in particular it is continuous at $\nu_{z}^{t}$.

Similarly one can check that 
\[\frac{d^{2}}{ds^{2}}{\bf I}(\rho_{z}^t,\rho_{z}^{t+s})\vert_{s=0}\leq C\ell^{S_{0}}_{\rho^{t}}(\nu^{t}_{z})\]

We deduce from this, Proposition~\ref{boundaryestimate2} and Proposition~\ref{escapeofmass2} that for fixed $t$, the derivatives $\left\vert\frac{d}{ds}{\bf I}(\rho_{z}^t,\rho_{z}^{t+s})\vert_{s=0}\right\vert$ and $\frac{d^{2}}{ds^{2}}{\bf I}(\rho_{z}^t,\rho_{z}^{t+s})\vert_{s=0}$ converge to zero when the cylinder height associated to $z$ goes to infinity.
Moreover these quantities are bounded above by a constant which depends on the whole path $(\rho^{s})_{0\leq s\leq T}$ but not on $t$ (notice that $\ell^{S_{0}}_{\rho^{t}} (\nu^{t}_{z})$ is bounded above by the total mass of $\nu^{t}_{z}$, which is bounded above by Lemma~\ref{lem:total mass of nu}).
By the dominated convergence theorem, we conclude that 
\[\int_{0}^{T}\frac{d^{2}}{ds^{2}}{\bf I}(\rho_{z}^t,\rho_{z}^{t+s})\vert_{s=0}dt\underset{L\to\infty}{\longrightarrow}0,\]
where $L$ denotes the cylinder height associated to $z$.

We now conclude using Proposition~\ref{prop:general bound on pressure lengths}, as in the proof of Theorem~\ref{pressurelength}.
\end{proof}

\section{Distortion}\label{sec:distortion}

The restriction of the pressure metric to the Fuchsian locus is 
a multiple of the Weil Petersson metric on Teichm\"uller space 
\cite{BCLS15} and hence its intrinsic large scale 
geometric properties are well understood. Moreover, 
by~\cite{PS17}, the Fuchsian locus can be characterized as the set of Hitchin representations
whose critical exponent for the symmetric metric as well as for the Hilbert metric 
(and other sufficiently well behaved Finsler metrics) 
assumes a maximum. This intrinsic geometric characterization of the Fuchsian locus
does however not reveal information on its significance for the large scale geometry
of the Hitchin component. 

In fact, the pressure metric for the \emph{Hilbert length}, 
which by definition is induced from the Hilbert metric for convex domains in projective space, 
is degenerate and hence \emph{not} a Finsler metric for the 
Hitchin component. Namely,
the contragredient involution of $\PSL_d(\mathbb{R})$ acts isometrically on the 
character variety equipped with the pressure metric. If $d=2m$ is even, then 
this involution is just conjugation with the standard symplectic form, with fixed 
point set the symplectic group 
${\rm PSp}_{2m}(\mathbb{R})$. It turns out that 
the pressure metric for the Hilbert length is 
degenerate on the normal bundle of the space of representations  with image in 
${\rm PSp}_{2m}(\mathbb{R})$. Note that since 
the involution is an isometry for the pressure metric, the locus of representations into
${\rm PSp}_{2m}(\mathbb{R})$ is totally geodesic.

In the case $d=3$, the fixed point set of the involution equals the 
image ${\rm PSO}(2,1)$ of $\PSL_2(\mathbb{R})$ under the irreducible representation and hence
the Fuchsian locus is totally geodesic for the pressure metric
(see e.g.~\cite{Da19}). 
However, in spite of recent refined information on the 
restriction of the pressure metric to the Fuchsian locus~\cite{LW18},
the following seems to be an open question.

\begin{question}\label{totgeo}
For $d\geq 4$, is the Fuchsian locus totally geodesic for the pressure metric for representations into 
$\PSL_d(\mathbb{R})$?
\end{question}

On purpose, we leave the specification of the length function defining the pressure metric open.

The main goal of this section is to show that 
from a global geometric perspective, 
the Fuchsian locus is distorted for
the pressure distance on the Hitchin component for $n\geq 3$ and genus 
$g\geq 3$, where the pressure distance is taken with respect to the Finsler length 
considered in the previous sections. We believe that similar arguments should lead
to corresponding results for all variants of the pressure metric.

\subsection{Regions of  finite diameter}

Let $S$ be a closed surface of genus $g\geq 3$ and a simple closed curve $\gamma^*\subset S$ that splits 
$S$ into two subsurfaces $S_0,S_1$ of genus $g_0=2,g_1=g-2\geq 1$.
Let $\sigma>0$ and $0<\ell\leq \sigma$.

Let ${\mathcal T}(S_i,\ell)$ (with $i=0,1$) and ${\mathcal T}(S,\ell)$ be the Teichm\"uller spaces  of marked hyperbolic metrics on 
$S_i$ and $S$ such that $\gamma^*$ has length $\ell$.
The restriction map 
$$(r_0,r_1):\mc T(S,\ell)\to \mc T(S_0,\ell)\times\mc T(S_1,\ell)$$
is a fiber bundle with fiber $\R$, on which the twist flow along $\gamma^*$ acts by translation.
Using for instance Fenchel--Nielsen coordinates, one can find a section and have a parametrisation of $\mc T(S,\ell)$ of the form
$$\mc T(S_0,\ell)\times\mc T(S_1,\ell)\times\R\simeq \mc T(S,\ell)\subset \mc T(S).$$
Given a different section 
we will obtain a different parametrisation but with the same image, and the resulting change of coordinate will be of the form
\begin{align*}\mc T(S_0,\ell)\times\mc T(S_1,\ell)\times\R&\rightarrow 
\mc T(S_0,\ell)\times\mc T(S_1,\ell)\times\R\\(X_0,X_1,t)&\mapsto (X_0,X_1,t+b)
\end{align*}
where $b>0$ is a fixed constant.


The following is the main result of this section. 

\begin{theorem}\thlabel{shortcut}
Let $X_0\in {\mathcal T}(S_0,\ell)$.
Consider the subset $r_0^{-1}\{X_0\}\subset\mc T(S)$: the points whose restriction to $S_0$ is isometric to $X_0$.
Equivalently, it is the image of the embedding
$$\{X_0\}\times{\mathcal T}(S_1,\ell)\times\R\hookrightarrow {\mathcal T}(S)\hookrightarrow{\rm Hit}(S).$$
Then for the pressure metric on the Hitchin component, this set has diameter 
 bounded by a number $C$ that only depends on 
 an upper bound $\sigma$ for $\ell$.
\end{theorem}

\begin{question}\label{finitediam}
Is the diameter for the pressure metric of the Fuchsian locus finite?
Is the diameter of the Hitchin component for the pressure metric finite?
\end{question}

The proof of Theorem~\ref{shortcut} has three main ingredients: Theorem~\ref{pressurelength} (upper bounds for pressure lengths of Hitchin grafting paths), Theorem~\ref{pressurelength2} (under an entropy gap assumption,  certain paths $\mc T(S)$ pushed in ${\rm Hit}(S)$ via grafting see their pressure length decrease to zero), and a celebrated 
result of Wolpert that the Weil--Petersson length of a path in $\mc T(S)$ obtained by pinching a simple closed curve 
is bounded above by a constant only depending on the length of the curve.
The idea will be, starting from a well chosen path between two arbitrary points of $\{X_0\}\times{\mathcal T}(S_1,\ell)\times\R$, to deform this path by first pinching curves in $S_0$ (to create an entropy gap) and then grafting along $\gamma^*$ to shrink the pressure length of this path to zero.

We now recall Wolpert's result more precisely.
Given essential disjoint simple closed curves
$\alpha_1,\dots,\alpha_{4}\subset S_0$ and $\beta_1,\dots,\beta_{3g-8}\subset S_1$ that, 
together with $\gamma^*$, decompose $S$ into pairs of pants, the Fenchel--Nielsen coordinates give a diffeomorphism 
from $\mc T(S)$ to $\R^{3g-3}_{>0}\times\R^{3g-3}$.
To a hyperbolic metric is associated the lengths of $\alpha_i,\beta_j,\gamma^*$ and twist parameters along these curves.
Pinching $\alpha_1,\dots,\alpha_4$ by multiplying their length by $\lambda<1$ in Fenchel--Nielsen coordinates (while keeping all other coordinates constant) is a transformation of $\mc T(S)$ that does not depend on the choice of $\beta_1,\dots,\beta_{3g-8}$.

\begin{fact}[\cite{wolpert}]\label{wolpert}
    For any $\sigma'>0$ there is a constant $C>0$ such that the following holds.
    Let $\gamma_1,\dots,\gamma_{4}\subset S_0$ be a multicurve splitting $S_0$ into pairs of pants.
    Then for any $S\in\mc T(S)$ giving length at most $\sigma'$ to 
    each $\gamma_1,\dots,\gamma_4,\gamma^*$, pinching the length of $\gamma_1,\dots,\gamma_4$ to zero, while keeping all other Fenchel--Nielsen coordinates constant, yields a path in $\mc T(S)$ with Weil--Petersson length at most $C$.
\end{fact}

\begin{proof}[Proof of \thref{shortcut}]
Let $X_0\in {\mathcal T}(S_0,\ell)$ be a marked 
hyperbolic metric.
Let $X,Y\in \{X_0\}\times{\mathcal T}(S_1,\ell)\times\R$.
We want to show that $X$ can be connected to $Y$ by a path of uniformly bounded pressure metric length.
This path will be constructed by concatenating five paths of Hitchin representations, illustrated in Figure~\ref{fig-bounded-path}. 

Let $X_1,Y_1\in\mc T(S_1,\ell)$ be the restrictions of $X,Y$ to $S_1$.
A result of Wolpert (Corollary 3.5 of \cite{Wolpert82}) says that the tangent space of $\mc T(S_1,\ell)$ at each point is spanned by the vector fields of twist flows along $6g-16$ well chosen simple closed curves (Wolpert's result is stated for closed surfaces, but it also applies to compact surfaces with boundary).
By a classical argument from differential geometry, one can hence connect $X_1$ and $Y_1$ via a path $(X_1(t))_{0\leq t\leq 1}$ which is a  finite concatenation of twisting paths along these well chosen simple closed curves.

As a consequence,
one can connect $X$ and $Y$ via a path 
$$(X(t))_{0\leq t\leq 1}\subset \{X_0\}\times \mc T(S_1,\ell)\times\R$$
which is a  finite concatenation of twisting paths along simple closed curves of $S_1$ (adding to $(\{X_0\}\times\{X_1(t)\}\times\{0\})_t$ a final twist along $\gamma^*$, if necessary).

By Proposition~\ref{surfacewithboundary}, for every $t$ the entropy of $X_1(t)$ is strictly less than $1$, and it varies continuously with $t$.
By compactness, this entropy is bounded from above by $h<1$ independent of $t$.

Now we want to pinch along curves in $S_0$ to create an entropy gap.
By Theorem~\ref{thm:short-pant-decomposition} there is a pair of pants decomponsition $\alpha_1,\dots,\alpha_4\subset X_0$ with length at most $\max(\ell,4\pi)$.
For any $Z\in \{X_0\}\times{\mathcal T}(S_1,\ell)\times\R$ and $\lambda<1$, denote by 
$p_\lambda(Z)\in \mc T(S_0,\ell)\times{\mathcal T}(S_1,\ell)\times\R$ the metric obtained by pinching $\alpha_1,\dots,\alpha_4$ with factor $\lambda$ (multipliying lengths by $\lambda$ in Fenchel--Nielsen coordinates).
By Fact~\ref{wolpert}, the Weil--Petersson length of $(p_\lambda(Z))_{0<\lambda<1}$ is 
bounded from above by some constant $C_3>0$ that depends on $\sigma$.

By Proposition~\ref{surfacewithboundary}, and since $\alpha_1,\dots,\alpha_4$ split $S_0$ into two pairs of pants one of which is not adjacent to $\gamma^*$, for $\lambda_0$ small enough the entropy of $p_{\lambda_0}(X_0)$ is strictly greater than $h$.
In other words, for each $t$ there is an entropy gap in $p_{\lambda_0}(X(t))$ between the $S_0$, that has entropy greater than $h$, and the $S_1$, whose entropy 
is bounded by $h$.

Fix a grafting parameter $z$ transverse to the twist direction, and for $s\geq0$ and $Z\in \mc T(S_0,\ell)\times{\mathcal T}(S_1,\ell)\times\R$ denote by $g_s(Z)\in {\rm Hit}(S)$ 
the Hitchin grafting representation obtained by grafting $Z$ along $\gamma^*$ with parameter $sz$.
By Theorem~\ref{pressurelength} the pressure length of $(g_s(Z))_{s\geq 0}$ is bounded 
above by a constant $C_4>0$ that only depends on $\sigma$ and $z$.

Notice that $(p_{\lambda_0}(X(t)))_{0\leq t\leq 1}$ is, like $(X(t))_{0\leq t\leq 1}$, a concatenation of paths obtained by twisting along closed curves in $S_1$.
Hence we can apply Theorem~\ref{pressurelength2}, which says that the pressure length of $(g_s\circ p_{\lambda_0}(X(t)))_{0\leq t\leq 1}$ goes to zero as $s$ goes to infinity.
Let $s_0$ be such that this length is less than $1$.

Then we consider the concatenation of five paths where we first pinch $(p_\lambda(X(0)))_{1\geq \lambda\geq \lambda_0}$, then graft $(g_s\circ p_{\lambda_0}(X(0)))_{0\leq s\leq s_0}$, then let $t$ vary $(g_{s_0}\circ p_{\lambda_0}(X(t)))_{0\leq t\leq 1}$, then ungraft $(g_s\circ p_{\lambda_0}(X(1)))_{s_0\geq s\geq 0}$, and finally we unpinch $(p_\lambda(X(1)))_{\lambda_0\leq \lambda\leq 1}$.
This connects $X(0)$ to $X(1)$ in ${\rm Hit}(\Sigma)$ and has pressure length at most $2C+2C'+1$, which only depends on $\sigma$.

\end{proof}

\begin{figure}[ht]
    \begin{center}
        \begin{picture}(145,70)(0,0)
        \put(0,0){\includegraphics[width=145mm]{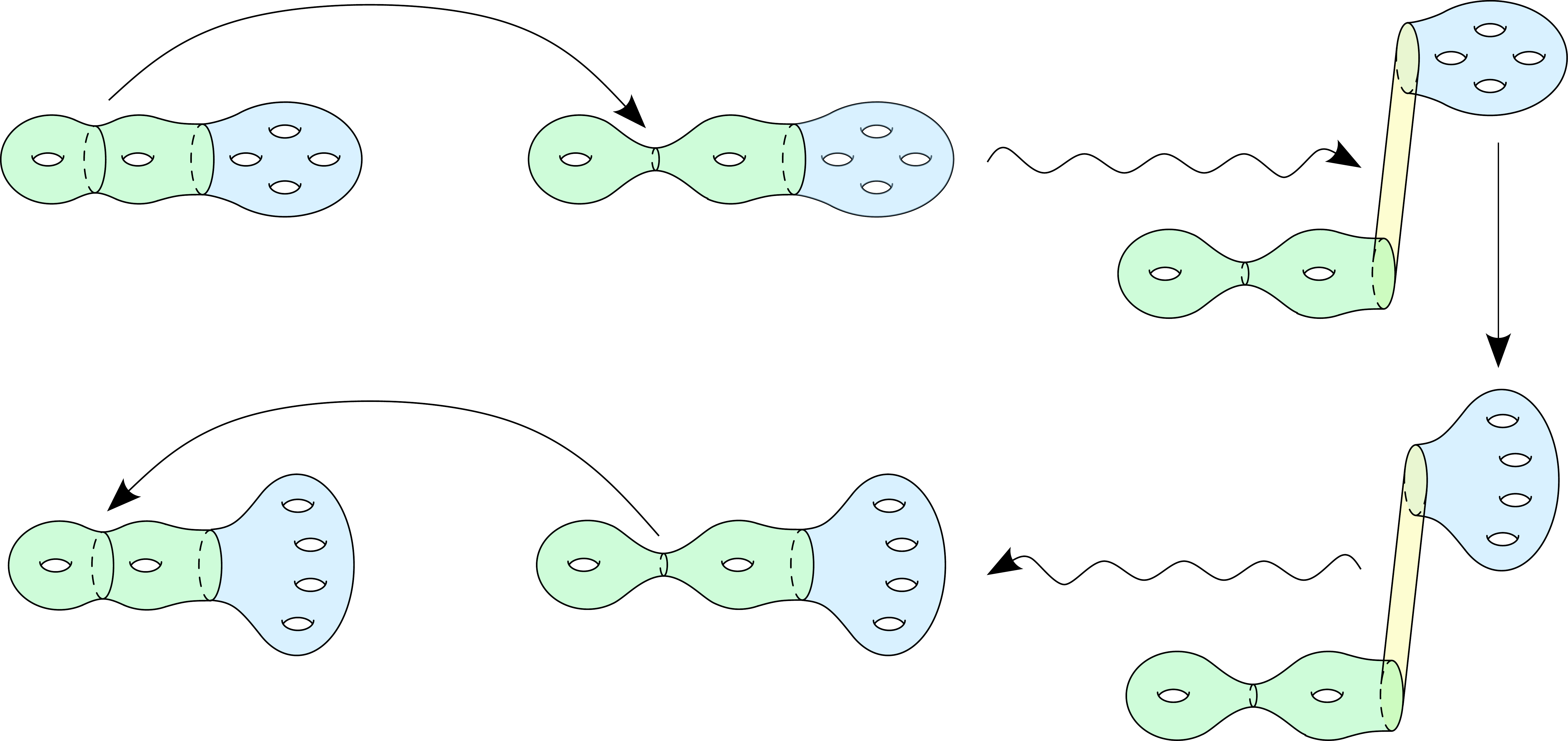}}
        \put(8, 47){$e$}
        \put(17, 46.8){$\gamma^*$}
        \put(2, 59.5){$\Sigma_0$}
        \put(30, 59.5){$\Sigma_1$}
        \put(94, 57){Hitchin grafting}
        \end{picture}
    \end{center}
    \caption{Bounded path of Hitchin representations for the pressure metric. Each path is bounded by a constant that depends only on the length of $\gamma^*$ and on the systole of $\Sigma_0$.}
    \label{fig-bounded-path}
\end{figure}

\subsection{Length comparison with a separating curve graph when $g\geq 5$}

\emph{In this section we assume $g\geq 5$.} 
Let ${\mathcal  S}{\mathcal C}{\mathcal G}(S)$ be 
the graph whose vertices are 
separating simple closed curves which decompose $S$ into a surface of genus $2$ and 
a surface of genus $g-2$ and where two such curves are connected by an edge if they can 
be realized disjointly. We have

\begin{lemma}\label{connected}
The graph ${\mathcal S}{\mathcal C}{\mathcal G}(S)$ is connected.
\end{lemma}

\begin{proof}
The mapping class group ${\rm Mod}(S)$ of $S$ clearly acts transitively on the vertices
of ${\mathcal S}{\mathcal C}{\mathcal G}(S)$. Thus to check connectedness, 
we can apply a trick due to Putman~\cite{Put08}: Choose a vertex $c $ of 
${\mathcal S}{\mathcal C}{\mathcal G}(S)$ and a generating set 
$\psi_1,\dots,\psi_k$ of ${\rm Mod}(S)$. If for each $j$ the vertex $c$ can be
connected to $\psi_j(c)$ by an edge path in 
${\mathcal S}{\mathcal C}{\mathcal G}(S)$, then the graph is connected.

To see that this condition is satisfied we choose the Humphries generating set 
$\psi_1,\dots,\psi_{2g+1}$ of ${\rm Mod}(S)$ 
consisting of Dehn twists about the non-separating
simple closed curves $a_1,\dots,a_g,c_1,\dots,c_{g-1},m_1,m_2$ 
in $S$ as shown in Figure 4.5 of 
\cite{FM12}. That these elements generate ${\rm Mod}(S)$ is explained in Theorem 4.14 
of~\cite{FM12}. 
Let furthermore~$c$ be the separating simple closed curve 
which intersects the simple closed curve $c_2$ in precisely two points and 
is disjoint from any of the curves $a_i,c_j,m_u$ for $j\not=2$.
Then $\psi_s(c)=c$ for $s\not=g+2$, moreover both $c,\psi_{g+2}(c)$ are disjoint 
from the vertex $b$ of ${\mathcal S}{\mathcal C}{\mathcal G}(S)$ which 
intersects $c_{g-2}$ in precisely two points (this is where we need $g\geq 5$) and is disjoint from the remaining curves. 
Thus $c,b,\psi_{g+2}(c)$ is an edge path connecting $c$ to $\psi_{g+2}(c)$, which suffices
for the proof of the lemma.
\end{proof}

Let $\Upsilon:{\mathcal T}(S)\to {\mathcal S}{\mathcal C}{\mathcal G}(S)$ 
be a map which associates to $X\in {\mathcal T}(S)$ 
a point 
in ${\mathcal S}{\mathcal C}{\mathcal G}(S)$ whose length is minimal among
the lengths of all separating geodesics which cut $S$ into a surface of genus $2$ and a surface
of genus $S_2$. 
The length of $\Upsilon(X)$ in $X$ is bounded above by a constant $\sigma$ that only depends on $g$ by Theorem~\ref{thm:short-pant-decomposition}.
We use Lemma~\ref{connected} to show

\begin{theorem}\label{lengthcontrol}
For any $d\geq 3$ there exists a number $C(d,g)>0$ with the following property. 
Let $X,Y\in {\mathcal T}(S)$ be any two points in the Fuchsian locus of the Hitchin 
component of representations $\pi_1(S)\to \PSL_d(\mathbb{R})$.
Then the pressure metric distance
between $X,Y$ is at most 
$C(d,g)d(\Upsilon(X),\Upsilon(Y))+C(d,g).$
\end{theorem}
\begin{proof} In this proof, distances between Hitchin representations are always
taken with respect to the path metric defined by the pressure metric.

Let $X,Y\in {\mathcal T}(S)$. 
Suppose first that $d(\Upsilon(X),\Upsilon(Y))=1$ (the case $d(\Upsilon(X),\Upsilon(Y))=0$ is similar)
The curves $\Upsilon(X)$ and $\Upsilon(Y)$ split $S$ into three subsurfaces $S_1,S_2,S_3$ such that $S_1$ have genus $2$ and one boundary component $\partial S_1=\Upsilon(X)$, $S_2$ has genus $g-4$ and two boundary components, and $S_3$ have genus $2$ and one boundary component $\partial S_3=\Upsilon(Y)$.

Let $Z\in \mc T(S)$ that coincides with $X$ on $S_1$ and coincides with $Y$ on $S_3$.
By Theorem~\ref{shortcut} the distance from $X$ to $Z$ is bounded by some constant that depends on $\sigma$ (upper bound on the length $X$ and $Z$ give to $\Upsilon(X)=\partial S_1$).
Similarly the distance from $Z$ to $Y$ is bounded by some constant that depends on $\sigma$.

Now if $m=d(\Upsilon(X),\Upsilon(Y))\geq 2$ then let $\alpha_0=\Upsilon(X),\alpha_1,\dots,\alpha_m=\Upsilon(Y)$ be a minimizing path in $\mc S\mc C\mc G(S)$.
For each $1\leq i\leq m-1$ let $X_i\in\mc T(S)$ such that $\Upsilon(X_i)=\alpha_i$, and let $X_0=X$ and $X_m=Y$.
For every $i$ we have $d(\Upsilon(X_i),\Upsilon(X_{i+1}))=1$ so we can apply the above: the distance from $X_i$ to $X_{i+1}$ is bounded above by some constant $C$ that only depends on $\sigma$, and hence on $g$.
Thus the distance from $X$ to $Y$ is at most $mC$.
\end{proof}

\subsection{Fixed point for a subgroup of the mapping class group}

The mapping class group ${\rm Mod}(S)$ of $S$ acts by precomposition of markings
on the Hitchin component ${\rm Hit}(S)$ preserving the Fuchsian locus $\mc T(S)$ and the pressure metric (whose restriction to $\mc T(S)$ is the Weil--Petersson metric).
Thus ${\rm Mod}(S)$ also acts on the Weil--Petersson metric completion $\overline{\mc T(S)}$ of $\mc T(S)$ and on the pressure metric completion $\overline{\mr{Hit}(S)}$ of ${\rm Hit}(S)$.

Note that the embedding of $\mc T(S)\hookrightarrow{\rm Hit}(S)$ extends to a continuous but noninjective map $\overline{\mc T(S)}\rightarrow\overline{{\rm Hit}(S)}$ which is equivariant under the actions of the mapping class group.

Recall that $\overline{\mc T(S)}$ is stratified.
A stratum is defined by a simple geodesic multicurve
$c\subset S$, and it consists of the Teichm\"uller space of 
all marked complete finite volume hyperbolic metrics on 
$S-c$. By this we mean that each component of $S-c$ is an essential subsurface of 
$S$ of negative Euler characteristic, and hence it determines a Teichm\"uller space
of marked complete finite volume hyperbolic metrics on the component. 
The stratum of $S-c$ is then the product of these Teichm\"uller spaces.

The action of the mapping class group ${\rm Mod}(S)$ of $S$ 
on boundary points for the metric completion of ${\mathcal T}(S)$ 
projects to the action of the
mapping class group on the curve complex, thought of as remembering the 
nodes (or cusps) of the completion points. Dehn multitwists have global 
fixed points acting on this boundary: if $T_c$ 
is a Dehn twist about~$c$, then any surface with node at $c$ is 
fixed by $T_c$. However, there is no subgroup of the 
mapping class group containing a free group with two generators which 
acts with a global fixed point.

In contrast, the action of the outer automorphism group of the free group 
$F_k$ with $k\geq 3$ generators on the metric completion of Outer space 
of marked graphs with fundamental group $F_k$, equipped with an analog of 
the pressure metric, has a global fixed point (see~\cite{ACR23}). 

Our final result shows that a weaker but related
statement holds true for the action of the mapping class
group ${\rm Mod}(S)$ 
on the metric completion of the Hitchin component, equipped with the 
pressure metric, provided that the genus of $S$ is at least $3$. 
For the formulation of our result, 
recall that for every essential subsurface $S_1$ of the surface $S$
with connected boundary, the mapping 
class group ${\rm Mod}(S_1)$ of $S_1$ embeds into the mapping class
group ${\rm Mod}(S)$ of $S$ as a group of isotopy classes of homeomorphisms of 
$S$ which fix~$S-S_1$ pointwise.

We will prove that if $S_1$ has genus $g-2$ and one boundary component, then $\mr{Mod}(S_1)$ fixes a point of the metric completion of $\Hit(S)$ for the Pressure metric, and this point is explicit: let us describe it now.

Let $\phi_1,\dots,\phi_k$ be a generating set of $\mr{Mod}(S_1)$ consisting of Dehn twists.
It suffices to find a point fixed by these generators.
Denote by $\phi_i^t:\mc T(S_1)\to\mc T(S_1)$ the twist flow whose time $1$ map is $\phi_i$.
Fix a hyperbolic metric $X_1$ on $S_1$.
Let $a<1$ be the maximum of the entropies of all points $\phi_i^t(X_1)\in\mc T(S_1)$, for $1\leq i\leq k$ and $0\leq t\leq 1$.
Let $X_0$ be a hyperbolic metric on $S-S_1$ with entropy greater than $a$ and with same boundary length as $X_1$ (using that $S-S_1$ has genus $2$).
Let $X\in\mc T(S)$ be a gluing of $X_0$ and $X_1$, and for $L>0$ let $X(L)\in\Hit(S)$ be a grafting of $X$ along $\partial S_1$ with cylinder height $L$ (and fixed grafting direction).
By Theorem~\ref{pressurelength}, the path $(X(L))_{L\geq0}$ has finite length and hence converges to a point of the completion $X(\infty)\in\overline{\Hit(S)}$

\begin{theorem}\label{metriccompletion}
The subgroup ${\rm Mod}(S_1)\subset {\rm Mod}(S)$ 
fixes $X(\infty)$.
\end{theorem}
\begin{proof}
As mentioned it suffices to fix $1\leq i\leq k$ and prove $\phi_i(X(\infty))=X(\infty)$.

The point $\phi_i(X(\infty))$ is the limit as $L\to\infty$ of $\phi_i(X(L))$, which is obtained by gluing $X_0$ to $\phi_i(X_1)$ and grafting along $\partial S_1$ with cylinder height $L$.
 
By the discussion in the proof of Theorem~\ref{shortcut},
as $L\to \infty$, the length for the pressure metric of the path 
$t\to \phi_i^t(X(L))$ tends to zero. 
Hence the Pressure distance between $X(L)$ and $\phi_i(X(L))$ tends to zero as $L\to\infty$, which proves that $\phi_i(X(\infty))=X(\infty)$.
\end{proof}

\subsection{Proof of Theorem~\ref{loftin}}

As mentioned in the introduction, in \cite{Loftin04,LoftinNeck}, Loftin constructed
a natural bordification of the space of Hitchin representations $\Hit_3(S)$ of a
closed surface $S$, called the augmented Hitchin space $\Hit_3^{\mr{aug}}(S)$, which
extends the augmented Teichm\"uller space.
This construction applies more generally to noncompact finite type surfaces and
their moduli spaces of convex projective structures.
Our goal in this section is to relate Loftin's bordification with our grafting
procedure.
More precisely we want to show that, starting with a Fuchsian representation and
grafting it with grafting parameter going to infinity in a specific direction, the
resulting family of Hitchin representations will converge to a point in Loftin's
bordification.

Let $S$ be a connected surface of finite type, seen as a closed surface with punctures.
Recall that a projective structure is an atlas of charts on $S$ into the projective
plane such that the change of charts are projective transformations.
To such a structure can be associated a holonomy representation of the fundamental
group into the group of projective transformations $\PSL_3(\R)$, and a
holonomy-equivariant developing map from the universal cover $\tilde S$ into the
projective plane.
A projective structure is called convex if the developing map is injective and its
image is properly convex (convex and bounded in some affine chart), which implies
the holonomy representation is faithful with discrete image.
In this case, the projective structure is completely determined by the data of the
holonomy representation and the image of the developing map by Proposition 2.5 of
\cite{LZ21}.

The moduli space of convex projective structures $\mc C(S)$ can be described as the
quotient under the action of $\PSL_3(\R)$ of the set of pairs $(\Omega,\rho)$, where
$\Omega\subset\R\PP^2$ is open and properly convex and $\rho$ is a discrete and
faithful representation of $\pi_1(S)$ into $\PSL_3(\R)$ that preserves $\Omega$.
It is topologized so that $(\Omega_n,\rho_n)\to(\Omega,\rho)$ if $\Omega_n\to\Omega$
for the Hausdorff topology and $\rho_n\to\rho$ on a set of generators (up to the
action of $\PSL_3(\R)$).

The projective structures around punctures can be classified, and in particular the
conjugacy class of the holonomy of a curve enclosing a puncture can be of three
types: parabolic
$\left(\begin{smallmatrix}1&1&0\\0&1&1\\0&0&1\end{smallmatrix}\right)$,
quasi-hyperbolic
$\left(\begin{smallmatrix}\lambda&0&0\\0&\mu&1\\0&0&\mu\end{smallmatrix}\right)$ or
hyperbolic
$\left(\begin{smallmatrix}\lambda&0&0\\0&\mu&0\\0&0&\nu\end{smallmatrix}\right)$
(where $\lambda,\mu,\nu$ are distinct).
As explained in the Appendix A of \cite{LZ21}, the projective structure around the
puncture is determined by this holonomy in the parabolic and quasi-hyperbolic cases.
However in the hyperbolic case there are many structures with the same holonomy.
In particular any such structure can be deformed locally with out changing the
holonomy by a bulging procedure (one can ``inflate'' or ``deflate'' the structure
near the puncture).
The two special degenerate structures obtained by inflating or deflating to infinity
any other structure are called respectively bulge $+\infty$ and bulge $-\infty$.
See e.g.\ Figure 4 of \cite{LZ21}.
To conclude, for any pair $(\Omega,\rho)$, the convex set $\Omega$ is determined by
$\rho$ and the projective structure around punctures of hyperbolic type.

In particular, if $S$ is closed then every point of $\mc C(S)$ is determined by the
holonomy representation.
By work of Choi and Goldman \cite{Goldman90,CG93}, $\mc C(S)$ is connected, open and
closed as a subset of the set of representations of $\pi_1(S)$, and it contains the
representations coming from hyperbolic structures, so $\mc C(S)=\Hit_3(S)$.

To define the augmented Hitchin space,
Loftin first defines admissible convex projective structures by allowing only bulge
$\pm \infty$ structures near the punctures of hyperbolic type.
Then $\Hit_3^{\mr{aug}}(S)$ is defined as the set, over all multicurves $\mc
D\subset S$, of admissible convex projective structures
$(\Omega_1,\rho_1),\dots,(\Omega_k,\rho_k)$ on the connected components
$S_1,\dots,S_k$ of $S-\mc D$ that satisfy some compatibility conditions between the
pairs of ends corresponding to the same curve $\gamma\subset\mc D$: they have the
same holonomy and a bulge $+\infty$ end must face a bulge $-\infty$ end.
It is further topologised so that $(\Omega^{(n)},\rho^{(n)})\in \Hit(S)$ converge to
$((\Omega_1,\rho_1),\dots,(\Omega_k,\rho_k))$ in the boundary if
$(\Omega^{(n)},\rho^{(n)}_{|\pi_1S_i})\to (\Omega_i,\rho_i)$ for every $i$ (up to
the action of $\PSL_3(\R)$).

Let us now relate the above construction with the algebraic bending deformation of a
Fuchsian representation $\rho$ along a multicurve $\mc D\subset S$, as recalled in
Section~\ref{sec:Abstract grafting}: it was defined by partially conjugating the image
by $\rho$ of the fundamental groups of the connected components $S_{1},\dots,S_{k}$
of $S-\mc D$.
We gave in \cite{BHM25} and \ref{sec-HG} a
geometric interpretation of this deformation, inside the symmetric space of
$\PSL_3(\R)$, in terms of grafting a flat cylinder along the multicurve $\mc D$.
Suppose now that all the grafting parameters (which are vectors of the Cartan
subspace) are parallel to the special direction
$\left(\begin{smallmatrix}1&0&0\\0&-2&0\\0&0&1\end{smallmatrix}\right)$.
Then there is another geometric interpretation of bending due to Goldman
\cite[\S5.5]{Goldman90} using convex projective geometry: bending induces a
deformation of the underlying convex projective structure called \emph{bulging},
which is the same procedure as the local surgery around punctures mentioned
previously.
The idea is the same as before (when
$k=2$ and $\mc D$ has only one curve): suppose
$\rho_{z}(\pi_{1}(S_{1}))=\rho(\pi_{1}(S_{1}))$ is unchanged and
$\rho_{z}(\pi_{1}(S_{2}))=e^{z}\rho(\pi_{1}(S_{2}))e^{-z}$.
The $\rho$-invariant convex domain $\Omega\subset\R\PP^{2}$ is made of a tree of
infinitely many copies of universal covers of $S_{1}$ and $S_{2}$, each copy being
invariant under a conjugate of $\rho(\pi_{1}(S_{1}))$ or $\rho(\pi_{1}(S_{2}))$.
The $\rho_{z}$-invariant convex domain $\Omega_{z}$ is then produced by deforming
each of these copies using $e^{z}$ and $e^{-z}$ and conjugates of them: e.g.\ if
$\Omega_{2}$ is a $\rho(\pi_{1}(S_{2}))$-invariant copy of $\wt S_{2}$ then
$e^{z}\Omega_{2}$ is  $\rho_{z}(\pi_{1}(S_{2}))$-invariant.
One can see that $e^{z}$ acts by inflating $\Omega_{2}$, without disconnecting it
from the adjacent copies of $\wt S_{1}$ (so there is no need to graft a flat
cylinder as in the symmetric space).
The following fact is an immediate consequence of Goldman's work and the above
definition of Loftin's bordification.
Fix a grafting parameter $z$ parallel to
$\left(\begin{smallmatrix}1&0&0\\0&-2&0\\0&0&1\end{smallmatrix}\right)$.

\begin{fact}
For any $t>0$ let $[\rho_{t}]\in\Hit_{3}(S)$ be obtained by grafting $\rho$ along
$\mc D$ with parameter $tz$.
Then as $t$ goes to infinity, $[\rho_{t}]$ converges to
$[(\Omega_1,\eta_1),\dots,(\Omega_k,\eta_k)]\in\Hit_{3}^{\mr{aug}}(S)$ (projective
structures on $S_{1},\dots,S_{k}$) such that the projective structures near the two
ends associated to a $\gamma\subset\mc D$ are of hyperbolic type with bulge
$+\infty$ and $-\infty$ respectively, and the holonomies $\eta_{i}$ are the
restrictions of $\rho$ to $\pi_{1}(S_{i})$.
\end{fact}

To prove Theorem~\ref{loftin}, we consider the case where $S$ is cut into two
subsurfaces $S_{1},S_{2}$ such that the entropy of $\eta_{1}$ is strictly greater
than that of $\eta_{2}$.
We slightly perturb $\rho$ into $(\rho^{s})_{-\epsilon\leq s\leq \epsilon}$ so that
$\eta_{1}^{s}=\eta_{1}$ for any $s$ with entropy still greater than that of
$\eta^{s}_{2}$, and $\eta^{s}_{2}$ and $\eta_{2}^{\sigma}$ are not conjugate for
$s\neq \sigma$.
Now we graft, and by Theorem~\ref{pressurelength2} the pressure length of
$(\rho_{t}^{s})_{-\epsilon\leq s\leq \epsilon}$ goes to zero as $t$ diverges, which
implies all $(\rho_{t}^{s})_{t\to\infty}$ converge to the same point of the pressure
metric completion of $\Hit_{3}(S)$, independent of $s$.
However by the above fact they converge to different points of Loftin's augmented
Hitchin space.
Heuristically, the pressure metric is not fine enough to distinguish points in
$\Hit^{\mr{aug}}_{3}(S)$, because it focuses too much on the component with bigger
entropy and can only see changes there.

Another interesting remark can be made about another description of the augmented
Hitchin space (which is in fact Loftin's original definition), in terms of cubic
differentials.
Recall that by independent work of Labourie \cite{Labouriecubic} and Loftin
\cite{Loftincubic}, there is a vector bundle structure $\pi:\Hit_{3}(S)\to\mc T(S)$
such that the fiber above a point of $\mc T(S)$, seen as a (marked) complex
structure on $S$, is the vector space of holomorphic cubic differentials on $S$.
It turns out this vector bundle structure extends to
$\pi:\Hit_{3}^{\mr{aug}}(S)\to\mc T^{\mr{aug}}(S)$.
Moreover, using the notations from the above fact and denoting the limit of
$[\rho_{t}]$ as $t\to\infty$ by
$[\rho_{\infty}]=[(\Omega_1,\eta_1),\dots,(\Omega_k,\eta_k)]$, it follows from
Theorem~12 of \cite{LoftinNeck} that the projection $\pi[\rho_{\infty}]\in\mc
T^{\mr{aug}}(S)$ is the noded hyperbolic surface obtained by pinching to zero the
multicurve $\mc D\subset S$.

Hence for $t$ large the Hitchin grafting representation $\rho_{t}$, which we think
of in this paper as the hyperbolic structure $\rho$ where we grafted long flat
cylinder along $\mc D$, naturally stands above another hyperbolic structure
$\pi(\rho_{t})$ on $S$ with long and narrow hyperbolic collars around $\mc D$.
Since pinching a curve in $\mc T(S)$ is a finite length surgery for the
Weil--Petersson metric, it seems likely that $(\pi[\rho_{t}])_{t>0}$ has finite
length.
As $([\rho_{t}])_{t>0}$ also has finite length, for any $t$ the pressure distance
from $\pi[\rho_{t}]$ to $[\rho_{t}]$ is bounded independently of $t$.
Moreover, there is a natural straight-line path between these two points, since
$[\rho_{t}]$ lies in the fiber above $\pi[\rho_{t}]$, which is a vector space.
A natural question is then: is the pressure length of this path bounded above
independently of $t$?

\appendix

	\section{Entropy of hyperbolic surfaces with boundary}
	
The goal of this appendix is to collect some basic
results on the entropy of hyperbolic surfaces with boundary. We give proofs for the ones 
we did not find in the literature, although they should be well known by the experts. Some of the following 
statements are consequences of more general theorems. 

    Consider a compact surface $\Sigma$, of genus $g$, with at least one boundary component. Let~$S$ be a hyperbolic surface obtained by equipping $\Sigma$ with a hyperbolic metric, so that its boundary is geodesic, that is, $S$ belongs to 
    the Teichm\"uller space ${\mathcal T}(\Sigma)$ for $\Sigma$. Denote by~$h(S)$ the topological entropy of the geodesic flow on $T^1S$. 
    We also denote by $\delta(S)$ the critical exponent of any representation $\pi_1(\Sigma)\to\PSL_2(\R)$ representing the metric $S$.

	\begin{proposition}\thlabel{surfacewithboundary} 
        The following holds true:
		\begin{enumerate}
            \item $h(S)=\delta(S)$ (see {\rm~\cite{sullivan79}}).
            \item The function $\delta(S)$ is real analytic in $S$ and invariant under the action of ${\rm Mod}(\Sigma)$ 
            (see {\rm~\cite{Ruelle}}).
			\item $h(S)<1$. \label{surfacewithboundary-less-than-one}
			\item 
            Take a pants decomposition of $\Sigma$. When sending to zero the lengths of all boundary curves of a fixed pair of pants,
            the entropy goes to one. \label{surfacewithboundary-to-one}
		\end{enumerate}
	\end{proposition}
	
	\begin{proof}

        \emph{Statement 3.}
		It follows from Proposition 5 of~\cite{PollicottSharp} that the Poincaré series $P(\delta(S))$ is diverging at the critical exponent
  $\delta(S)$.
		Consider a closed hyperbolic surface $\Sigma_d$ obtained by doubling $\Sigma$ along its boundary, 
  equipped with the double $S_d$ of the given hyperbolic metric $S$. It follows from Proposition 2 of~\cite{DOP} that we have $\delta(S)<\delta(S_d)$ (it uses as hypothesis that $P(\delta(S))$ is diverging). The latter is known to be equal to one. 
        Namely, for a hyperbolic metric $S_d$ without boundary and with finite volume, the limit set of $\pi_1(\Sigma)$ in $\partial_\infty\H^2$, that is, the accumulation points of the orbit $\pi_1(\Sigma)\cdot x$ for any~$x\in\H^2$, is equal to all of $\partial_\infty\H^2$. 
        It follows from Theorem 1.1 of~\cite{BishopJones} that the critical exponent of $S_d$ is one.
        
        \emph{Statement 4.}
		Take a compact hyperbolic surface with boundary 
  and pinch all boundary components. The critical exponent of Kleinian groups is lower semi-continuous for the so called algebraic convergence, see Theorem 2.4 of~\cite{BishopJones}. 
  It implies that when decreasing the lengths of the boundary curves to zero, the limit inferior of the critical exponents is at least the critical exponent of the surface obtained by pinching the boundary curves. That is one according to the proof of statement (3). 
	\end{proof}

    We mention a result of Hugo Parlier, which is a neat improvement of results already known previously.

    \begin{theorem}[\cite{parlier2023shorter}]\thlabel{thm:short-pant-decomposition}
        Let $S$ be a hyperbolic surface, possibly with boundary, and with finite volume. Then $S$ admits a pant decomposition for which the length of each curve is at most $\max(length(\de S), area(S))$.
    \end{theorem}

    \begin{proposition}\label{prop:entropy of hypsurf}
    There exists a function $f_1$ depending on $\Sigma$ ($\de\Sigma\neq\emptyset$) such that the following holds.
               If every boundary component has length at most $\sigma$ and at least one of them has length at most $\epsilon\leq \sigma$ then $\delta(S)\geq f_1(\sigma,\epsilon)>0$ with $\liminf_{\epsilon\to 0} f_1(\sigma,\epsilon)>\tfrac12$ for fixed $\sigma$.
    \end{proposition}

 \begin{proof}
     Denote by $S_n$ a sequence of metrics as in the proposition, so that all boundary components 
     of $\Sigma$ have length at most $\sigma$. 
     Using \thref{thm:short-pant-decomposition}, $S_n$ admits a decomposition into hyperbolic pairs of pants 
     $P_1^{(n)},\cdots, P_r^{(n)}$ so that the decomposing curves have a length bounded from above by some constant $C(\Sigma)$, and so that the shortest boundary component of $S_n$ is in~$P_1^{(n)}$. Our goal is to show that $\delta(S_n)$ is bounded from
     below by a universal constant.

     Suppose by contradiction that $\delta(S_n)\to 0$.
     Then $\delta(P_1^{(n)})\to 0$ since it is bounded from above by $\delta(S_n)$.
     Up to extraction we may assume that the boundary lengths of $P_1^{(n)}$ converge, 
     which implies that $P_1^{(n)}$ converge to some hyperbolic pair of pants $P$, possibly with cusps.
     By lower semicontinuity of $\delta$ (see Theorem 2.4 of~\cite{BishopJones}) 
     we get that $0=\lim_n\delta(P_1^{(n)})$, which is absurd.
     Thus the critical exponents are bounded away from zero.

     Let us now prove the second part of the statement.
     Suppose by contradiction that the shortest boundary curve of $S_n$ has length tending to zero, but $\liminf_n\delta(S_n)\leq 1/2$.
     Then  $\liminf_n\delta(P_1^{(n)})\leq 1/2$.
     Once again, up to extracting we may assume $P_1^{(n)}\to P$, with $P$ having a cusp (since the shortest boundary of $P_1^{(n)}$, which is that of $S_n$, is pinched to zero).
     By lower semicontinuity of $\delta$ (see Theorem 2.4 of~\cite{BishopJones}) we get that $\liminf_n\delta(P_1^{(n)})\geq \delta(P)$.
     This is absurd  as $\delta(P)>1/2$ by Proposition 2 of~\cite{DOP}, since the critical exponent of a neighbourhood of a cusp is $1/2$, with a diverging Poincaré series at the critical exponent.
 \end{proof}

    \paragraph{Hyperbolic pairs of pants.}
    Here suppose that $\Sigma$ is a sphere with three boundary components, and $S_{a,b,c}$ is the metric of a hyperbolic pair of pants with boundary length $a$, $b$ and~$c$. 

    \begin{proposition}\thlabel{lem:entropy-pants-convergence}
    There exists a function $f_2$ depending on $\Sigma$ ($\partial \Sigma \not=\emptyset)$ with the following property. 
    If $\Sigma$ is a pair of pants, two boundary components of $S$ have length at least $\sigma>0$ and the third at least $\ell\geq \sigma$, then $\delta(S)\leq f_2(\sigma,\ell)$ with $f_2(\sigma,\ell)\to 0$ 
            for fixed $\sigma >0$ as~$\ell\to \infty$.
    \end{proposition}

    We use the notations from~\cite{Martone2019}, where the authors give some control on the entropy of a hyperbolic surface 
    using the lengths of the small curves on the surface. Denote by $L(S)$ the systole of $S$, 
    that is, the length of the shortest closed geodesic in $S$. Denote by $K(S)$ the length of the shortest 
    closed geodesic in $S\setminus\partial S$ ($K(S)$ is more complicated to define when $S$ is not a pair of pants). 
    Also let $\delta(S)$ be the critical exponent of $S$.

    \begin{theorem}[Particular case of {Theorem 1.4 of~\cite{Martone2019}}]
        There exists a constant $C>0$ for which we have
        $$\frac{1}{4}\log(2)\leq\delta(S)K(S)\leq C\left(\log(4)+1+\log\left(1+\frac{1}{x_0}\right)\right)$$
        where $x_0$ is the unique positive solution of the equation $(1+x)^{\left\lceil \frac{K(S)}{L(S)}-1\right\rceil}x=1$.
    \end{theorem}

    \begin{lemma}\label{lem:lower bound K(S)}
        Let $S$ be a pair of pants with boundary lengths $a,b,c$.
        Then $K(S)\geq \max(a,b,c)$.
    \end{lemma}
    \begin{proof}
        Up to reordering we may assume $\max(a,b,c)=c$.
        The surface $S$ is obtained by gluing two isometric right-angled 
        hyperbolic hexagons $H_1=H,H_2$ along three nonadjacent sides, such that the three other sides have lengths $\tfrac a2,\tfrac b2,\tfrac c2$.
        In particular, there is a natural projection $\pi:S\to H$.
        Let $\bar A,\bar B,\bar C$ be the sides of $H$ which are glued, so that the hyperbolic distance from $\bar B$ to $\bar C$ is $a/2$, the distance from $\bar C$ to $\bar A$ is $b/2$, and the distance from $\bar A$ to $\bar B$ is $c/2$.
        
        Let $\gamma$ be a closed geodesic in $S\setminus \partial S$, and let us check it has length at least $c$.
        Note that $\pi(\gamma)\subset H$ is a concatenation of geodesics between the sides $\bar A,\bar B,\bar C$.
        This path has to intersect all three sides, for if it was alternating between only two sides,
        then $\gamma$ is freely homotopic to a multiple of the boundary curve 
        of $S$ between these two sides.

        Say $\gamma$ starts on the side $\bar A$ at some point $x$, then travels until it hits $\bar B$ at some point~$y$ (maybe bouncing off $\bar C$ and $\bar A$ in between), and then comes back to $x$.
        The first part of the path from $x$ to $y$ must have length at least the distance from $\bar A$ to $\bar B$, which is~$c/2$, and similarly the second part has length at least $c/2$ too, so in total $\gamma $ has length at least~$c$.
    \end{proof}

    \begin{proof}[Proof of \thref{lem:entropy-pants-convergence}]
        Let $(a_n)_n,(b_n)_n,(c_n)_n$ be three sequences in $\R^+$ so that $a_n$ and~$b_n$ are bounded away from zero, and $c_n$ tends to to infinity with $n$. Let 
        $S_n=S_{a_n,b_n,c_n}$ be the pair of pants with boundary lengths $a_n,b_n,c_n$. 
        By Lemma~\ref{lem:lower bound K(S)}, $K(S_n)$ tends to infinity with $n$. 
        
        By assumption, $L(S_n)$ is bounded away from zero. So up to passing to a subsequence, we can assume 
        that $\frac{K(S_n)}{L(S_n)}$ converges to $y\in(0,+\infty]$.
        If $y<+\infty$, then the solutions $x_n$ of $(1+x)^{\left\lceil \frac{K(S_n)}{L(S_n)}-1\right\rceil}x=1$ remain bounded away from zero. So $C\left(\log(4)+1+\log\left(1+\frac{1}{x_n}\right)\right)$ is bounded, and $\delta(S_n)\leq \frac{cste}{K(S_n)}$ goes to zero.

        If $y=+\infty$, then $x_n$ goes to zero, and a simple analysis yields that $-\frac{\log(x_n)}{x_n}$ is equivalent to $\frac{K(S_n)}{L(S_n)}$. 
        It follows that
        \begin{align}
            \delta(S_n)K(S_n)
                &\leq C\left(\log(4)+1+\log\left(1+\frac{1}{x_n}\right)\right) \\
                &\leq \mathrm{Cst}\cdot x_n\frac{K(S_n)}{L(S_n)} \\
        \text{ and hence } \quad    \delta(S_n) &\leq \mathrm{Cst}\cdot \frac{x_n}{L(S_n)}\xrightarrow[n\to 0]{}0
        \end{align}
    \end{proof}

\paragraph{Surfaces with one boundary component.}

Assume now that the surface $\Sigma$ is of genus genus $g=g(\Sigma)$, with exactly one boundary component. Let ${\mathcal T}(\Sigma,\ell)$ be the Teichm\"uller space  of marked hyperbolic structures on $\Sigma$ with geodesic connected boundary of length~$\ell$. 
Denote also by ${\mc T}_\epsilon(\Sigma,\ell)\subset {\mc T}(\Sigma,\ell)$ the subset of structures whose systole is at least 
$\epsilon$. 

\begin{lemma}\label{entropyfunction}
The following holds true:
    \begin{enumerate}
        \item If $g=1$ and $S\in\mathcal T(\Sigma,\ell)$, then $\delta(S)$ is bounded from above by some $b(\ell)<1$.
        \item If $g\geq 2$ then for all $\nu,\ell >0$ there exists a surface $S\in \mathcal T(\Sigma,\ell)$ with $\delta(S)>1-\nu$.
        \item If $S\in\mathcal T_\epsilon(\Sigma,\ell)$, then $\delta(S)$ is bounded from above by some $b(\epsilon,\ell)<1$
    \end{enumerate}  
\end{lemma}

\begin{proof}
    \emph{Statement 1.} 
    Note that the critical exponent is invariant under the action of the mapping class group. Let $S_i\subset \mathcal T(\Sigma,\ell)$ be a sequence so that 
    $$\delta(S_i)\xrightarrow[i\to+\infty]{}\sup\{\delta(Z)\mid Z\in {\mathcal T}(\Sigma,\ell)\}$$
    Up to passing to a subsequence, we may assume that the projections of the marked surfaces $S_i$ to the moduli space ${\rm Mod}(S)\backslash \mathcal T(\Sigma,\ell)$ converge in the Deligne--Mumford compactification of the moduli space to a surface $Z$ with connected geodesic boundary of length $\ell$, of genus $g' \leq 1$, possibly with one node. Either $Z$ is smooth and $\delta(Z)<1$ (see point~\ref{surfacewithboundary-less-than-one} of Proposition~\ref{surfacewithboundary}). Or the surface obtained by removing the node is a sphere with 3 punctures. In this case the entropies of the surfaces $S_i$ converge to the \emph{metric} entropy $\delta(Z)$ of the geodesic flow on the surface $Z$, equipped with the normalized Liouville measure, which is also less than 1.

    \emph{Statement 2.} It follows from the statement~\ref{surfacewithboundary-to-one} of Proposition~\ref{surfacewithboundary}. Find a pair of pants decomposition of $S$, take one pair of pants disjoint from $\partial S$ and shrink all its boundary components. The critical exponent of the resulting metrics goes to one.

    \emph{Statement 3.} This part of the lemma follows from invariance under the mapping class group and compactness. Namely, let us assume that $S_i\subset \mathcal T_\epsilon(\Sigma,\ell)$ is a sequence of marked metrics so that the entropy 
    \[\delta(S_i)\to \sup\{\delta(S)\mid S_i\in {\mathcal T}_\epsilon(\Sigma,\ell)\}.\] By adjusting with elements of the mapping class group, we may assume that $S_i\to S$ in ${\mathcal T}_\epsilon(\Sigma,\ell)$. Then $\delta(S_i)\to \delta(S)$, on the other hand we have $\delta(S)<1$. This completes the proof of the lemma.
\end{proof}

\bigskip\bigskip

\bigskip\bigskip


	\printbibliography[heading=bibintoc]{}


\noindent
Pierre-Louis Blayac \\
Université de Strasbourg, IRMA, 
7 rue René-Descartes
67084 Strasbourg Cedex, France\\
e-mail: blayac@unistra.fr
\medskip

\noindent
Ursula Hamenst\"adt\\
Math. Institut der Univ. Bonn, Endenicher Allee 60, 53115 Bonn, Germany\\
e-mail: ursula@math.uni-bonn.de

\medskip

\noindent
Théo Marty \\
Institut de mathématiques de Bourgogne, 
9 avenue Alain Savary,
21078 Dijon, France\\
e-mail: theo.marty@u-bourgogne.fr 

\medskip 

\noindent 
Andrea Egidio Monti\\
Max Planck Institut für Mathematik, Vivatsgasse 7, 53111 Bonn, Germany\\
e-mail: amonti@math.uni-bonn.de

\end{document}